\documentclass[12pt, leqno]{amsart}

\usepackage{graphicx}
\usepackage{amssymb,amsmath,amsthm,amscd}
\usepackage{mathrsfs,mathdots}
\usepackage{enumerate,stmaryrd}
\usepackage{upgreek}
\usepackage[usenames,dvipsnames]{color}
\usepackage[colorlinks=true, pdfstartview=FitV,
 linkcolor=blue,citecolor=blue,urlcolor=blue]{hyperref}
\usepackage[all]{xy}
\usepackage{marginnote}

\usepackage{verbatim}
\usepackage{oldgerm}
\usepackage[normalem]{ulem}  
\usepackage{xspace}
\usepackage{xcolor,cancel}

\usepackage{enumitem}

\numberwithin{equation}{section}

\allowdisplaybreaks

\newcommand{\nc}{\newcommand}

\newenvironment{red}
{\relax\color{red}}
{\hspace*{.5ex}\relax}

\newenvironment{blue}
{\relax\color{Dandelion}}
{\hspace*{.5ex}\relax}

\newenvironment{pink}
{\relax\color{WildStrawberry}}
{\hspace*{.5ex}\relax}

\newcommand{\ber}{\begin{red}}
\newcommand{\er}{\end{red}}

\newcommand{\berm}[1]{\begin{red}{}\marginnote{\fbox{\scshape\lowercase{M}}}
#1}  

\newcommand{\bermo}[1]{\begin{pink}{}\marginnote{\fbox{\scshape\lowercase{M}}}
#1}  

\newcommand{\berE}[1]{\begin{red}{}\marginnote{\fbox{\scshape\lowercase{E}}}
#1}  

\newcommand{\berMH}[1]{\begin{red}{}\marginnote{\fbox{\scshape\lowercase{MH}}}
#1}  

\newcommand{\berS}[1]{\begin{red}{}\marginnote{\fbox{\scshape\lowercase{S}}}
#1}  

\nc{\ero}{\end{pink}}

\newcommand{\beb}{\begin{blue}}
\newcommand{\eb}{\end{blue}}

\renewcommand{\le}{\leqslant}

\renewcommand{\ge}{\geqslant}

\nc{\op}{\operatorname}

\theoremstyle{plain}
\newtheorem{lemma}{Lemma}[section]
\newtheorem{proposition}[lemma]{Proposition}
\newtheorem{theorem}[lemma]{Theorem}

\theoremstyle{definition}
\newtheorem{remark}[lemma]{Remark}

\newtheorem{definition}[lemma]{Definition}
\newtheorem{corollary}[lemma]{Corollary}
\newtheorem*{convention}{Convention}

\nc{\Prop}{\begin{proposition}}
\nc{\enprop}{\end{proposition}}
\nc{\Lemma}{\begin{lemma}}
\nc{\enlemma}{\end{lemma}}
\nc{\Cor}{\begin{corollary}}
\nc{\encor}{\end{corollary}}
\nc{\Rem}{\begin{remark}}
\nc{\enrem}{\end{remark}}
\nc{\Def}{\begin{definition}}
\nc{\edf}{\end{definition}}
\nc{\shc}{\mathcal{C}}

\newcommand{\Q}{\mathbb{Q}}
\newcommand{\C}{\mathbb{C}}
\newcommand{\R}{\mathbb{R}}

\newcommand{\Seq}{\Sigma}

\newcommand{\dT}{\mathrm{T}}

\newcommand{\Ca}{\mathscr{C}}
\nc{\F}{\mathcal{F}}
\newcommand{\D}{\mathscr{D}}
\nc{\HOM}{\on{H\textsc{om}}}
\newcommand{\M}{\mathrm{M}}
\newcommand{\W}{\mathsf{W}}

\newcommand{\Z}{\mathbb{\ms{1mu}Z\ms{1mu}}}
\newcommand{\seteq}{\mathbin{:=}}

\nc{\ov}[1]{\overline{#1}}
\nc{\Wlmj}[3]{\W_{#2,#3}^{(#1)}}
\nc{\Mkl}[2]{\M_\ttww(#1,#2)}
\nc{\rmat}[1]{{\mathbf r}_{\mspace{-2mu}\raisebox{-.5ex}{${\scriptstyle{#1}}$}}}
\newcommand{\on}{\operatorname}
\nc{\de}{\on{\textfrak{d}}}
\nc{\tl}{\widetilde{\lambda}}
\nc{\mqs}{(-q^2)}
\nc{\Cquiver}{\upsigma}

\nc{\mut}[1]{{\mu}_{\mspace{-2mu}\raisebox{-.5ex}{${\scriptstyle{#1}}$}}}

\newcommand{\g}{{\mathfrak{g}}}

\newcommand{\n}{{\mathfrak{n}}}
\newcommand{\isoto}[1][]{\mathop{\xrightarrow%
[{\raisebox{.3ex}[0ex][.3ex]{$\scriptstyle{#1}$}}]%
{{\raisebox{-.6ex}[0ex][-.6ex]{$\mspace{2mu}\sim\mspace{2mu}$}}}}}

\newcommand{\id}{\on{id}}

\newcommand{\soplus}{\mathop{\mbox{\normalsize$\bigoplus$}}\limits}

\newcommand{\ww}{ \textbf{\textit{w}}}
\newcommand{\ttww}{{\widetilde{\ww}} }

\nc{\nconv}{\mathop{\mbox{\large $\odot$}}}
\nc{\nnconv}{\mathop{\mbox{\large $\star$}}}

\nc{\lb}{\llbracket}
\nc{\rb}{\rrbracket}

\newcommand{\ko}{{{\mathbf{k}}}}
\nc{\la}{\lambda}
\nc{\La}{\Lambda}

\nc{\ve}{\varepsilon}
\nc{\ep}{\epsilon}
\nc{\vp}{\varphi}
\nc{\lan}{\langle}
\nc{\ran}{\rangle}
\nc{\Uqg}{U_q(\g)}
\nc{\Aqg}{A_q(\g)}
\nc{\Aqn}{A_q(\n)}
\nc{\al}{\alpha}
\nc{\be}{\beta}
\nc{\ga}{\gamma}
\nc{\wt}{\operatorname{wt}}
\nc{\ch}{\operatorname{ch}}

\nc{\norm}{{\mathrm{norm}}}
\nc{\aff}{{\mathrm{aff}}}
\nc{\Maf}{M_\aff}
\nc{\ev}{{\mathrm{even}}}
\nc{\od}{{\mathrm{odd}}}
\nc{\Sev}{\Seq^{\ev}}
\nc{\Sod}{\Seq^{\od}}
\nc{\Spl}{\Seq^{+}}
\nc{\Smi}{\Seq^{-}}
\nc{\low}{{\mathrm{low}}}
\nc{\upper}{{\mathrm{up}}}
\nc{\one}{{\bf{1}}}
\nc{\To}[1][{\hspace{2ex}}]{\xrightarrow{\,#1\,}}
\nc{\te}{\tilde{e}}
\nc{\tw}{{\widetilde{w}}}
\nc{\tww}{\ww}
\nc{\tuu}{{\mathsf{u}}}
\nc{\tel}{\tilde{e}^\low}
\nc{\teu}{\tilde{e}^\upper}
\nc{\tf}{\tilde{f}}
\nc{\tfl}{\tilde{f}^\low}
\nc{\tfu}{\tilde{f}^\upper}
\nc{\tE}{\widetilde{E}}
\nc{\tF}{\widetilde{F}}
\nc{\tFF}{\widetilde{\F}}
\nc{\tB}{\widetilde{B}}
\nc{\tz}{\tilde{z}}
\nc{\tQ}{\hspace{-.2ex}\textbf{\textit{Q}}}
\nc{\tb}{\tilde{b}}
\nc{\Ft}{\F^\dT}
\nc{\Seed}{\mathcal{S}}

\nc{\cor}{\mathbf{k}}
\nc{\tens}{\mathop\otimes}
\nc{\gmod}{\mbox{-$\mathrm{gmod}$}}
\nc{\md}{\mbox{-$\mathrm{mod}$}}
\nc{\uqm}{\mathscr{C}_\g}
\nc{\gMod}{\mbox{-$\mathrm{gMod}$}}

\nc{\proj}{\mbox{-$\mathrm{proj}$}}
\nc{\gproj}{\mbox{-$\mathrm{gproj}$}}
\nc{\smod}{\mbox{-$\mathrm{mod}$}}
\nc{\nmod}{\mbox{-$\mathrm{nilmod}$}}

\nc{\seed}{\mathscr{S}}

\newcommand{\cmA}{\mathsf{A}}

\nc{\Rnorm}{R^{\mathrm{norm}}}
\nc{\Runiv}{R^{\ms{1mu}\mathrm{univ}}}
\nc{\Rren}{R^{\ms{1mu}\mathrm{ren}}} 
\nc{\col}{\colon}
\nc{\epiTo}[1][]{\xymatrix{\ar@{->>}[r]^-{{#1}}&}}
\nc{\epito}{\twoheadrightarrow}
\nc{\monoTo}[1][]{\xymatrix{\ar@{>->}[r]^-{{#1}}&}}
\nc{\monogets}[1][]{\xymatrix{&\ar@{_{(}->}[l]^-{{#1}}}}
\nc{\sym}{\mathfrak{S}}
\nc{\rl}{\mathsf{Q}}
\nc{\prl}{\rl_+}
\nc{\crl}{\mathsf{Q}^\vee}
\nc{\pcrl}{\crl_+}
\nc{\Qq}{{\Q(q)}}
\nc{\wl}{\mathsf{P}}   
\nc{\Oint}{\mathcal{O}_{{\mathrm{int}}}}
\newcommand{\scbul}{{\,\raise1pt\hbox{$\scriptscriptstyle\bullet$}\,}}

\nc{\conv}{\mathop{\mathbin{\mbox{\large $\circ$}}}}

\nc{\pv}{  \to\updownarrow\gets }
\nc{\nv}{  \longleftrightarrow {\raise -1pt\hbox{$\hspace{-2ex}\begin{matrix}\downarrow \\[-1ex] \uparrow\end{matrix}$}} }

\newcommand{\Hom}{\operatorname{Hom}}
\newcommand{\Ext}{\operatorname{Ext}}
\newcommand{\End}{\operatorname{End}}

\nc{\K}{{K}}
\nc{\Kex}{{\K}^{\mathrm{ex}}}
\nc{\Uex}{\Uppsi_{\mathrm{ex}}}
\nc{\Kfr}{{\K}^{\mathrm{f\mspace{.01mu}r}}}
\nc{\cl}{{\mathrm{cl}\ms{1mu}}}
\nc{\ben}{\begin{enumerate}}
\nc{\ee}{\end{enumerate}}
\nc{\bnum}{\begin{enumerate}[label=\rm(\roman*)]}
\nc{\bna}{\begin{enumerate}[label=\rm(\alph*)]}
\nc{\bc}{\begin{cases}}
\nc{\ec}{\end{cases}}

\newenvironment{myequation}
{\relax\setlength{\arraycolsep}{1pt}\begin{eqnarray}}
{\end{eqnarray}}
\newenvironment{myequationn}
{\relax\setlength{\arraycolsep}{1pt}\begin{eqnarray*}}
{\end{eqnarray*}}

\nc{\eq}{\begin{myequation}}
\nc{\eneq}{\end{myequation}}
\nc{\eqn}{\begin{myequationn}}
\nc{\eneqn}{\end{myequationn}}
\nc{\hs}{\hspace*}

\newenvironment{myarray}[1]{\relax\setlength{\arraycolsep}{.5pt}
\renewcommand{\arraystretch}{1.3}
\begin{array}{#1}}{\end{array}\relax}

\newcommand{\ba}{\begin{myarray}}
\newcommand{\ea}{\end{myarray}}

\nc{\noi}{\noindent}
\nc{\ang}[1]{\langle{#1}\rangle}
\nc{\fr}{{\mathrm{fr}}}
\nc{\qt}[1]{\quad\text{#1}\quad}
\nc{\ol}{\overline}
\nc{\true}{\delta}
\nc{\ms}{\mspace}
\renewcommand{\mod}{\ms{3mu}\mathbin{\mathrm{mod}}\ms{1mu}}
\nc{\vs}{\vspace*}
\nc{\bl}{\bigl(}
\nc{\br}{\bigr)}
\nc{\bep}{\ol{\ep}}
\nc{\bal}{\,\ol{\al}}
\nc{\qtq}[1][{and}]{\quad\text{#1}\quad}
\nc{\set}[2]{\left\{{#1}\mid{#2}\right\}}
\nc{\ro}{{\rm(}}
\nc{\rf}{{\rm)}\xspace}
\nc{\Proof}{\begin{proof}}
\nc{\QED}{\end{proof}}
\nc{\monoto}[1][]{\xymatrix@C=2ex{\ar@{>->}[r]^-{{#1}}&}\ms{-8mu}}
\nc{\etens}{\boxtimes}
\nc{\height}[1]{\vert{#1}\vert}
\nc{\dsum}{\mathop\sum\limits}
\nc{\Lrev}{L^{\bek{\rev}\ek}}
\nc{\rev}{\mathrm{rev}}
\nc{\fw}{\Lambda}
\nc{\uqpg}{U_q'(\g)}
\nc{\tp}{\ms{1.5mu}{\widetilde{p}}\ms{2mu}}
\nc{\Deg}{\mathrm{Deg}}
\nc{\Bg}{\mathcal{G}}

\nc{\wb}[1]{\mbox{$\rule[-1.1ex]{0ex}{2ex}#1$}}

\nc{\bwr}{\mbox{\large$\wr$}}
\nc{\vphi}{\varphi}
\nc{\G}{\mathcal{G}}
\nc{\Li}{\La^\infty}
\nc{\Di}{\Deg^\infty}
\nc{\zero}{\ms{1mu}\mathrm{zero}\ms{1mu}}
\nc{\cwl}{\wl^\vee}
\nc{\rc}{renormalizing coefficient\xspace}
\nc{\cz}{{\cor[z^{\pm1}]}}
\nc{\ake}[1][2ex]{\rule[-.5ex]{0ex}{#1}}
\nc{\akew}[1][2ex]{\rule[-1ex]{#1}{0ex}}
\nc{\rd}{{}^*\ms{-3mu}}
\nc{\st}[1]{\{{#1}\}}
\nc{\corh}{\widehat{\cor}}
\nc{\czt}{\cz^\times}
\nc{\eps}{\varepsilon}
\nc{\rr}{rationally renormalizable\xspace}
\nc{\QHA}{\mathrm{QHA}}

\nc{\sig}{{\sigma(\g)}}
\nc{\sigZ}{{\sigma_0(\g)}}
\nc{\sigQ}{{\sigma_Q(\g)}}
\nc{\rs}{ \mathsf{s} }
\nc{\rE}{ \mathsf{E} }
\nc{\rW}{ \mathcal{W} }
\nc{\rES}{ \mathcal{E} }
\nc{\scor}{ {\cor_0} }
\nc{\cat}{\mathcal{C}}
\nc{\catB}{\mathcal{B}}
\nc{\trivial}{\mathbf{1}}
\nc{\lP}{\mathcal{P}}
\nc{\lQ}{\mathcal{Q}}
\nc{\lse}{\mathsf{e}}
\nc{\bse}{\bar{\lse}}

\nc{\wlP}{\mathsf{P}}
\nc{\rlQ}{\mathsf{Q}}

\nc{\cmC}{\mathsf{C}}
\nc{\tcmC}{\widetilde{\mathsf{C}}}
\nc{\tcmc}{\widetilde{c}}
\nc{\CaQ}{{\Ca_\g^{Q}}}

\nc{\MQ}{{ \mathsf{M}_Q }}
\nc{\mQ}[1]{{ \mathsf{m}^Q_{#1} }}
\nc{\rank}{{ \mathrm{rank} }}
\nc{\vpi}{\varpi}

\newcommand{\drt}[2]{ { \left( \begin{matrix}  \ \ \ \   #1 \\ #2 \end{matrix}\right) } }
\newcommand{\ddrt}[2]{ { \left( \begin{matrix}   \ \ \   #1 \\ #2 \end{matrix}\right) } }
\newcommand{\prt}[2]{ \left( \begin{matrix} #1 \\ #2 \end{matrix}\right) }

\newcommand{\pprt}[4]{ {  \fontsize{10}{10}\selectfont \left( \begin{matrix} #1 \\ #2 \\ #3 \\ #4 \end{matrix}\right) \fontsize{12}{12}\selectfont} }

\nc{\sn}[1][{1.2ex}]{\vs{#1}\noi}
\nc{\dual}{\D}
\nc{\Up}{U'_q(\g)}
\nc{\corz}{\cor[z^{\pm1}]}
\nc{\kz}{\cor(z)}
\nc{\afwt}{affine highest weight}
\nc{\gf}{{\g_{\mathrm{fin}}}}
\nc{\Df}{{\Delta^+_{Q}}} 
\nc{\Pf}{{\Pi_{Q}}}
\nc{\If}{{\mathrm{I_{fin}}}}
\nc{\scb}[2][.8]{\scalebox{#1}{$#2$}}
\nc{\df}{{\Delta_{\mathrm{fin}}}}
\nc{\rlq}{\mathsf{Q}_Q}

\nc{\Dynkin}{\mathsf{D}}
\nc{\Dynkinf}{\Dynkin_{\rm fin}}
\nc{\orb}{\mathrm{orb}}
\nc{\ord}{\mathrm{ord}}
\newcommand{\fd}{\mathsf{d}}
\newcommand{\hI}{\widehat{I}}
\nc{\esig}{\zeta}
\newcommand{\qm}{{q}_{\mspace{1mu}\raisebox{-.5ex}{${\scriptstyle{\mathrm{sh}}}$}}}
\nc{\ble}[1]{\underline{#1}}
\nc{\dd}{\mathrm{d}_\circ}
\nc{\dist}{\mathrm{d}}
\nc{\Qd}{\mathscr{Q}}

\title[Simply-laced root systems arising from quantum affine algebras]%
{ Simply-laced root systems arising from quantum affine algebras }

\author[M. Kashiwara]{Masaki Kashiwara}
\thanks{The research of M.\ Kashiwara
was supported by Grant-in-Aid for Scientific Research (B)
 20H01795,  Japan Society for the Promotion of Science.}
\address[M. Kashiwara]{
Kyoto University Institute for Advanced Study,
Research Institute for Mathematical Sciences, Kyoto University,
Kyoto 606-8502, Japan \& Korea Institute for Advanced Study, Seoul 02455, Korea }
\email[M. Kashiwara]{masaki@kurims.kyoto-u.ac.jp}

\author[M. Kim]{Myungho Kim}
\address[M. Kim]{Department of Mathematics, Kyung Hee University, Seoul 02447, Korea}
\email[M. Kim]{mkim@khu.ac.kr}
\thanks{The research of M.\ Kim was supported by the National Research Foundation of
Korea(NRF) Grant funded by the Korea government(MSIP) (NRF-2017R1C1B2007824 and NRF-2020R1A5A1016126).}

\author[S.-j. Oh]{Se-jin Oh}
\thanks{ The research of S.-j.\ Oh was supported by the Ministry of Education of the Republic of Korea and the National Research Foundation of Korea (NRF-2019R1A2C4069647).}
\address[S.-j. Oh]{Department of Mathematics, Ewha Womans University, Seoul 03760, Korea}
\email[S.-j. Oh]{sejin092@gmail.com}

\author[E. Park]{Euiyong Park}
\thanks{The research of E.\ P.\ was supported by the National Research Foundation of Korea(NRF) Grant funded by the Korea Government(MSIP)(NRF-2020R1F1A1A01065992 and NRF-2020R1A5A1016126)}
\address[E. Park]{Department of Mathematics, University of Seoul, Seoul 02504, Korea}
\email[E. Park]{epark@uos.ac.kr}

\date{September 23, 2021.}

\begin{document}

\begin{abstract}

Let $U_q'(\g)$ be a quantum affine algebra with an indeterminate $q$ and let $\Ca_\g$ be the category of finite-dimensional integrable $U_q'(\g)$-modules.
We write $\Ca_\g^0$ for the monoidal subcategory of $\Ca_\g$ introduced by Hernandez-Leclerc. 
In this paper, we associate a simply-laced finite type root system to
each quantum affine algebra $U_q'(\g)$ in a natural way, and 
show that the  block decompositions of $\Ca_\g$ and $\Ca_\g^0$ are parameterized by the lattices associated with the root system.
We first define a certain abelian group $\rW$ (resp.\ $\rW_0$) arising from simple modules of $ \Ca_\g$ (resp.\ $\Ca_\g^0$) by using the invariant $\Li$ introduced in the previous work by the authors.
The groups  $\rW$ and $\rW_0$ have the subsets $\Delta$ and $\Delta_0$ determined by the fundamental representations in $ \Ca_\g$ and $\Ca_\g^0$ respectively. 
We prove that the pair $( \R \otimes_\Z \rW_0, \Delta_0)$  
is an irreducible simply-laced root system 
of finite type and the pair $( \R \otimes_\Z \rW, \Delta) $ is isomorphic to the direct sum of infinite copies of $( \R \otimes_\Z \rW_0, \Delta_0)$ as a root system.

\end{abstract}

\maketitle
\tableofcontents

\section*{Introduction}

Let $q$ be an indeterminate and let $U_q'(\g)$ be a quantum affine algebra. The category $\Ca_\g$ of finite-dimensional integrable $U_q'(\g)$-modules has a rich structure.
For example, the category $\Ca_\g$ is not semi-simple and has a rigid monoidal category structure. 
Because of its rich structure,  it has been studied actively in various research areas of mathematics and physics (see \cite{AK, CP94, FR99, Kas02,  Nak01} for examples).

The category $\Ca_\g$ has been studied in the viewpoint of cluster algebras. 
Suppose that  $\g$ is of simply-laced affine ADE types.
In \cite{HL10}, 
Hernandez and Leclerc defined the full subcategory $\Ca_\g^0$ of $\Ca_\g$
 such that all simple subquotients of its objects are obtained by simple subquotients of tensor products of certain fundamental representations.  
They then introduced certain monoidal subcategories $\Ca_\ell$ ($\ell \in \Z_{>0 }$) and studied their Grothendieck rings using cluster algebras. 
As any simple module in $\Ca_\g$ can be obtained from a tensor product of suitable parameter shifts of simple modules in $\Ca_\g^0$, 
the category $\Ca_\g^0$ takes an essential position in $\Ca_\g$. 
Note that an algorithm for computing $q$-character of Kirillov-Reshetikhin modules for any untwisted quantum affine algebras was described in \cite{HL16}
by studying the cluster algebra structure of the Grothendieck ring of the subcategory $\Ca_\g^-$ of $\Ca_\g^0$.  
On the other hand, they introduced another abelian monoidal subcategory $\CaQ$
 which categorifies the coordinate ring $\C[N]$ of the unipotent group associated with the finite-dimensional simple Lie algebra $\g_0$ inside $\g$ (\cite{HL15}).
For each Dynkin quiver $Q$, they defined an abelian subcategory $\CaQ$ of $\Ca_\g^0$ which contains some fundamental representations parameterized by the coordinates of vertices of the Auslander-Reiten quiver of $Q$, and 
proved that $\CaQ$ is stable under taking tensor product and its complexified Grothendieck ring $\C \otimes _\Z K(\CaQ)$ is isomorphic to the coordinate ring $\C[N]$.
Moreover, under this isomorphism, the set of isomorphism classes of simple modules in $\CaQ$ corresponds to the upper global basis of $\C[N]$.

The notion of the categories $\Ca_\g^0$ and $\CaQ$ is extended to all untwisted and twisted quantum affine algebras (\cite{KKKO15D, KO18, OS19, OhSuh19}). 
Let 
$
 \sig \seteq I_0 \times \cor^\times / \sim, 
$
where the equivalence relation is given by 
$ (i,x) \sim (j,y) $ if and only if $V(\varpi_i)_x \simeq V(\varpi_j)_y$.
The set $\sig$ has a quiver structure determined by the pole of R-matrices between tensor products of fundamental representations $V(\varpi_i)_x$ ($(i,x)\in \sig$).
 Let $\sigZ$ be a connected component of $\sig$. 
Then the category $\Ca_\g^0$ is defined to be the smallest full subcategory of $\Ca_\g$ such that 
\bna
\item $\Ca_\g^0$ contains $V(\varpi_i)_x$ for all $(i,x) \in \sigZ$,
\item $\Ca_\g^0$ is stable by taking subquotients, extensions and tensor products.
\ee
The subcategory $\CaQ$ was introduced in \cite{KKKO15D} for twisted affine type $A^{(2)}$ and $D^{(2)}$, 
in \cite{KO18} for untwisted affine type $B^{(1)}$ and $C^{(1)}$, and in \cite{OS19, OhSuh19} for exceptional affine type.
 For a Dynkin quiver $Q$ of a certain type with additional data,
a finite subset $\sigQ$ of $\sigZ$ was determined. 
Then the category $\CaQ$ is defined to be the smallest full subcategory of $\Ca_\g^0$ such that 
\bna
\item $\CaQ$ contains $\trivial$ and $V(\varpi_i)_x$ for all $(i,x) \in \sigQ$,
\item $\CaQ$ is stable by taking subquotients, extensions and tensor products.
\ee
(see Section \ref{Sec: HL cat} and \ref{Sec: CaQ} for more details)

\smallskip
We can summarize our results of this paper as follows:
\bnum
\item we associate a simply-laced root system to
each quantum affine algebra $U_q'(\g)$ in a natural way, and
\item
we give the  block decomposition of $\Ca_\g$ parameterized by a lattice $\rW$ associated with the root system.
\ee
Let $U_q'(\g)$ be a quantum affine algebra of \emph{arbitrary} type.
We first consider certain subgroups $\rW$ and $\rW_0$ of the abelian group $\Hom(\sig, \Z)$ arising from simple modules of $ \Ca_\g$ and $\Ca_\g^0$, respectively (see $\eqref{Eq: W and Delta}$).
The subgroups  $\rW$ and $\rW_0$ have the subset $\Delta$ and $\Delta_0$ determined by the fundamental representations in $ \Ca_\g$ and $\Ca_\g^0$ respectively. 
Let $ \rES \seteq \R \otimes_\Z \rW $ and $\rES_0 \seteq \R \otimes_\Z \rW_0$.
Let $\gf$ be
the simply-laced finite-type Lie algebra  corresponding to the affine type of $\g$ 
in Table $\eqref{Table: root system}$. 
When $\g$ is of untwisted affine type $ADE$, 
$\gf$ coincides with $\g_0$. 
We prove that the pair $( \rES_0, \Delta_0)$  is the irreducible root system of the Lie algebra $\gf$ and the pair $( \rES, \Delta) $ is isomorphic to the direct sum of infinite copies of $(\rES_0, \Delta_0)$ as a root system (see Theorem \ref{Thm: root system} and Corollary \ref{Cor: root system for W}).
Interestingly enough, the quantum affine algebra $U_q'(\g)$ and its Langlands dual $U_q'(^{L}\g)$, whose Cartan matrix is the transpose of that of $\g$,  yield the same simply-laced root system. 
This coincidence is also viewed as the mysterious duality between $U_q'(\g)$ and its Langlands dual $U_q'(^{L}\g)$ (see \cite{FH11A, FH11B, FR98}).
We conjecture  that the categories of representations of two quantum affine algebras are equivalent if and only if their associated root systems are the same.
In this viewpoint, the simply-laced finite type root system plays a role of an invariant for the representation categories of quantum affine algebras. 
For each simply-laced finite type root system, the corresponding untwisted quantum affine algebras, the one of the twisted type (if it exists) and its Langland dual have the same categorical structure.

\smallskip
We then show that there exist direct decompositions of $\Ca_\g$ and $\Ca_\g^0$ parameterized by elements of
$\rW$ and $\rW_0$, respectively (Theorem \ref{Thm: decomposition for Ca}),
and prove that each direct summand of the decompositions is a block (Theorem \ref{Thm: block for Ca_al}).
This approach covers all \emph{untwisted and twisted} quantum affine algebras in a uniform way, and explains transparently how the blocks of $\Ca_\g^0$ exist in the viewpoint of the root system $( \rES_0, \Delta_0)$ and the category $\CaQ$. 

 When $\g$ is of untwisted type, the block decomposition was studied in \cite{CM05, EM03, JM11}.
 Etingof and Moura (\cite{EM03})  found the  block decomposition of $\Ca_\g$ whose blocks are parameterized by the \emph{elliptic central characters} with the condition $|q|<1$.  
Later, Chari and Moura (\cite{CM05}) gave a different description of the block decomposition of $\Ca_\g$ by using the quotient group $ \mathscr{P}_q /  \mathscr{Q}_q$ of 
the \emph{$\ell$-weight lattice} $\mathscr{P}_q$ by the \emph{$\ell$-root lattice} $\mathscr{Q}_q$.
In case of the quantum affine algebra  $U_\xi({\g})$ at roots of unity, its block decomposition was studied in \cite{JM11}.
For affine Kac-Moody algebras,  the  block decomposition of the category of finite-dimensional modules was studied in \cite{CM04, Senesi10}. Note that the block decomposition for affine Kac-Moody algebras does not explain blocks for quantum affine algebras $U'_q(\g)$.
We remark that, in case of untwisted type,  the quotient group $ \mathscr{P}_q /  \mathscr{Q}_q$ given in \cite{CM05} (and also the result of \cite{EM03}) gives another group presentation of $\rW$ (see Remark \ref{Rmk: CM EM}).

The main tools to prove our results are the new invariants $\La $, $\Li $ and $\de  $ for a pair of modules in $\Ca_\g$ introduced in \cite{KKOP19C}. 
For  non-zero modules $M$ and $N$ in $\Ca_\g$ such that $\Runiv_{M,N_z}$ is \emph{\rr}, the integers $\Lambda(M,N)$, $\Lambda^\infty(M,N)$ and $\de(M,N)$ are defined by using the \emph{renormalizing coefficient} $c_{M,N}(z)$ (see Section \ref{Sec: new invariants} for details).
These invariants are quantum affine algebra analogues of the invariants (with the same notations) 
for pairs of graded modules over quiver Hecke algebras 
arising from the grading of R-matrices. 
The new invariants play similar roles in the representation theory of quantum affine algebras as the ones for quiver Hecke algebras do.

\medskip

Let us explain our results more precisely.  
Let $U_q'(\g)$ be a quantum affine algebra of \emph{arbitrary} type.
For $M \in \Ca_\g$ such that the universal R-matrix $ \Runiv_{M, V(\varpi_i)_z} $ is rationally renormalizable for any $i\in I_0$, we define $ \rE(M) \in \Hom(\sig, \Z)$ 
by 
$$
\rE(M) (i,a) := \Li(M, V(\varpi_i)_a) \qquad \text{ for } (i,a) \in \sig
$$
and investigate its properties (Lemma \ref{Lem: prop for E}). For $(i,a) \in \sig$, we set 
\begin{align*}
\rs_{i,a} \seteq \rE(V(\varpi_i)_a) \in \Hom(\sig, \Z),
\end{align*}
and 
\begin{equation*} 
\begin{aligned}
\rW \seteq \{  \rE(M) \mid M \text{ is simple in } \Ca_\g\}, & \qquad \Delta \seteq \{ \rs_{i,a} \mid (i,a) \in \sig  \} \subset \rW, \\
\rW_0 \seteq \{  \rE(M) \mid M \text{ is simple in } \Ca_\g^0 \}, & \qquad \Delta_0 \seteq \{ \rs_{i,a} \mid (i,a) \in \sigZ  \} \subset \rW_0.
\end{aligned}
\end{equation*}
Then $\rW$ and $\rW_0$ are abelian subgroups of $\Hom(\sig, \Z)$. 
Moreover, we see in Lemma \ref{Lem: (,)on W} that  
there exists a unique symmetric bilinear form $(-,-)$ on $\rW$ such that $ ( \rE(M), \rE(N) ) = - \Li(M,N) $ for any simple modules $M,N \in \Ca_\g$.
It induces a symmetric bilinear form on $ \rES$. 
Then  we prove that  the pair $(\rES_0, \Delta_0)$ is an irreducible 
 root system of the simply-laced finite type given in $\eqref{Table: root system}$ (Theorem \ref{Thm: root system}) 
and the pair $( \rES, \Delta) $ is isomorphic to the direct sum of infinite copies of $(\rES_0, \Delta_0)$ as a root system (Corollary \ref{Cor: root system for W}). Moreover the bilinear form $(-,-)$ is invariant 
under the Weyl group action. 
Theorem  \ref{Thm: root system} is proved in Section \ref{Sec: proofs} by case-by-case approach using the explicit descriptions of $\sigQ$ for $\CaQ$ given in Section \ref{Sec: CaQ} and the denominator formulas written in Appendix \ref{App: denominators}.

We then consider the block decompositions of
 $\Ca_\g$ and $\Ca_g^0$. 
For $\al \in \rW$, let $\Ca_{\g, \al}$ be the full subcategory of $\Ca_\g$ consisting of objects $X$ such that
$\rE(S)=\al$ for any simple subquotient $S$ of $X$. We show that there exist the following direct decompositions 
\begin{align*}
\Ca_\g = \bigoplus_{\al \in \rW} \Ca_{\g, \al} \quad \text{and} \quad \Ca_\g^0 = \bigoplus_{\al \in \rW_0} \Ca_{\g, \al},
\end{align*}
by proving $\Ext^1_{U_q'(\g)}(M,N)=0$ for $M \in \Ca_{\g, \alpha} $ and $ N \in \Ca_{\g, \be}$ with $\alpha \ne \be$ (Theorem \ref{Thm: decomposition for Ca}).
We set 
$\lP \seteq \bigoplus_{(i,a) \in \sig} \Z \lse_{(i,a)}$   
and $\lP_0 \seteq \bigoplus_{(i,a) \in \sigZ} \Z \lse_{(i,a)}$,
where $\lse_{(i,a)}$ is a symbol. Then we define a group homomorphism $ p\col \lP \twoheadrightarrow \rW $ by 
$ p( \lse_{(i,a)} )= \rs_{i,a}$,  
and set $ p_0 \seteq p|_{\lP_0}\col \lP_0 \twoheadrightarrow \rW_0.  $
It turns out that the kernel $\ker p_0$  coincides with the subgroup $ \lQ_0$ of $\lP_0$ generated by the elements of the form $\sum_{k=1}^m \lse_{(i_k,a_k)}$ ($(i_k,a_k) \in \sigZ$) 
such that the trivial module $\trivial$ appears in 
$ V(i_1)_{a_1} \otimes \cdots \otimes V(i_m)_{a_m}$ 
as a simple subquotient (Lemma \ref{Lem: Q ker p}). 
We then prove that $\Ca_{\g, \al}$ is a block for any $\al \in \rW$ (Theorem \ref{Thm: block for Ca_al}), which implies that the above decompositions are block decompositions of $\Ca_\g$ and $\Ca_\g^0$.

This paper is organized as follows.
In Section \ref{sec: preliminaries},  we give the necessary background on quantum affine algebras, R-matrices, and the categories $\Ca_\g$ and $\CaQ$.
In Section \ref{Sec: new invariants}, we review the new invariants introduced in \cite{KKOP19C}.
In Section \ref{Sec: Root system}, we investigate properties of $\rW$, $\Delta$ and $ \rs_{i,a}$ and state the main theorem for the root systems $( \rES_0, \Delta_0) $ and $( \rES, \Delta) $.
In Section \ref{Sec: Block decomposition}, we prove the block decompositions of $\Ca_\g$ and $\Ca_\g^0$.
Section \ref{Sec: proofs} is devoted to giving a case-by-case proof of Theorem \ref{Thm: root system}.

\medskip
{\bf Acknowledgments}

The second, third and fourth authors gratefully acknowledge for the hospitality of RIMS (Kyoto University) during their visit in 2020.
The authors would like to thank the anonymous referee for valuable comments and suggestions.

\vskip 2em 

\section{Preliminaries} \label{sec: preliminaries}

\begin{convention}
\ 
\bnum
\item
For a statement $P$, $\delta(P)$ is $1$ or $0$ according that
$P$ is true or not.
\item For an element $a$ in a field $ \cor$ and $f(z)\in\cor(z)$,
we denote
by $\zero_{z=a}f(z)$ the order of zero of $f(z)$ at $z=a$.\label{conv:zero}
\ee
\end{convention}

\subsection{Quantum affine algebras}
The quintuple $( \cmA, \wlP, \Pi, \wlP^\vee, \Pi^\vee )$ is called an \emph{affine Cartan datum} if it consists of 
\bnum
\item  an \emph{affine Cartan matrix} $\cmA=(a_{ij})_{i,j\in I}$ with a finite index set $I$, 
\item a free abelian group $\wl$ of rank $|I|+1$, called the \emph{weight lattice},
\item   a set $\Pi=\{\al_i \in \wl \ | \ i \in I\}$, whose elements are called \emph{simple roots}, 
\item the group $\wl^\vee \seteq \Hom_\Z(\wl,\Z)$, called the \emph{coweight lattice},
\item a set $\Pi^\vee=\{h_i \ | \ i \in I\} \subset \wl^\vee$, whose elements are \emph{simple coroots}, 
\ee
which satisfies the following properties:
\bna
\item $\langle h_i, \alpha_j  \rangle = a_{i,j}$ for any $i,j\in I$,
\item
for any $i\in I$, there exists $\La_i\in\wl$ such that
$\ang{h_j,\La_i}=\delta(i=j)$ for any $j\in I$,
\item $\Pi$ is linearly independent. 
\ee

 Let $\g$ be the \emph{affine Kac-Moody algebra} associated with $(\cmA,\wl,\Pi,\wl^\vee,\Pi^\vee)$. 
We set $\rlQ \seteq \bigoplus_{i \in I}\Z\al_i\subset\wlP$ called the \emph{root lattice}, and $\rlQ^+ \seteq \sum_{i \in I}\Z_{\ge 0}\al_i\subset \rlQ$.
For $\beta = \sum_{i\in I} b_i \al_i \in \rlQ^+$, we write $ |\beta| = \sum_{i\in I} b_i $.
We denote by $\updelta \in \rlQ$  the \emph{imaginary root} and by $c \in \rlQ^\vee$ 
the  \emph{central element}. Note that 
the positive  imaginary root $\Delta_+^{\rm im}$ is equal to $\Z_{>0} \updelta$  and the center of  $\g$ is generated by $c$.
We write $\wl_\cl  \seteq \wl / (\wl\cap\Q \updelta) $, called the \emph{classical weight lattice}, and take $\rho \in \wl$ (resp.\ $\rho^\vee \in \wl^\vee$) such that $\lan h_i,\rho \ran=1$ (resp.\ $\lan \rho^\vee,\al_i\ran =1$) for any $i \in I$.
There exists a  $\Q$-valued non-degenerate  symmetric bilinear form $( \ , \ )$ on $\wlP$ satisfying 
$$
\lan h_i,\la \ran= \dfrac{2(\al_i,\la)}{(\al_i,\al_i)} \quad \text{ and} \quad \lan c,\la \ran = (\updelta,\la)
$$
for any $i \in I$ and $\la \in \wlP$. We write $\W \seteq \lan s_i \mid i \in I \ran \subset \textrm{Aut}(\wlP)$ for the \emph{Weyl group} of $\cmA$, where
$s_i (\la) \seteq \la - \lan h_i,\la \ran \al_i$ for 
$\la \in  \wlP $. 
We will use the standard convention in~\cite{Kac}
to choose $0\in I$ except $A^{(2)}_{2n}$ type, in which we take the longest simple root as $\al_0$, and $B_2^{(1)}$, $A_3^{(2)}$, and $E_k^{(1)}$ ($k=6,7,8$) types, in which we take the following Dynkin diagrams:
\begin{equation} \label{Eq: DD}
\begin{aligned} 
& A^{(2)}_{2n} :   \xymatrix@C=4ex@R=3ex{
  *{\circ}<3pt> \ar@{<=}[r]_<{n \ } & *{ \circ }<3pt> \ar@{-}[r]_<{n-1}  & *{ \circ }<3pt> \ar@{-}[r]_<{n-2} & \cdots \ar@{-}[r]_<{ }   &*{\circ}<3pt> \ar@{-}[l]^<{1} &*{ \circ }<3pt>  \ar@{=>}[l]^<{ \ \ 0}  } \hs{3ex}
B^{(1)}_{2} :   \xymatrix@C=4ex@R=3ex{
  *{\circ}<3pt> \ar@{=>}[r]_<{0 \ } & *{ \circ }<3pt>  \ar@{<=}[l]^<{2} &*{ \circ }<3pt>  \ar@{=>}[l]^<{ \ \ 1}  } \hs{3ex}
A^{(2)}_{3} :   \xymatrix@C=4ex@R=3ex{
  *{\circ}<3pt> \ar@{<=}[r]_<{0 \ } & *{ \circ }<3pt>  \ar@{=>}[l]^<{2} &*{ \circ }<3pt>  \ar@{<=}[l]^<{ \ \ 1}  }
\\
&
 E^{(1)}_6 :  \raisebox{2.3em}{\xymatrix@C=3.4ex@R=2ex{ && *{\circ}<3pt>\ar@{-}[d]^<{0}  \\ && *{\circ}<3pt>\ar@{-}[d]^<{2} \\
*{ \circ }<3pt> \ar@{-}[r]_<{1}  &*{\circ}<3pt>
\ar@{-}[r]_<{3} &*{ \circ }<3pt> \ar@{-}[r]_<{4} &*{\circ}<3pt>
\ar@{-}[r]_<{5} &*{\circ}<3pt>
\ar@{-}[l]^<{\ \ 6}}} \qquad 
E^{(1)}_7 : \raisebox{1.3em}{\xymatrix@C=3.4ex@R=3ex{ && & *{\circ}<3pt>\ar@{-}[d]^<{2} \\
*{ \circ }<3pt> \ar@{-}[r]_<{0}  & *{ \circ }<3pt> \ar@{-}[r]_<{1}  &*{\circ}<3pt>
\ar@{-}[r]_<{3} &*{ \circ }<3pt> \ar@{-}[r]_<{4} &*{\circ}<3pt>
\ar@{-}[r]_<{5} &*{\circ}<3pt>
\ar@{-}[r]_<{6} &*{\circ}<3pt>
\ar@{-}[l]^<{7} } } \allowdisplaybreaks \\
& E_8^{(1)}  :  \raisebox{1.3em}{\xymatrix@C=3.4ex@R=3ex{ && *{\circ}<3pt>\ar@{-}[d]^<{2} \\
*{ \circ }<3pt> \ar@{-}[r]_<{1}  &*{\circ}<3pt>
\ar@{-}[r]_<{3} &*{ \circ }<3pt> \ar@{-}[r]_<{4} &*{\circ}<3pt>
\ar@{-}[r]_<{5} &*{\circ}<3pt>
\ar@{-}[r]_<{6} &*{\circ}<3pt>
\ar@{-}[r]_<{7} &*{\circ}<3pt>
\ar@{-}[l]^<{8} &*{\circ}<3pt>
\ar@{-}[l]^<{0} } }
\end{aligned}
\end{equation}
Note that $B^{(1)}_2$ and $A^{(2)}_3$ in the above diagram 
are denoted by $C_2^{(1)}$ and $D_3^{(2)}$ respectively in \cite{Kac}.

Let  $\g_0$ be the subalgebra of $\g$ generated by the Chevalley generators $e_i$, $f_i$  and $h_i$ for $i \in I_0 \seteq I \setminus \{ 0 \}$
and let  $\W_0$ be the subgroup of $\W$ generated by $s_i$ for $i \in I_0$. 
Note that $\g_0$ is a finite-dimensional simple Lie algebra and $\W_0$ contains the longest element $w_0$.

Let $q$ be an indeterminate and $\ko$ the algebraic closure of the subfield $\C(q)$
in the algebraically closed field $\corh\seteq\bigcup_{m >0}\C((q^{1/m}))$. 
 For $m,n \in \Z_{\ge 0}$ and $i\in I$, we define
$q_i = q^{(\alpha_i,\alpha_i)/2}$ and
\begin{equation*}
 \begin{aligned}
 \ &[n]_i =\frac{ q^n_{i} - q^{-n}_{i} }{ q_{i} - q^{-1}_{i} },
 \ &[n]_i! = \prod^{n}_{k=1} [k]_i ,
 \ &\left[\begin{matrix}m \\ n\\ \end{matrix} \right]_i=  \frac{ [m]_i! }{[m-n]_i! [n]_i! }.
 \end{aligned}
\end{equation*}

Let $d$ be the smallest positive integer such that 
$d \frac{(\al_i, \al_i)}{2}\in\Z$  for all $i\in I$.

\begin{definition} \label{Def: GKM}
The {\em quantum affine algebra} $U_q(\g)$ associated with an affine Cartan datum $(\cmA,\wl,\Pi,\wl^\vee,\Pi^\vee)$ is the associative algebra over $\ko$ with $1$ generated by $e_i,f_i$ $(i \in I)$ and
$q^{h}$ $(h \in  d^{-1} \wl^{\vee})$ satisfying following relations:
\bnum
\item  $q^0=1, q^{h} q^{h'}=q^{h+h'} $\quad for $ h,h' \in d^{-1} \wl^{\vee},$
\item  $q^{h}e_i q^{-h}= q^{\langle h, \alpha_i \rangle} e_i$,
$q^{h}f_i q^{-h} = q^{-\langle h, \alpha_i \rangle }f_i$\quad for $h \in d^{-1}\wl^{\vee}, i \in I$,
\item  $e_if_j - f_je_i =  \delta_{ij} \dfrac{K_i -K^{-1}_i}{q_i- q^{-1}_i }, \ \ \text{ where } K_i=q_i^{ h_i},$
\item  $\displaystyle \sum^{1-a_{ij}}_{k=0}
(-1)^ke^{(1-a_{ij}-k)}_i e_j e^{(k)}_i =  \sum^{1-a_{ij}}_{k=0} (-1)^k
f^{(1-a_{ij}-k)}_i f_jf^{(k)}_i=0 \quad \text{ for }  i \ne j, $
\ee
where $e_i^{(k)}=e_i^k/[k]_i!$ and $f_i^{(k)}=f_i^k/[k]_i!$.
\end{definition}

Let us denote by $U_q'(\g)$ the $\cor$-subalgebra of $U_q(\g)$ generated by $e_i,f_i,K^{\pm 1}_i$ $(i \in I)$.
The coproduct $\Delta$ of $\uqpg$ is given by
\begin{align*} 
\Delta(q^h)=q^h \tens q^h, \ \ \Delta(e_i)=e_i \tens K_i^{-1}+1 \tens e_i, \ \Delta(f_i)=f_i \tens 1 +K_i \tens f_i,
\end{align*}
and the bar involution $ \ \bar{ } \ $ of $U_q'(\g)$ is defined as
\[
 q^{1/m} \to q^{-1/m}, \qquad\qquad
e_i \mapsto e_i, \qquad\qquad f_i \mapsto f_i, \qquad\qquad K_i \mapsto K_i^{-1}.
\]

Let $\uqm$ be the category of finite-dimensional integrable
$\uqpg$-modules, 
i.e., finite- \allowbreak dimensional modules $M$ with a weight decomposition
$$
M=\soplus_{\la\in\wl_\cl}M_\la  \qquad \text{ where } M_\la=\st{u\in M\mid K_iu=q_i^{\ang{h_i,\la}}u }.
$$
Note that   the trivial  module $\trivial$ is contained in $\Ca_\g$ and
 the tensor product $\otimes$ gives a monoidal category structure on $\Ca_\g$. 
It is known that the Grothendieck ring $K(\uqm)$ is a commutative ring \cite{FR99}.
A simple module $L$ in $\uqm$ contains a non-zero vector $u \in L$ of weight $\lambda\in \wl_\cl$ such that (i) $\langle h_i,\lambda \rangle \ge 0$ for all $i \in I_0$,
(ii) all the weight of $L$ are contained in $\lambda - \sum_{i \in I_0} \Z_{\ge 0} \cl(\alpha_i)$, where $\cl\colon \wl\to \wl_\cl$ is the canonical projection.
Such a $\la$ is unique and $u$ is unique up to a constant multiple. We call $\lambda$ the {\it dominant extremal weight} of $L$ and $u$ a {\it dominant extremal weight vector} of $L$.

Let $\wl_\cl^0\seteq\{\la\in\wl_\cl\mid \ang{c,\la}=0\}$. For each $i \in I_0$, we set 
$$
\varpi_i \seteq {\rm gcd}(\mathsf{c}_0,\mathsf{c}_i)^{-1}\cl(\mathsf{c}_0\Lambda_i-\mathsf{c}_i \Lambda_0) \in \wl_\cl^0,
$$
where the central element $c$ is equal to $ \sum_{ i\in I} \mathsf{c}_i h_i$.
Note that $\wl_\cl^0= \soplus_{i\in I_0}\Z\varpi_i$. For any $i\in I_0$, there
exists a unique simple module $V(\varpi_i)$ in $\uqm$ satisfying certain good conditions (see \cite[Section 5.2]{Kas02}), which is called the \emph{$i$-th fundamental representation}. 
Note that the dominant extremal weight of $V(\varpi_i)$ is $\varpi_i$.

For simple modules $M$ and $N$ in $\uqm$,
we say that $M$ and $N$ {\em commute}
or $M$ commutes with $N$ if $M\tens N\simeq N\tens M$.
We say that $M$ and $N$ \emph{strongly commute} or $M$ \emph{strongly commutes with} $N$ if $M\tens N$ is simple.
Note that $M\tens N$ is simple if and only if $N\tens M$ is simple since $K(\uqm)$ is a commutative ring.
It is clear that, if simple modules $M$ and $N$ strongly commute, then they commute.
We say that a simple module $M$ is \emph{real} if $M$ strongly commutes with itself. 

For an integrable $U_q'(\g)$-module $M$, we denote by $M_z$ the {\it affinization} of $M$ and 
by $z_M \colon M_z \to M_z$ the $U_q'(\g)$-module automorphism of weight $\updelta$. Note that $M_z \simeq \cor[z^{\pm 1}]\otimes_{\cor} M $ for an indeterminate $z$ as a $\cor$-vector space.
 For $x \in \ko^\times$, we define
$$M_x \seteq M_z / (z_M -x)M_z.$$
We call $x$ a {\it spectral parameter}.
The functor $T_x$ defined by $T_x(M)=M_x$ is an endofunctor of $\uqm$
which commutes with tensor products (see \cite[Section 4.2]{Kas02} for details).

It is known that a fundamental representation is a \emph{good module}, which is a simple $U_q'(\g)$-module with good properties including a {\em bar involution}, a crystal basis with {\em simple
	crystal graph}, and a {\em global basis} (see~\cite{Kas02} for the precise definition).
We say that  a $U_q'(\g)$-module $M$ is {\em quasi-good} if
$$ M \simeq V_c $$
for some good module $V$ and $c \in \cor^\times$.
Note that any quasi-good module is a simple $U_q'(\g)$-module. Moreover the tensor product $M^{\otimes k} \seteq \underbrace{M \tens \cdots \tens M}_{k\text{-times}}$ for a quasi-good module $M$ and $k \in \Z_{\ge 1}$
is again quasi-good.

For a $U_q'(\g)$-module $M$, we denote by $\overline{M}=\{ \bar{u} \mid u \in M \}$
the $U_q'(\g)$-module defined by $x \bar{u} \seteq \overline{\ms{2mu}\overline{x} u\ms{2mu}}$ for $x \in U_q'(\g)$.
Then we have
\begin{align*} 
\overline{M_a} \simeq (\overline{M})_{\,\overline{a}} \quad \text{and} \quad \overline{M \otimes N} \simeq \overline{N} \otimes \overline{M}\qt{  for any $M,N\in\uqm$ and $a\in\cor^\times$. }
\end{align*}
Note that $V(\varpi_i)$ is \emph{bar-invariant}; i.e., $\overline{V(\varpi_i)}\simeq V(\varpi_i)$ (see \cite[Appendix A]{AK}).

Let $m_i$ be a positive integer such that
$$
\W\pi_i\cap\bigl(\pi_i+\Z\updelta\bigr)=\pi_i+\Z m_i\updelta,
$$
where $\pi_i$ is an element of $\wl$ such that $\cl(\pi_i)=\varpi_i$.
Note that $m_i=(\al_i,\al_i)/2$ in the case when $\g$ is the dual of an untwisted affine algebra, and $m_i=1$ otherwise.
Then, for $x,y\in \ko^\times$, we have (see \cite[Section 1.3]{AK})
\begin{align} 
V(\varpi_i)_x \simeq V(\varpi_i)_y  
\quad \text{if and only if \quad $x^{m_i}=y^{m_i}$.}\label{eq:mi}
\end{align}
We set
\eq
 \sig \seteq I_0 \times \cor^\times / \sim, 
\label{eq:sig}
\eneq
where the equivalence relation is given by 
\begin{align} \label{Eq: equiv on I_0 k} 
 (i,x) \sim (j,y) \Longleftrightarrow
V(\varpi_i)_x \simeq V(\varpi_j)_y
 \Longleftrightarrow\text{$i=j$ and $x^{m_i}=y^{m_j}$.}
\end{align}
We denote by $[(i,a)]$ the equivalence class of $(i,a)$ in $ \sig$.  When no confusion arises, we simply write $(i,a)$ for the equivalence class $[(i,a)]$.

\medskip
The monoidal category $\Ca_\g$ is rigid.  For $M\in \uqm$, we denote by  $\rd M$ and $M^*$ the right and the left dual of $M$, respectively.
We set 
\begin{align} \label{Eq: p* tilde p}
p^* \seteq (-1)^{ \langle \rho^\vee, \delta \rangle} q^{\langle c, \rho \rangle} \quad \text{and}\quad  \tilde{p} \seteq (p^*)^2 = q^{2 \langle c, \rho \rangle}.
\end{align}
The integer $\langle \rho^\vee, \delta \rangle$ is called
the \emph{Coxeter number} and  $\langle c, \rho \rangle $ is called
the \emph{dual Coxeter number} (See \cite[Chapter 6]{Kac}).
We write $p^*$ for all types for the reader's convenience. 
\renewcommand{\arraystretch}{1.5}
\begin{align} \label{Table: p*} \small
\begin{array}{|c||c|c|c|c|c|c|c|}
\hline
\text{Type of $\g$} & A_n^{(1)} & B_n^{(1)}& C_n^{(1)}& D_n^{(1)} & A_{2n}^{(2)} & A_{2n-1}^{(2)} & D_{n+1}^{(2)}  \\
&(n\ge1)&(n\ge2)&(n\ge3)&(n\ge4)&(n\ge1)&(n\ge2)&(n\ge3)\\
\hline
 p^*  & (-q)^{n+1} & q^{2n-1}    & q^{n+1}   &  q^{2n-2} & -q^{2n+1} & -q^{2n} & (-1)^{n+1} q^{2n}  \\
\hline
\hline
\text{Type of $\g$} & E_6^{(1)}  & E_7^{(1)} & E_8^{(1)} & F_4^{(1)} & G_{2}^{(1)} & E_{6}^{(2)} & D_{4}^{(3)}  \\
\hline
p^* & q^{12} & q^{18}    & q^{30}   & q^9 & q^4 & -q^{12} & q^6  \\
\hline
\end{array}
\end{align}
Then, for any $M \in \uqm$, we have
$$ 
 M^{**} \simeq M_{ (\tilde{p})^{-1} }, \qquad    {}^{**}\ms{-1mu}M \simeq M_{\tilde{p} }.
$$
and for $i\in I_0$ and $x\in \cor^\times$, 
\begin{equation}\label{eq: dual}
\bigl( V(\varpi_i)_x\bigr)^*  \simeq   V(\varpi_{i^*})_{(p^*)^{-1}x},
\qquad {}^*\bigl(V(\varpi_i)_x \bigr) \simeq   V(\varpi_{i^*})_{p^*x},
\end{equation}
where  $i^*\in I_0$ is defined by  $\al_{i^*}=-w_0\,\al_i$ (see \cite[Appendix A]{AK}).
Note that the involution $i \mapsto i^*$ is the identity for all types except type $A_n, D_n, E_6$, which are given as follows:
\bna
\item (Type $A_n$)   $i^* =  n+1-i $, 
\item (Type $D_n$) 
 $i^* =
\begin{cases}
n-(1-\epsilon)& \text{ if  $n$ is odd and $i=n-\epsilon$ ($\epsilon=0,1$)},  \\
i& \text{ otherwise,}
\end{cases}
 $
\item (Type $E_6$)  The map $i \mapsto i^*$ is determined by 
$$i^* =
\begin{cases}
6 & \text{ if } i=1,  \\
i & \text{ if } i=2,4,  \\
5 & \text{ if } i=3,  
\end{cases}
 $$
where the Dynkin diagram of type $E_6$ is given in Appendix~\ref{App: denominators}~\eqref{Eq: DQ}. 
\ee

\subsection{R-matrices} \label{subsec:R}

We recall the notion of R-matrices on $U_q'(\g)$-modules and their coefficients (see \cite{D86}, and also 
\cite[Appendices A and B]{AK} and \cite[Section 8]{Kas02} for details).
Choose a basis $\{P_\nu\}_\nu$ of $U_q^+(\g)$ and a basis $\{Q_\nu\}_\nu$ of $U_q^-(\g)$ dual to each other with respect to a suitable coupling between
$U_q^+(\g)$ and $U_q^-(\g)$. For $\uqpg$-modules $M$ and $N$, we define
\begin{align*}
\Runiv_{M,N}(u\otimes v) \seteq q^{(\wt(u),\wt(v))} \sum_\nu P_\nu v\otimes Q_\nu u
\qt{  for $u\in M$ and $v\in N$,}
\end{align*}
so that
$\Runiv_{M,N}$ gives a $\uqpg$-linear homomorphism $M\otimes N \rightarrow N\otimes M$, called the \emph{universal R-matrix}, 
provided that the infinite sum has a meaning.
As $\Runiv_{M,N_z}$ converges in the $z$-adic topology for 
$ M,N\in \Ca_\g$, we have
a morphism of
$\ko((z))\tens\uqpg$-modules
\begin{align*} 
\Runiv_{M,N_z} \colon \ko((z))\tens_{\ko[z^{\pm1}]} (M \tens N_z) \To \ko((z))\tens_{\ko[z^{\pm1}]} (N_z\tens M).
\end{align*}
Note that  $\Runiv_{M,N_z}$ is an isomorphism.

Let $M$ and $N$ be non-zero modules in $\uqm$.
The universal R-matrix  $\Runiv_{M,N_z}$ is \emph{rationally renormalizable}
if there exists $f(z) \in \ko((z))^\times$ 
such that
$$f(z) \Runiv_{M,N_z}\bl M\tens N_z\br\subset N_z\tens M. $$
In this case, we can choose
$c_{M,N}(z) \in \ko((z))^\times$ such that,
 for any $x \in \ko^\times$, the specialization of $\Rren_{M,N_z} \seteq c_{M,N}(z)\Runiv_{M,N_z}
\col M \otimes N_z \to N_z \otimes M$ at $z=x$
$$  \Rren_{M,N_z}\big\vert_{z=x} \colon M \otimes N_x  \to N_x \otimes M$$
does not vanish.  Note that $\Rren_{M,N_z}$ and $c_{M,N}(z)$ are unique up to a multiple of $\cz^\times = \bigsqcup_{\,n \in \Z}\ko^\times z^n$. 
We call $c_{M,N}(z)$  the \emph{renormalizing coefficient}.
We denote by $\rmat{M,N}$ the specialization at $z=1$ 
\eq
\rmat{M,N} \seteq \Rren_{M,N_z}\vert_{z=1} \colon M \otimes N \to N \otimes M,
\label{eq:rmat}
\eneq
and call it the \emph{R-matrix}. The  R-matrix $\rmat{M,N}$ is well-defined up to a constant multiple whenever $\Runiv_{M,N_z}$ is \rr. By the definition, $\rmat{M,N}$ never vanishes.

\medskip
Suppose that $M$ and $N$ are simple $\uqpg$-modules in $\uqm$.
Let $u$ and $v$ be dominant extremal weight vectors of $M$ and $N$,
respectively. 
Then there exists $a_{M,N}(z)\in\ko[[z]]^\times$
such that
$$
\Runiv_{M,N_z}\big( u \tens v_z\big)= a_{M,N}(z)\big( v_z\tens u \big).
$$
Thus we have  a unique $\ko(z)\tens\uqpg$-module isomorphism
\begin{align*}
 \Rnorm_{M,N_z}\seteq a_{M,N}(z)^{-1} &  \Runiv_{M,N_z}\big\vert_{\;\ko(z)\otimes_{\ko[z^{\pm1}]} ( M \otimes N_z) } 
\end{align*}
from $\ko(z)\otimes_{\ko[z^{\pm1}]} \big( M \otimes N_z\big)$ to $\ko(z)\otimes_{\ko[z^{\pm1}]}  \big( N_z \otimes M \big)$, which
satisfies
\begin{equation*}
\Rnorm_{M, N_z}\big( u  \otimes v_z\big) = v_z\otimes u .
\end{equation*}

We call $a_{M,N}(z)$ the {\it universal coefficient} of $M$ and $N$, and $\Rnorm_{M,N_z}$ the {\em normalized $R$-matrix}.

Let $d_{M,N}(z) \in \ko[z]$ be a monic polynomial of the smallest degree such that the image of $d_{M,N}(z)
\Rnorm_{M,N_z}(M\tens N_z)$ is contained in $N_z \otimes M$, which is called the {\em denominator of $\Rnorm_{M,N_z}$}. 
Then we have
\begin{equation*}
\Rren_{M,N_z}  =  d_{M,N}(z)\Rnorm_{M,N_z}
\col M \otimes N_z \To N_z \otimes M
\qt{up to a multiple of $\cz^\times$.}
\end{equation*}
Thus
\begin{align*}
 \Rren_{M,N_z} =a_{M,N}(z)^{-1}d_{M,N}(z)\Runiv_{M,N_z}
\quad \text{and} \quad  c_{M,N}(z)= \dfrac{d_{M,N}(z)}{a_{M,N}(z)}
\end{align*}
up to a multiple of $\ko[z^{\pm1}]^\times$.
In particular, $\Runiv_{M,N_z}$ is \rr whenever $M$ and $N$ are simple. 

The denominator formulas between fundamental representations are
recollected for all types 
in Appendix \ref{App: denominators}.

The following theorem follows from the results of \cite{AK,Chari,Kas02,KKKO15}.
In the theorem, (ii) follows essentially from \cite[Corollary 3.16]{KKKO15} with properties of R-matrices (see also \cite[Proposition 3.16, Corollary 3.17]{KKOP19C}), and (i), (iii), (iv) were conjectured in \cite[Section 2]{AK} and proved in \cite[Section 4]{AK} for affine type $A$ and $C$,   in \cite[Section 9]{Kas02} for general cases in terms of good modules, and in \cite[Section 4 and Section 6]{Chari} using the braid group actions.

\begin{theorem}[{\cite{AK,Chari,Kas02,KKKO15}}]  \label{Thm: basic properties}
	\hfill
	\bnum
	\item For good modules $M$ and $N$, the zeroes of $d_{M,N}(z)$ belong to
	$\C[[q^{1/m}]]q^{1/m}$ for some $m\in\Z_{>0}$. 
	\item \label{it:comm} For simple modules $M$ and $N$ such that one of them is real, $M_x$ and $N_y$ strongly commute to each other if and only if $d_{M,N}(z)d_{N,M}(1/z)$ does not vanish at $z=y/x$.
	\item  Let $M_k$ be a good module
	with a dominant extremal vector $u_k$ of weight $\lambda_k$, and
	$a_k\in\ko^\times$ for $k=1,\ldots, t$.
	Assume that $a_j/a_i$ is not a zero of $d_{M_i, M_j}(z) $ for any
	$1\le i<j\le t$. Then the following statements hold.
	\bna
	\item  $(M_1)_{a_1}\otimes\cdots\otimes (M_t)_{a_t}$ is generated by $u_1\otimes\cdots \otimes u_t$.
	\item The head of $(M_1)_{a_1}\otimes\cdots\otimes (M_t)_{a_t}$ is simple.
	\item Any non-zero submodule of $(M_t)_{a_t}\otimes\cdots\otimes (M_1)_{a_1}$ contains the vector $u_t\otimes\cdots\otimes u_1$.
	\item The socle of $(M_t)_{a_t}\otimes\cdots\otimes (M_1)_{a_1}$ is simple.
	\item  Let $\rmat{}\col (M_1)_{a_1}\otimes\cdots\otimes (M_t)_{a_t} \to (M_t)_{a_t}\otimes\cdots\otimes (M_1)_{a_1}$  be the specialization of 
	$\rmat{M_1,\ldots, M_t}\seteq\prod\limits_{1\le j<k\le t}\rmat{M_j,\,M_k}$ at $z_k=a_k$.
	 \ro see \eqref{eq:rmat}\rf.
	Then the image of $\rmat{}$ is simple and it coincides with the head of $(M_1)_{a_1}\otimes\cdots\otimes (M_t)_{a_t}$
	and also with the socle of $(M_t)_{a_t}\otimes\cdots\otimes (M_1)_{a_1}$.
\end{enumerate}
\item\label{Thm: bp5}
For any simple integrable $U_q'(\g)$-module $M$, there exists
a finite sequence 
in  $\sig$  \ro see \eqref{eq:sig}\rf
such that
$M$ 
has $\sum_{k=1}^t \varpi_{i_k}$ as a dominant extremal weight
and it is isomorphic to a simple subquotient of
$V(\vpi_{i_1})_{a_1}\tens\cdots V(\vpi_{i_t})_{a_t}$.
Moreover, such a sequence $\big((i_1,a_1),\ldots, (i_t,a_t)\big)$
is unique up to a permutation.

We call
$\sum_{k=1}^t(i_k,a_k)\in\Z^{\oplus \sig}$ the {\em \afwt}
of $M$. 
\end{enumerate}
\end{theorem}

\subsection{Hernandez-Leclerc categories} \label{Sec: HL cat}
Recall $\sig$ in \eqref{eq:sig}.
For $(i,x)$ and $ (j,y) \in \sig$, we put $d$ many arrows  from $(i,x)$ to $(j,y)$, where $d$ is the order of the zeros of $d_{ V(\varpi_i), V(\varpi_j) } \allowbreak ( z_{V(\varpi_j)} / z_{V(\varpi_i)}  )$
at $  z_{V(\varpi_j)} / z_{V(\varpi_i)}  = y/x$.
Then $\sig$ has a quiver structure.  Note that $(i, x)$ and $(j,y)$  are linked in $ \sig$  if and only if 
the tensor product $ V(\varpi_i)_x \otimes V(\varpi_j)_y$ is reducible (\cite[Corollary 2.4]{AK}).
The denominator formulas are explicitly given in Appendix \ref{App: denominators}.

We choose a connected component $\sigZ$ of $\sig$. Since 
a connected component of $\sig$ is unique up to a spectral parameter shift, 
$\sigZ$ is uniquely determined up to a quiver isomorphism.
We set
\eq
q_s=q^{1/2}\qtq q_t=q^{1/3}.
\eneq
The \emph{distance} $\dist(u,v)$ between two vertices $u$ and $v$ in a finite Dynkin diagram is the length of  the  path connecting them.
For example, $\dist(1,4)=2$ in Dynkin diagram of type $D_4$, and $ \dist(1,3)=2$ in Dynkin diagram of type $F_4$. 
We denote by  $\dd(i,j)$ the distance between $i$ and $j$ in the Dynkin diagram of $\g_0$.
For the rest of this paper,  we take the following choices of $\sigma_0(\g)$
(see Table~\eqref{Table: p*} for the range of $n$): 
\begin{align*}
	&\sigma_0(X) \seteq \{(i,(-q)^{p}) \in I_0 \times \ko^{\times} \ | \ p \equiv_2 \dd(1,i)  \} \ \ ( X=A^{(1)}_{n}, D^{(1)}_{n}, E^{(1)}_{k} (k=6,7,8) ),  \allowdisplaybreaks  \\
	& \sigma_0({B_{n}^{(1)}}) \seteq \set{ (i,(-1)^{ n+i}  q_s q^m),  (n,q^m)}{1\le i \le n-1,\  m \in \Z}, \allowdisplaybreaks\\
	&\sigma_0(C^{(1)}_{n}) \seteq \{(i,(-q_s)^{p}) \in  I_0 \times \ko^{\times} \ | \ p \equiv_2 \dd(1,i)  \}, \allowdisplaybreaks  \\
	&\sigma_0(F^{(1)}_{4}) \seteq \{(i,(-1)^iq_s^{2p-\delta_{i,3}})  \in  I_0 \times \ko^{\times} \ |  \ p \in \Z \}, \allowdisplaybreaks  \\
	&\sigma_0(G^{(1)}_{2}) \seteq \{(i,(-q_t)^{p}) \in  I_0 \times \ko^{\times} \ | \ p \equiv_2 \dd(2,  i )  \}, \allowdisplaybreaks  \\
	&\sigma_0(A^{(2)}_{2n}) \seteq \{(i,(-q)^{p}) \in  I_0 \times \ko^{\times}
	\ | \ p \in \Z \},  \allowdisplaybreaks  \\
	&\sigma_0(A^{(2)}_{2n-1}) \seteq \{( i,\pm(-q)^{p}), ( n, (-q)^{r})  | \ 1 \le i  < n, \ p \equiv_2 i+1, \ r \equiv_2 n+1   \}, \allowdisplaybreaks  \\
	& \sigma_0(D^{(2)}_{n+1}) \seteq \{ (i, (\sqrt{-1}^{n+1-i}) ( -q)^p), \ (n, \pm ( -q)^{r}) \ | \ 1\le i < n,\ p \equiv_2 i+1 ,\ r \equiv_2 n+1  \}, \allowdisplaybreaks\\
	& \sigma_0(E^{(2)}_{6}) \seteq \{ (i,\pm q^r), \ (i', \sqrt{-1} (-q )^{r'}) \ | \  i \in \{ 1,2 \},  i' \in \{3,4\}, \  r \equiv_2 i+1, r'\equiv_2 i'+1 \},   \allowdisplaybreaks\\
	& \sigma_0(D^{(3)}_{4}) \seteq \{ (1, q^r),(1, \omega q^r), (1, \omega^2 q^r), (2,-q^{r+1}) \ | \ r\equiv_2 0  \} \ (\omega^2+\omega+1=0),
\end{align*}
where $a \equiv_2 b$ means that $a\equiv b\mod 2$ (see \cite[Section 3.7]{HL10}, \cite[Section 4.1]{KKKO15D}, \cite[Section 6]{KO18} and  \cite[Section 6]{OS19}). 
Note that, in \cite[Section 6]{OS19}, the category $\CaQ$ and $\sigQ$ were dealt only for exceptional cases, but it is easy to obtain $\sigZ$ using $\sigQ$.
We use the notation $B_2^{(1)}$ and $A^{(2)}_{3} $ instead of $C_2^{(1)}$ and $D_{3}^{(2)}$ respectively.
Here we use the standard convention for Dynkin diagrams appeared in \cite[Chapter 4]{Kac} except $A^{(2)}_{2n}$, $A_3^{(2)}$, $B_2^{(1)}$ and $E^{(1)}_{k}$ ($k=6,7,8$), which are given in $\eqref{Eq: DD}$. 

We define $\Ca_\g^0$ to be the smallest full subcategory of $\Ca_\g$ such that 
\bna
\item $\Ca_\g^0$ contains $V(\varpi_i)_x$ for all $(i,x) \in \sigZ$,
\item $\Ca_\g^0$ is stable by taking subquotients, extensions and tensor products.
\ee
For symmetric affine types, this category was introduced in \cite{HL10}. 
Note that every simple module in $\Ca_\g$ is isomorphic to a tensor product of certain spectral parameter shifts  
of some simple modules in $\Ca_\g^0$ (\cite[Section 3.7]{HL10}).

\subsection{The categories $\CaQ$} \label{Sec: CaQ} \
In this subsection, we recall  very briefly 
a certain subcategory $\CaQ$ of $\Ca_\g$ categorifying the coordinate ring $\C[N]$ of the maximal unipotent group $N$
associated with a certain simple Lie algebra.

This subcategory $\CaQ$ was introduced in \cite{HL15} for simply-laced affine type $ADE$, in \cite{KKKO15D} for twisted affine type $A^{(2)}$ and $D^{(2)}$, 
in  \cite{KO18, OhSuh19} for untwisted affine type $B^{(1)}$ and $C^{(1)}$, and in \cite{OS19} for exceptional affine type.
The quantum affine Schur-Weyl duality functor between the finite-dimensional module category of a quiver Hecke algebra and $\CaQ$ was also constructed in \cite{KKK15B} for untwisted affine type $A^{(1)}$ and $D^{(1)}$, 
in \cite{KKKO15D} for twisted affine type $A^{(2)}$ and $D^{(2)}$, 
in \cite{KO18} for untwisted affine type $B^{(1)}$ and $C^{(1)}$, 
in \cite{OS19} for exceptional affine type, 
and in \cite{Fu18} for simply-laced affine type $ADE$ in a geometric manner.

We shall describe $\sigQ$ and $\CaQ$ by using \emph{ $Q$-data} (\cite{FO20}).
 A $Q$-datum generalizes a Dynkin quiver with a height function, which gives a uniform way to describe the Hernandez-Leclerc category $\CaQ$.
We explain briefly following \cite[Section 3]{FO20} (see also \cite[Section 4]{FHOO21} and \cite[Section 6]{KKOP21}). 
Let $\g$ be an affine Kac-Moody algebra and let $\gf$ be the simply-laced finite type Lie algebra 
corresponding to the affine type of $\g$ 
in Table $\eqref{Table: root system}$. 
Let $\If$ be the index set of $\gf$ and let $\Dynkinf$ be the Dynkin diagram for $\gf$. 

We first assume that $\g$ is of untwisted type. We define an Dynkin diagram automorphism $\varrho $ of $\Dynkinf$ as follows.
For $\g=A_n^{(1)}, D_n^{(1)}, E_k^{(1)}$ type ($k=6,7,8$), we set $\varrho :={\rm id}$, and 
for the remaining types,
$\varrho$ is defined as follows (see \cite[Section 3.1]{FO20}).  
{ \footnotesize
	\begin{align*} 
		\hs{-1ex} & \text{ $B_n^{(1)}$-type: }  \big( \Dynkinf: \xymatrix@R=0.5ex@C=3ex{ *{\circ}<3pt>
			\ar@{-}[r]_<{1 \ \ }  &*{\circ}<3pt> \ar@{-}[r]_<{2 \ \ }  &   {}
			\ar@{.}[r] & *{\circ}<3pt> \ar@{-}[r]_>{\,\,\,\  _{2n-2} }
			&*{\circ}<3pt>\ar@{-}[r]_>{\,\,\,\,  _{2n-1} } &*{\circ}<3pt> }, \
		\varrho (k) = 2n-k \big)\hs{-.5ex} \Longrightarrow \Dynkin_{B_{n}}\hs{-.5ex}:\hs{-.5ex}
		\xymatrix@R=0.5ex@C=3ex{ *{\circ}<3pt> \ar@{-}[r]_<{ 1 \ \ }
			&*{\circ}<3pt> \ar@{-}[r]_<{2 \ \ }  &   {} \ar@{.}[r] &
			*{\circ}<3pt> \ar@{}_>{ _{n-1}} &*{\circ}<3pt>\ar@{<=}[l]^>{ \qquad
				\ \  \; \ _{n}}  },
		\allowdisplaybreaks\\
		&\text{ $C_n^{(1)}$-type:}
		\hs{-.4ex} \left(\hs{-1ex}   \ 
		\Dynkinf \hspace{-.5ex} :
		\hspace{-.5ex}\raisebox{1em}{\xymatrix@R=0.5ex@C=2.5ex{
				& & &  *{\circ}<3pt>\ar@{-}[dl]^<{ \  _{n}} \\
				*{\circ}<3pt> \ar@{-}[r]_<{1\ \ }  &*{\circ}<3pt>
				\ar@{.}[r]_<{2 \ \ } & *{\circ}<3pt> \ar@{.}[l]^<{ \ \  _{n-1}}  \\
				& & &   *{\circ}<3pt>\ar@{-}[ul]^<{\quad \ \  _{n+1}}   \\
		}} \hspace{-1ex}   ,\hs{1ex}     \varrho( k) \hspace{-.5ex}  =
		\hspace{-.5ex} \begin{cases} k & \text{ if } k \le n-1, \\ n+1 &
			\text{ if } k = n, \\ n & \text{ if } k = n+1 \end{cases} \right)\hs{-.5ex}
		\Longrightarrow \Dynkin_{C_{n}}\hs{-.5ex}:\hs{-.5ex} \xymatrix@R=0.5ex@C=3ex{
			*{\circ}<3pt> \ar@{-}[r]_<{1 \ \ }  &*{\circ}<3pt> \ar@{-}[r]_<{ 2 \
				\ }  &   {}
			\ar@{.}[r] & *{\circ}<3pt> \ar@{}_>{ _{n-1} } &*{\circ}<3pt>\ar@{=>}[l]^>{ \qquad \ \  \; \ _{ n } }  },
		\allowdisplaybreaks\\
		&\text{ $F_4^{(1)}$-type: } 
		\left(
		\Dynkinf: \raisebox{2em}{\xymatrix@R=3ex@C=3ex{
				& & *{\circ}<3pt>\ar@{-}[d]_<{\quad \ \  2} \\
				*{\circ}<3pt> \ar@{-}[r]_<{1 \ \ }  & *{\circ}<3pt> \ar@{-}[r]_<{3 \
					\ }  & *{\circ}<3pt> \ar@{-}[r]_<{4 \ \ }  & *{\circ}<3pt>
				\ar@{-}[r]_<{5 \ \ }  & *{\circ}<3pt> \ar@{-}[l]^<{ \ \ 6 }
		}},\hs{.5ex} \begin{cases} \varrho( 1)=6, \ \varrho( 6)=1, \\ \varrho( 3)=5,
			\ \varrho( 5)=3, \\ \varrho( 4)=4, \ \varrho( 2)=2 \end{cases}  \right)
		\Longrightarrow \Dynkin_{F_{4}}: \xymatrix@R=0.5ex@C=3ex{
			*{\circ}<3pt> \ar@{-}[r]_<{ 1 \ \ }  &*{\circ}<3pt> \ar@{}[r]_<{2 \
				\ }  &  *{\circ}<3pt> \ar@{<=}[l]^<{ \ \ 3 } &
			*{\circ}<3pt>\ar@{-}[l]^>{ \qquad \ \  \; \ 4}  },
		\allowdisplaybreaks\\
		&\text{ $G_2^{(1)}$-type: } 
		\left(  \Dynkinf:  \raisebox{1em}{
			\xymatrix@R=0.5ex@C=3ex{
				& &   *{\circ}<3pt>\ar@{-}[dl]^<{ \ 3} \\
				*{\circ}<3pt> \ar@{-}[r]_<{1 \ \ }  &*{\circ}<3pt>
				\ar@{-}[l]^<{2 \ \ }   \\
				& &    *{\circ}<3pt>\ar@{-}[ul]^<{\quad \ \  4} \\
		}}, \ \begin{cases} \varrho( 1)=3, \ \varrho( 3)=4, \ \varrho( 4)=1, \\
			\varrho( 2)=2 \end{cases} \right)  \Longrightarrow\Dynkin_{G_{2}}:
		\xymatrix@R=0.5ex@C=3ex{ *{\circ}<3pt> \ar@{-}[r]_<{ 1 \ \ }
			&*{\circ}<3pt>
			\ar@{<=}[l]^<{ \ \ 2 }  }.
	\end{align*}
}

Let $I_0 =\{ 1,2,\ldots, n \}$ be the index set of
$\g_0$. Note that $\If = I_0$ when $\g=A_n^{(1)}, D_n^{(1)}, E_k^{(1)}$ ($k=6,7,8$). 
Let $\ord(\varrho)$ be the order of $\varrho$. 
For $i\in \If$, we denote by $\orb(i)$ the orbit of $i$ under the action $\varrho$, and set
$\fd_i  := | \orb(i) | $. 
We identify the set of orbits of $\If$ with $I_0$ by mapping $ \orb(i) \mapsto \min\{ \orb(i) \}$ for $\g \ne F_4^{(1)}$, and by mapping $ \orb(1) \mapsto 1 $, $ \orb(3) \mapsto 2 $, $ \orb(4) \mapsto 3 $ and $ \orb(2) \mapsto 4 $ for $\g = F_4^{(1)}$.
We write $\pi\col \If\to I_0$ for the
projection via this identification. 

\smallskip

\begin{definition} [{\cite[Definition 3.5]{FO20}}]	
	\label{def: height} A function $\xi \colon \If \to \Z$ is called a
	\emph{height function on $(\Dynkinf, \varrho)$} if the following two
	conditions are satisfied. \bnum
	\item \label{it:ht1}
	For any $i, j \in \If$ such that $\dist(i,j)=1$ and $\fd_i=
	\fd_j$, we have $|\xi_{i} - \xi_{j}| =\fd_i$.
	\item \label{it:ht2}
	For any $i,j \in \If$ such that $\dist(i,j)=1$ and $1=\fd_i <
	\fd_j=\ord(\varrho)$, there exists a unique element $j^{\circ} \in
	\orb(j)$ such that $|\xi_{i} - \xi_{j^{\circ}}| = 1$ and
	$\xi_{\varrho^{k}(j^{\circ})} = \xi_{j^{\circ}} - 2k$
	for any $0 \le k < \ord(\varrho)$.
\end{enumerate}
Here $\dist(i,j)$ denotes the distance between $i$ and $j$ in the Dynkin diagram $\Dynkinf$.
We call the triple $Q = (\Dynkinf, \varrho, \xi)$ a
\emph{$Q$-datum} for $\g$.
\end{definition}

For a $Q$-datum $Q = (\Dynkinf, \varrho, \xi)$ associated to $\g$, let
\begin{align*}
\hI_Q \seteq \{ ( i ,p) \in \If \times \Z \ | \ p -\xi_i \in 2\fd_i\Z\}.
\end{align*}
The \emph{generalized $\varrho$-Coxeter element} $\tau_Q \in \W_{\rm fin} \rtimes \textrm{Aut} (\Dynkinf)$ associated with $Q$ is defined in \cite[Definition 3.33]{FO20},  which can be understood as a generalization of a Coxeter element. Here $\W_{\rm fin}$ is the Weyl group of $\gf$.

For $i\in I_0$, we denote by $o(i)$ the corresponding orbit of $\If$. 
For each $i\in I_0$, we denote by $i^\circ $ the unique vertex in the orbit $o( i)$ satisfying $\xi_{i^\circ} = \max\{ \xi_j \mid j\in  o(i) \}$. 
 In this paper, we assume further that the height function $\xi$ satisfies:
 \begin{align} \label{Eq: ex cond}
 	\text{ $\xi_{\varrho^k(i^\circ)} = \xi_{i^\circ} - 2k$  for each $i\in I_0$ and $0 \le k < \fd_i $. }
 \end{align}
 Let $\{ i_1, i_2, \ldots, i_n\}$ be a total order of $I_0$ satisfying $\xi_{i_1^\circ} \ge \xi_{i_2^\circ}  \ge \cdots \ge \xi_{i_n^\circ} $.
Then we have
 $$
 \tau_Q = s_{i_1^\circ} s_{i_2^\circ} \cdots s_{i_n^\circ} \varrho \in \W_{\rm fin} \rtimes \textrm{Aut} (\Dynkinf)
 $$ 
(see \cite[Section 3.6]{FO20} and also \cite[Proposition 4.4]{FHOO21} for 
more details).

Let $\Df$ be the set of positive roots of $\gf$,
and let
$\widehat{\Phi} \seteq \Df \times \Z$. For each $i \in
\If$, we define
$$ 
\gamma^Q_i \seteq (1-\tau_Q^{\fd_i})\La_i \in \Df,
$$
 where $\La_i$ is the $i$-th fundamental weight of $\gf$. 
It is shown in \cite[Section 2.2]{HL15} and \cite[Theorem 3.35]{FO20} that there exists a unique
bijection $\psi_Q \col \hI_Q  \to \widehat{\Phi}$ defined
inductively as follows:
\begin{eqnarray*}&&
	\parbox{70ex}{
		\bnum
		\item $\psi_Q(i,\xi_i)=(\ga^Q_i,0)$,
		\item if $\psi_Q(i,p)=(\be,m)$, then we define
		\bna
		\item $\psi_Q(i,p\pm2\fd_i) = (\tau_Q^{\mp \fd_i} (\be),m) \qquad\qquad\quad$ if $ \tau_Q^{ \mp \fd_i}(\be) \in \Df$,
		\item $\psi_Q(i,p\pm 2\fd_i) = (-\tau_Q^{\mp \fd_i}(\be),m\pm 1)\qquad \ \ $ if $\tau_Q^{\mp \fd_i}(\be) \in - \Df$.
	\end{enumerate}
\end{enumerate}
}\label{eq: twisted bijection}
\end{eqnarray*}
Let 
$
I_Q\seteq \psi_Q^{-1}(\Df\times\{0\})\subset\If\times\Z.
$
Then one can describe 
$$
I_Q=\st{(i,p)\in\hI_Q\mid \xi_{ i^*} - \ord(\varrho)\mathsf{h}^\vee < p \le \xi_{i}},
$$
 where $\mathsf{h}^\vee$ is the dual Coxeter number of $\g_0$ 
(see \cite[Theorem 3.35]{FO20} and also see \cite[Proposition 4.15]{FHOO21}).
We define 
$$ 
\sigQ := \{ \esig(i,p ) \mid (i,p) \in I_Q \}, 
$$ 
where we set  $\qm \seteq q^{1/\ord(\varrho)}$ and  
\begin{align*} 
	\esig(i,p) := 
	\begin{cases}
		(\pi(i), (-\qm)^p) & \text{ if } \g = A^{(1)}_n, \; C_n^{(1)}, \; D^{(1)}_n, \; E_{6,7,8}^{(1)}, \; G_2^{(1)}, \\
		(\pi(i), (-1)^{i+n}(\qm)^p) & \text{ if } \g = B_n^{(1)}, \\
		(\pi(i), (-1)^{ \pi(i) }(\qm)^p) & \text{ if } \g = F_4^{(1)},
	\end{cases}
\end{align*}
(see \cite[Section 5.4]{FO20}). We define  
\begin{align} \label{Eq: phi}
\phi_Q\col \Df \isoto\sigQ
\end{align} 
by $ \phi_Q( \beta) :=  \esig \circ \psi_Q^{-1} (\beta, 0)$ for $\beta \in \Df$. The map $\phi_Q$ is bijective.

\smallskip

For the rest of this paper,  we take the following choices of $Q$-data:
\begin{itemize}
	\item for simple-laced $ADE$ type, $\ord(\varrho)=1$ and the height function $\xi$ is defined in Appendix \ref{App: ADE}, 
	\item for $\g=B_n^{(1)}$, $\ord(\varrho)=2$ and
	$ Q = \hspace{-2.5ex} \xymatrix@C=4ex@R=3ex{ *{ \circ }<3pt>
		\ar@{->}[r]^{ _{\ble{2n-3}} \qquad \quad}_<{1}  &*{\circ}<3pt>
		\ar@{->}[r]^{  _{\ble{2n-5}} \qquad \quad }_<{2}  &   {} \ar@{.}[r]&
		*{\circ}<3pt> \ar@{->}[r]^{ \ble{1} \qquad \ \  }_{ _{n-1} \quad \ \
			\ \ \ }&  *{\circ}<3pt> \ar@{->}[r]^{ \ble{0} \qquad \quad }_<{ _{n}
			\ \ } & *{\circ}<3pt> \ar@{-}[l]_{  \qquad \ \ \ble{-1}  }^{ \ \ \
			\quad\ \ _{n+1} }  & *{\circ}<3pt> \ar@{->}[l]_{  \qquad \ \ \ble{1}
		}^<{_{n+2}}
		{} \ar@{.}[r] &  *{\circ}<3pt> \ar@{<-}[r]^<{ _{\ble{2n-7}} \ \ }_<{ _{2n-2} }&  *{\circ}<3pt> \ar@{-}[l]_<{ \ \ _{\ble{2n-5}} \ \ }^<{ _{2n-1} } } \hs{-1ex}$,
	\item 
	for $\g=C_n^{(1)}$, $\ord(\varrho)=2$ and
	$ Q  =
	\raisebox{2.3em}{\xymatrix@C=3ex@R=3ex{  &&&& *{ \circ }<3pt>  \ar@{<-}[d]_<{\ble{-n-1}}^<{ _{n+1} } \\
			*{ \circ }<3pt> \ar@{->}[r]^{ _{\ble{0}}  \qquad}_<{1 \ \ }
			&*{\circ}<3pt> \ar@{->}[r]^{  _{\ble{-1}} \qquad  }_<{2  \ \
			}  &   {} \ar@{.}[r]&   \ar@{->}[r] &  *{\circ}<3pt> \ar@{->}[r]^{ \ble{-n+2 \quad } \qquad
				\quad \ }_<{ _{n-1} \ \ }    & *{\circ}<3pt> \ar@{-}[l]_{  \qquad \ \	\ble{-n+1} }^<{  \ \  _{n} }  } } $ 
	\item for $\g=F_4^{(1)}$, $\ord(\varrho)=2$ and
	$ Q  = \raisebox{2.3em}{\xymatrix@C=4ex@R=3ex{ && *{\circ}<3pt>\ar@{->}[d]_<{\ble{-2}}^<{2} \\
			*{ \circ }<3pt> \ar@{->}[r]^<{\ble{0}  \ \; }_<{1}  &*{\circ}<3pt>
			\ar@{->}[r]^<{\ble{-2} \ \; }_<{3} &*{ \circ }<3pt> \ar@{->}[r]^<{
				\; \ble{-3}   }_<{4} &*{\circ}<3pt> \ar@{-}[r]^<{\ble{-4}  \ \;
			}_<{5} &*{\circ}<3pt> \ar@{->}[l]_<{\  \ble{-2}}^<{6} }} $,
	\item for $\g=G_2^{(1)}$, $\ord(\varrho)=3$ and 
	$ Q  = \raisebox{2.3em}{\xymatrix@C=4ex@R=3ex{ & *{ \circ }<3pt> \ar@{<-}[d]_<{\ble{-3}}^<{3}  \\
			*{ \circ }<3pt> \ar@{<-}[r]^<{\ble{-1} \ \ }_<{1}   &*{\circ}<3pt>
			\ar@{->}[r]^<{\ble{0}  \ \ \ \  }_<{2}   &*{ \circ }<3pt>
			\ar@{-}^<{\ble{-5} \ }_<{4} } }$,
\end{itemize}
where an underline integer $\ble{*}$ is the value of $\xi_i$ at each vertex $i \in \Dynkinf$ and  an arrow $i \to j$ means that $\xi_i>\xi_j$ and $\dist(i,j)=1$ in the Dynkin diagram $\Dynkinf$.
 Note that our choice of $Q$ satisfies $\eqref{Eq: ex cond}$.
Then, $\tau_Q$ is given as follows:
\begin{itemize}
	\item for simple-laced $ADE$ type, $\tau_Q$ is the same as $\tau$ in Appendix \ref{App: ADE}, 
	\item for $\g=B_n^{(1)}, C_n^{(1)}$, $\tau_Q = s_1 s_2 \cdots s_n \varrho$,
	\item for $\g=F_4^{(1)}$, $\tau_Q = s_1 s_2 s_3 s_4 \varrho$,
	\item for $\g=G_2^{(1)}$, $\tau_Q = s_2 s_1 \varrho$.
\end{itemize}
In this case the set $\sigQ$ is contained in $\sigZ$ in Section \ref{Sec: HL cat}, and  can be written explicitly as follows (where $a\le_2 b$ means that $a\le b$ and $a\equiv b\mod 2$): 
\begin{align*}
	&\sigma_Q(A^{(1)}_{n }) \seteq \{(i,(-q)^{k}) \in \sigma_0(A^{(1)}_{n })  \ |  \    i-2n +1 \le_2  k   \le_2  -i +1   \}, \allowdisplaybreaks  \\
	&\sigma_Q(B^{(1)}_{n}) \seteq \{  (i,(-1)^{n+i}q_s^{k}), (n,q^{k'}) \in \sigma_0(B^{(1)}_{n})  \ |  \ i <n, \  -2n-2i+3  \le_2 k \le_2 2n-2i-1, \allowdisplaybreaks  \\
	&\hspace{46ex}  \ -2n+2 \le k' \le 0    \}, \allowdisplaybreaks  \\
	& \sigma_Q({C_{n}^{(1)}}) \seteq \{  (i, (-q_s)^{k})  \in \sigma_0(C^{(1)}_{n})  \ | \ -\dd(1,i)-2n  \le_2 k \le_2 -\dd(1,i)   \} , \allowdisplaybreaks \\
	& \sigma_Q({D_{n}^{(1)}}) \seteq \{  (i, (-q)^{k})  \in \sigma_0(D^{(1)}_{n})  \ | \ -\dd(1,i)-2n+4  \le_2 k \le_2  -\dd(1,i)  \} , \allowdisplaybreaks \\
	& \sigma_Q({E_{6}^{(1)}}) \seteq \{  (i,(-q)^k)  \in \sigma_0(E^{(1)}_{6})  \ | \    \dd(1,i)-14  \le_2  k \le_2 -\dd(1,i) + 2\delta_{i,2} \}, \  \allowdisplaybreaks  \\
	& \sigma_Q({E_{7}^{(1)}}) \seteq \{  (i,(-q)^k)  \in \sigma_0(E^{(1)}_{7})  \ | \    -\dd(1,i)-16 + 2\delta_{i,2} \le_2 k \le_2-\dd(1,i) + 2\delta_{i,2} \}, \  \allowdisplaybreaks  \\
	& \sigma_Q({E_{8}^{(1)}}) \seteq \{  (i,(-q)^k)  \in \sigma_0(E^{(1)}_{8})  \ | \    -\dd(1,i)-28 + 2\delta_{i,2} \le_2 k \le_2 -\dd(1,i) + 2\delta_{i,2} \}, \  \allowdisplaybreaks  \\
	& \sigma_Q({F_{4}^{(1)}}) \seteq \{  (i, (-1)^iq^{k})  \in \sigma_0(F^{(1)}_{4})  \ | \    \dd(i,3)-10 + \frac{\delta_{i,3}}{2}  \le k \le \dd(i,3)-2 + \frac{\delta_{i,3}}{2}  \},  \allowdisplaybreaks  \\
	& \sigma_Q({G_{2}^{(1)}}) \seteq \{  (i, (-q_t)^{k})  \in \sigma_0(G^{(1)}_{2})  \ | \ -\dd(2,i)-10 \le_2 k  \le_2 -\dd(2,i)  \} , \allowdisplaybreaks 
\end{align*}
where $\dd(i,j)$ denotes the distance between $i$ and $j$ in the Dynkin diagram of $\g_0$.

We now assume that $\g$ is of twisted type. Then one can define 
\begin{align*}
	&\sigma_Q(A^{(2)}_{N}) \seteq \{(i, (-q)^{k} )^\star   \ |  \   (i,(-q)^{k}) \in \sigma_Q(A^{(1)}_{N})  \}, \
	(N=2n-1 \text{ or } 2n) \allowdisplaybreaks \\
	&\sigma_Q(D^{(2)}_{n+1}) \seteq \{(i, (-q)^{k})^\star   \ |  \   (i,(-q)^{k}) \in \sigma_Q(D^{(1)}_{n+1})  \}, \allowdisplaybreaks \\
	&\sigma_Q(E^{(2)}_{6}) \seteq \{(i,( -q)^{k} )^\star \ |  \   (i,(-q)^{k}) \in \sigma_Q(E^{(1)}_{6})  \}, \allowdisplaybreaks \\
	&\sigma_Q(D^{(3)}_{4}) \seteq \{(i,( -q)^{k})^\dagger  \ |  \   (i,(-q)^{k}) \in \sigma_Q(D^{(1)}_{4})  \},
\end{align*}
where, for $(i,a) \in \sigma_0(\g^{(1)}_{N}) $,  we set 
\begin{equation*} 
	\begin{aligned}
		(i,a)^\star =
		\begin{cases}
			(i,a) &\text{ if } \g=A^{(1)}_{N}, \  i  \le \lfloor (N+1)/2  \rfloor \text{ or }  \g=E^{(1)}_{6}, \  i =1,  \\
			(N+1-i,(-1)^Na)  & \text{ if } \g=A^{(1)}_{N}, \  i  > \lfloor (N+1)/2 \rfloor, \\
			(i,\sqrt{-1}^{n+1-i}a)  & \text{ if } \g=D^{(1)}_{n+1}, \  i  \le  n-1,  \\
			(n,(-1)^ia)  & \text{ if } \g=D^{(1)}_{n+1},  \  i\in \{n,n+1\},  \\
			(2,a)   & \text{ if } \g=E^{(1)}_{6}, \ i=3, \\
			(2,-a)    & \text{ if } \g=E^{(1)}_{6}, \ i=5, \\
			(1,-a)   & \text{ if } \g=E^{(1)}_{6}, \ i=6 , \\
			(3,\sqrt{-1}a)   & \text{ if } \g=E^{(1)}_{6}, \ i=4 , \\
			(4,\sqrt{-1}a)  & \text{ if } \g=E^{(1)}_{6}, \ i=2,
		\end{cases}
	\end{aligned}
\end{equation*}
and
\begin{align*} 
	(i,a)^\dagger =
	\begin{cases}
		(2,a) & \text{ if } i=2 \\
		(1,(\delta_{i,1}+\delta_{i,3} \omega+\delta_{i,4} \omega^2)a) & \text{ if } i \ne 2,
	\end{cases}
\end{align*}
(see \cite[Proposition 4.3]{KKKO15D} and \cite[Proposition 6.5]{OS19} for details of $\star$ and $\dagger$).
The bijection $\phi_Q\col \Df \isoto\sigQ$ is defined by composing the bijection for untwisted type with the maps $\star$ and $\dagger$.

Comparing the above descriptions of $\sigQ$ with the descriptions of $\sigZ$ given in Section \ref{Sec: HL cat}, one can easily show that 
\begin{equation} \label{Eq: sigZ and sigQ}
\begin{aligned}
&\sigZ = \bigsqcup_{k \in \Z }  \sigQ ^{*k},   \\
&\sigQ ^{*k} \cap \sigQ ^{*k'} = \emptyset \quad  \text{ for } k,k' \in   \Z  \text{ with } k \ne k', 
\end{aligned}
\end{equation}
where  $\sigQ ^{*k} \seteq \{ (i^{*k},(p^*)^ka) \mid (i,a) \in \sigQ  \}$ and 
$ i^{*k} =  
\begin{cases}
i & \text{ if $k$ is even}, \\
i^* & \text{ if $k$ is odd,}
\end{cases}
$\\ 
(See \cite[Proposition 5.9]{FO20}.)
Note that $p^*$ is given in \eqref{Eq: p* tilde p}.

\vs{1ex}
Let $\CaQ$ be the smallest full subcategory of $\Ca_\g^0$ such that 
\bna
\item $\CaQ$ contains $\trivial$ and $V(\varpi_i)_x$ for all $(i,x) \in \sigQ$,
\item $\CaQ$ is stable by taking subquotients, extensions and tensor products.
\ee
It was shown in \cite[Theorem 6.1]{HL15}, \cite[Corollary 5.6]{KKKO15D}, \cite[Corollary 6.6]{KO18} and \cite[Section 6]{OS19}
that the Grothendieck ring $K(\CaQ)$ of the monoidal category $\CaQ$ is isomorphic to 
the coordinate ring $\C[N]$ of the maximal unipotent group $N$ associated with $\gf$. 
The set $\Df$ has a convex order $\prec_Q$ arising from $Q$. 

Let $\beta \in \Df$ and write $(i,a) = \phi_Q(\beta)$. Then we set 
$$
V_Q(\beta) \seteq V(\varpi_i)_a \in \CaQ.
$$
Under the categorification, the modules $V_Q(\beta)$ correspond to the dual PBW vectors of $ \C[N]$ with respect to the convex order $\prec_Q$ on $\Df$.

The proposition below follows from  \cite[Section 4.3]{KKK15B}, \cite[Proposition 4.9, Theorem 5.1]{KKKO15D}, \cite[Section 4.3]{KO18} and \cite[Section 6]{OS19}.

\begin{proposition} [\protect \cite{KKK15B,KKKO15D, KO18,OS19}]   \label{Prop: CQ categorification}  \
For a minimal pair $(\al, \be)$ of a positive root $\gamma \in \Df$, $V_Q(\gamma)$ is isomorphic to the head of $ V_Q(\al) \otimes V_Q(\beta) $.
Here, $(\al, \be)$ is called a minimal pair of $\gamma$ if $ \al \prec_Q \beta $, $\gamma = \al+\beta$ and there exists no pair $(\al', \be')$ such that 
$ \gamma = \al'+\be' $ and $\al \prec_Q \al' \prec_Q \be' \prec_Q \beta  $.
\end{proposition}

\vskip 2em 

\section{New invariants for pairs of modules} \label{Sec: new invariants}

In this section, we recall several properties of the new invariants arising from $R$-matrices introduced in \cite{KKOP19C}.

We set
\begin{align*} 
 \varphi(z) \seteq \prod_{s=0}^\infty (1-\tp^{s}z)
 =\sum_{n=0}^\infty\hs{.3ex}\dfrac{(-1)^n\tp^{n(n-1)/2}}{\prod_{k=1}^n(1-\tp^k)}\;z^n
 \in\ko[[z]]\subset\corh[[z]],
\end{align*}
where $\tilde{p}$ is given in $\eqref{Eq: p* tilde p}$. We consider the subgroup $\G$ of $\cor((z))^\times$ given by 
\begin{align*}
\G \seteq \left\{ cz^m \prod_{a \in \ko^\times} \varphi(az)^{\eta_a} \ \left|  \
\begin{matrix} \ c \in \ko^\times, \ m \in \Z , \\
\eta_a \in \Z \text{ vanishes except finitely many $a$'s. } \end{matrix} \right. \right\}.
\end{align*}
Note that, if $\Runiv_{M,N_z}$ is rationally renormalizable for $M, N \in \uqm$, then the renormalizing coefficient $c_{M,N}(z)$ belongs to $\G$ (see \cite[Proposition 3.2]{KKOP19C}).
In particular, for simple modules $M$ and $N$ in $\uqm$, the universal coefficient $a_{M,N}(z)$ belongs to $\G$.

For a subset $S$ of $\Z$, let $\tp^{S} \seteq \{ \tp^k \ | \ k \in S\} $.
We define the group homomorphisms
\begin{align*}
\Deg \col   \Bg \to  \Z \quad \text{ and } \quad \Di \col   \Bg \to  \Z,
\end{align*} by
$$
\Deg(f(z)) = \sum_{a \in \tp^{\,\Z_{\le 0}} }\eta_a - 
\sum_{a \in \tp^{\,\Z_{> 0}} } \eta_a \quad \text{ and } \quad \Di(f(z)) = \sum_{a \in \tp^{\,\Z}} \eta_a
$$
for $f(z)=cz^m \prod_{  a\in\cor^\times } \varphi(az)^{\eta_a} \in \Bg$.

\Lemma [\protect{\cite[Lemma 3.4]{KKOP19C}}]   \label{Lem:properties of Deg}
Let $f(z)\in\Bg$.
\bnum
\item \label{Lem: pdeg:1} 
If $f(z)\in  \cor(z)^\times $, then we have $f(z)\in \Bg$ and
$$\Deg^\infty(f(z))=0,\text{\ and\;\ }
\Deg(f(z))=2\zero_{z=1}f(z).$$
\item If $g(z), \;h(z)\in\G$ satisfy
$g(z)/h(z)\in\cz$, then $\Deg(h(z))\le\Deg(g(z))$.
\label{degpol}
\item $\Deg^\infty f(z)=  -\Deg\bl f(\tp^{n}z)\br=\Deg\bl f(\tp^{-n}z)\br$
 for $n \gg 0$. \label{it:ninf}
\item  \label{Lem: pdeg:4}
 If $\Deg^\infty\bl f(cz)\br=0$ for any $c\in\cor^\times$, then
$f(z)\in\cor(z)^\times$.
\ee
\enlemma

The following invariants for a pair of modules $M$, $N$ in $\uqm$ such that $\Runiv_{M,N_z}$ is \rr  
have been introduced in \cite{KKOP19C} by using the homomorphisms $\Deg$ and $\Di$.
\begin{definition} \label{def: Lams}
For non-zero modules $M$ and $N$ in $\uqm$ such that $\Runiv_{M,N_z}$ is \rr, we define the integers $\Lambda(M,N)$ and $\Lambda^\infty(M,N)$ by
\begin{align*}
  \Lambda(M,N) & \seteq\Deg(c_{M,N}(z)),\\
\Lambda^\infty(M,N) & \seteq\Deg^\infty(c_{M,N}(z)).
\end{align*}
\end{definition}

We have $\La(M,N)\equiv
\Li(M,N)\mod 2$.

\begin{proposition}  [\protect{\cite[Lemma 3.7]{KKOP19C}}] \label{Prop: para shift for La}
For any simple modules $M,N\in \Ca_\g$ and $x \in \cor^\times$, we have 
\eqn&&\La(M,N)=\La(M_x,N_x) \qtq
\Li(M,N)=\Li(M_x,N_x).
\eneqn

\end{proposition}

\begin{proposition}  [\protect{\cite[Lemma 3.7, 3.8 and Corollary 3.23 ]{KKOP19C}}] \label{Prop: Li}
Let $M,N$ be simple modules in $\Ca_\g$.
\bnum
\item  $\Li(M,N)   =-\Deg^\infty(a_{M,N}(z))$, 
\item  \label{Prop: Li: item2} $ \Li(M,N) = \Li(N,M)$,
\item $ \Li(M,N) = - \Li(M^*, N) = - \Li(M,\rd N)$. 
\item In particular, $ \Li(M,N) = \Li(M^*, N^*) = \Li( \rd M,\rd N)$.
\ee
\end{proposition}

\begin{proposition} [\protect{\cite[Lemma 3.7  and Proposition 3.18]{KKOP19C}}]
Let $M,N$ be simple modules in $\Ca_\g$.
\bnum
\item $ \La(M,N) =  \La(N^*, M) =  \La(N,\rd M)$. 
\item In particular, 
\begin{align*}
 \La(M,N) &= \La(M^*, N^*) = \La( \rd M,\rd N).
\end{align*}
\ee
\end{proposition}

\Prop  [\protect{\cite[Proposition 3.9]{KKOP19C}}]  \label{Prop: subquotients for Li}
Let $M$ and $N$ be modules in $\uqm$, and let
$M'$ and $N'$ be a non-zero subquotient of $M$ and $N$, respectively.
Assume that $\Runiv_{M,N_z}$ is \rr.
Then $\Runiv_{M',N'_z}$ is \rr, and
$$\La(M',N')\le \La(M,N)\qtq \Li(M',N')=\Li(M,N).$$
\enprop

\begin{proposition} [\protect{\cite[Proposition 3.11]{KKOP19C}}] \label{Prop: subquotient for Li}
Let $M$, $N$ and $L$ be non-zero modules in $\uqm$,
and let $S$ be a non-zero subquotient of $M\tens N$.
\bnum
\item
Assume that $\Runiv_{M,L_z}$ and $\Runiv_{N,L_z}$ are \rr.
Then $\Runiv_{S,L_z}$ is \rr and
$$ \Lambda(S,L)\le  \Lambda(M,L) + \Lambda(N,L)
\qtq
\Li(S,L)=\Li(M,L) + \Li(N,L).
$$
\item
Assume that $\Runiv_{L,M_z}$ and $\Runiv_{L,N_z}$ are \rr.
Then $\Runiv_{L,S_z}$ is \rr and
$$\Lambda(L,S)\le\Lambda(L,M) + \Lambda(L,N)
\qtq \Li(L,S)=\Li(L,M) + \Li(L,N).$$
\ee
\end{proposition}

\begin{corollary} [\protect{\cite[Corollary 3.12]{KKOP19C}}] \label{Cor: Li and s}
Let $M,N$ be simple modules in $\Ca_\g$. Suppose that 
$M$ (resp.\ $N$) is isomorphic to a subquotient  of $V(\varpi_{i_1 })_{a_1} \otimes V(\varpi_{i_2 })_{a_2}\otimes \cdots\otimes V(\varpi_{i_k})_{a_k}$ 
(resp.\ $V(\varpi_{j_1 })_{b_1} \otimes V(\varpi_{j_2 })_{b_2}\otimes \cdots\otimes V(\varpi_{j_l})_{b_l}$).
Then we have 
$$
 \Lambda^{\infty}(M,N) = \sum_{ 1 \le \nu \le k,  \  1 \le \mu \le l} \Lambda^{\infty}(V(\varpi_{i_\nu})_{a_{\nu}},V(\varpi_{j_\mu})_{b_{\mu}}).
$$ 
\end{corollary}

For simple modules $M$ and $N$ in $\uqm$, we define $\de(M,N)$ by
$$ 
\de(M,N) \seteq \dfrac{1}{2}\bl\Lambda(M,N) + \Lambda(M^*,N)\br.
$$

\begin{proposition}  [\protect{\cite[Proposition 3.16 and Corollary 3.19]{KKOP19C}}] \label{Prop: d and d}
Let $M,N$ be simple modules in $\Ca_\g$. Then we have 
\bnum
\item  $ \de(M,N)=\zero_{z=1}\bl d_{M,N}(z)d_{N,M}(z^{-1})\br$,
\item $\de(M,N)=\dfrac{1}{2}\Bigl(\La(M,N)+\La(N,M)\Bigr)$,
\item In particular,  $\de(M,N) = \de(N,M)$.
\ee
\end{proposition}

\begin{corollary}  [\protect{\cite[Corollary 3.17 and 3.20]{KKOP19C}}]  \label{Cor: comm}
Let $M$ and $N$ be simple modules in $\uqm$.
\bnum
\item
Suppose that one of $M$ and $N$ is real.
Then $M$ and $N$ strongly commute if and only if
$\de(M,N)=0$.
\item
In particular, if $M$ is real, then $ \La(M,M)=0.$
\ee

\end{corollary}

\begin{proposition}   [\protect{\cite[Proposition 3.22]{KKOP19C}}] \label{Prop: La and d}
For simple modules $M$ and $N$ in $\uqm$,  we have 
\begin{align*}
\Lambda(M,N) & = \sum_{k \in \Z} (-1)^{k+\delta(k<0)} \de(M,\D^{k}N) , \\
\Lambda^\infty(M,N) &= \sum_{k \in \Z} (-1)^{k} \de(M,\D^{k}N),
\end{align*}
where $\D^kN$ is defined as follows:
$$
\D^kN \seteq
\begin{cases}
(\cdots(N^* \underbrace{ )^* \cdots )^* }_{\text{$(-k)$-times}} & \text{ if } k <0, \\
 \underbrace{\rd  ( \cdots ( }_{\text{$k$-times}} \rd N  )\cdots)  & \text{ if } k \ge 0.
\end{cases}
$$
\end{proposition}

\vskip 2em

\section{Root systems associated with $\Ca_\g$} \label{Sec: Root system}

Let $\Hom(\sig, \Z)$ be the set of $\Z$-valued functions on $\sig$. 
It is obvious that $\Hom(\sig, \allowbreak \Z)$ forms a torsion-free abelian group under addition.
Let $M \in \Ca_\g$ be a module such that $ \Runiv_{M, V(\varpi_i)_z} $ is rationally renormalizable for any $i\in I_0$. 
Then we define $ \rE(M) \in \Hom(\sig, \Z)$ by 
\begin{align}
\rE(M)(i,a)\seteq \Li(M, V(\varpi_i)_a) \qquad \text{ for } (i,a) \in \sig,
\end{align}
which is well-defined by $\eqref{Eq: equiv on I_0 k}$.

\begin{lemma} \label{Lem: prop for E}
Let $M$ and $N$ be simple modules in $\Ca_\g$. 
\bnum

\item \label{Lem: prop for E: item1}
 $  \rE(M) = -\rE(M^*) = - \rE(\rd M )$.
\item \label{Lem: prop for E: item2}
Let $\st{M_k}_{1\le k\le r}$ be a sequence of simple modules.
Then for any non-zero subquotient $S$ of $M_1\tens\cdots\tens M_r$, we have 
$$
\rE(S) =\sum_{k=1}^r\rE(M_k).
$$
\item \label{Lem: prop for E: item3}
$\rE(M) =\rE(N)$ if and only if $ \displaystyle \frac{a_{M,V(\varpi_i)}(z) }{a_{N,V(\varpi_i)}(z) } \in \cor(z)^\times $ for any $i\in I_0$. 
\ee
\end{lemma}
\begin{proof}
\ref{Lem: prop for E: item1}
and  \ref{Lem: prop for E: item2}
easily follow from Proposition \ref{Prop: Li} and Proposition \ref{Prop: subquotient for Li}.

Let us show \ref{Lem: prop for E: item3}.
For $(i,a) \in \sig$, the condition
$\Li(M, V(\varpi_i)_a) = \Li(N, V(\varpi_i)_a)$  is equivalent to
$$
\Di(a_{M, V(\varpi_i)} (az) ) = \Di(a_{N, V(\varpi_i)}(az) ).
$$
Since $\Di\col \Bg \rightarrow \Z$ is a group homomorphism, it is equivalent to
$$
\Di \left( \frac{a_{M, V(\varpi_i)}(az)}{a_{N, V(\varpi_i)}(az)} \right)=0 \qquad \text{ for any }a\in \cor^\times.
$$
Then \ref{Lem: prop for E: item3} follows from 
Lemma~\ref{Lem:properties of Deg}\;\ref{Lem: pdeg:4}. 
\end{proof}

For $(i,a) \in \sig$, we set 
\begin{align*}
\rs_{i,a} \seteq \rE(V(\varpi_i)_a) \in \Hom(\sig, \Z),
\end{align*}
and 
\begin{equation} \label{Eq: W and Delta}
\begin{aligned}
\rW \seteq \{  \rE(M) \mid M \text{ is simple in } \Ca_\g\}, & \qquad \Delta \seteq \{ \rs_{i,a} \mid (i,a) \in \sig  \} \subset \rW, \\
\rW_0 \seteq \{  \rE(M) \mid M \text{ is simple in } \Ca_\g^0 \}, & \qquad \Delta_0 \seteq \{ \rs_{i,a} \mid (i,a) \in \sigZ  \} \subset \rW_0.
\end{aligned}
\end{equation}
It is obvious that $\rW_0 \subset \rW $ and $\Delta_0 \subset \Delta$.

\Lemma \label{Lem: (,)on W}
We have
\bnum
\item \label{Prop: sig and rW: item3} $\rW = \sum_{(i,a) \in \sig} \Z \rs_{i,a} $   and $ \rW_0 = \sum_{(i,a) \in \sigZ} \Z \rs_{i,a}=\sum_{(i,a) \in \sigQ} \Z \rs_{i,a}$.
In particular, $\rW_0$ is a finitely generated free $\Z$-module. 
\item There exists a unique symmetric bilinear form
$(-,-)$ on $\rW$ such that
\begin{align*}
( \rE(M), \rE(N) ) =  - \Li(M,N),
\end{align*}
for any simple modules $M,N \in \Ca_\g$.
\ee
\enlemma

\Proof
(i) follows from Theorem \ref{Thm: basic properties} \ref{Thm: bp5},
Lemma~\ref{Lem: prop for E} and \eqref{Eq: sigZ and sigQ}.

\sn
Let us show (ii).
By Corollary~\ref{Cor: Li and s},  it is reduced to the existence of
the bilinear form $(-,-)$ on $\rW$ such that
$$\bl\rs_{i,a},\rs_{j,b}\br=-\Li\bl V(\vpi_i)_a,V(\vpi_j)_b\br.$$
Therefore it is enough to show that
for a sequence $\st{(i_k,a_k)}_{k=1,\ldots,r}$ in $\sig$ such that
$\sum_{k=1}^r\rs_{i_k,a_k}=0$, we have
$\sum _{k=1}^r\Li\bl V(\vpi_{i_k})_{a_k},V(\vpi_j)_b\br=0$ for any $(j,b)\in\sig$.
Let us take a simple subquotient $M$ of
$ V(\vpi_{i_1})_{a_1}\tens\cdots   \tens V(\vpi_{i_r})_{a_r}$.
Then we have
$\rE(M)= \sum_{k=1}^r\rs_{i_k,a_k}=0$. Hence we obtain
\eqn
\sum _{k=1}^r\Li\bl V(\vpi_{i_k})_{a_k},V(\vpi_j)_b\br&&=
\Li(M,V(\vpi_j)_b)=
-\rE(M)(j,b)=0.
\eneqn

\QED

\begin{lemma} \label{Lem: Li = -2}
For $i\in I_0$ and $a\in\cor^\times$, we have 
\eq
\de( V(\varpi_i), \dual^k V(\varpi_{i}))= \delta(k=\pm1)\qt{for $k\in\Z$.}
\label{eq:dfund}
\eneq
In particular, we have
$$
(\rs_{i,a},\rs_{i,a})=-\Li( V(\varpi_i),  V(\varpi_{i}) ) = 2.
$$
\end{lemma}
\begin{proof}
The statement $\Li( V(\varpi_i),  V(\varpi_{i}) ) = -2$  follows from
\eqref{eq:dfund} and Proposition \ref{Prop: La and d}.

Let us show \eqref{eq:dfund}.
Let $h^\vee$ be the dual Coxeter number of $\g$, and write 
$$
d_{i,j}(z) \seteq d_{V(\varpi_i), V(\varpi_j)} (z)  \qquad \text{ for } i,j \in I.
$$ 
The denominator formula for $d_{i,j}(z)$ is written in Appendix \ref{App: denominators}.
Using this formula, one can easily check that, if $\epsilon q^t$ ($ |\epsilon|=1 $) is a zero of $d_{i,i}(z)$, then 
$t$ should be between $1$ and $h^\vee$.
Combining this with Proposition \ref{Prop: d and d},  we obtain 
\begin{align*}
\de( V(\varpi_i),  V(\varpi_{j})_{ (p^*)^{k}} ) = \zero_{z=1}\bl d_{i,j }( (p^*)^k z) d_{j,i}( (p^*)^{-k} z^{-1})\br = 0
 \qquad \text{ unless } k=\pm 1.
\end{align*}

Now we shall show $\de( V(\varpi_i), \dual^{\pm1} V(\varpi_{ i }) )=1$.

\sn
\textbf{(Case of simply-laced affine $ADE$ type)}
 \quad In this case, the dual Coxeter number is equal to the Coxeter number.
It follows from the denominator formula written in Appendix \ref{App: denominators}, we have 
 $$
\de(  V(\varpi_{i}),\dual^{\pm1} V(\varpi_{i})) = \tcmc_{i,i^*}(h^\vee-1).
 $$
 Since $ \tcmc_{i,j}(k) = \tcmc_{j,i^*}(h^\vee - k) $ for $1 \le k \le h^\vee-1$ (see \cite[Lemma 3.7]{Fu19})
 and $\tcmc_{i ,i}( 1)=1$ by Proposition \ref{Prop: formula for ADE},  we have 
 $$
 \de(  V(\varpi_{i}),\dual^{\pm1} V(\varpi_{i}))=   \tcmc_{i,i}( 1)= 1.
 $$

\sn
\textbf{(Other case)}\quad  
 In this case, we know that  $i^* = i$ for any $i\in I_0$. Thus we have 
 \begin{align*}
\de(  V(\varpi_i), \dual^{\pm1} V(\varpi_{i})) = 
\de( V(\varpi_i),  V(\varpi_{i })_{  p^* } ).
\end{align*}
 Using $\eqref{Table: p*}$ and the denominator formula written in Appendix \ref{App: denominators}, 
 one can compute directly
 $$
 \de( V(\varpi_i),  V(\varpi_{i })_{  p^* } ) = 1.
 $$
\end{proof}

For $t \in \cor^\times$, $(i,a) \in \sig$ and $f \in  \Hom(\sig, \Z)$, 
we define 
\begin{align} \label{Eq: tau}
 \tau_t( i,a ) \seteq  ( i, ta ) \quad \text{and} \quad  (\tau_t f)( i, a) \seteq f(i,t^{-1}a).
\end{align}

\begin{lemma} \label{Lem: tau_t} \
\bnum
\item For $(i,a) \in \sig$, we have $ \rs_{i,a}  = - \rs_{i^*,a p^* } = - \rs_{i^*,a (p^*)^{-1} } $.
\item For $t\in \cor^\times$ and $(i,a) \in \sig$, we have $ \tau_t (\rs_{i,a}) = \rs_{i,ta} $.
\ee
\end{lemma}
\begin{proof}

(i) follows from $\eqref{eq: dual}$ and Lemma \ref{Lem: prop for E}.

\sn
(ii) For $(j,b)\in \sig$, we have 
 \begin{align*}
(\tau_t (\rs_{i,a})) (j,b) &=  (\rs_{i,a})  (j,t^{-1}b) =  \Li( V(\varpi_i)_a, V(\varpi_j)_{t^{-1}b} )
 =  \Li( V(\varpi_i)_{ta}, V(\varpi_j)_{b} ) \\
& = (\rs_{i,ta})  (j, b),
\end{align*}
where the third equality follows from Proposition \ref{Prop: para shift for La}.
Thus, we have the desired assertion.
\end{proof}

For  $t\in\cor^\times$,   $A \subset \sig$ and $ F \subset  \Hom(\sig, \Z)$, we set 
$$
 A _t \seteq \{ \tau_t(a) \mid a \in A \} \quad \text{ and } \quad   F_t \seteq \{ \tau_t(f) \mid f\in F \}.
$$
We write $\scor$ for  the  stabilizer subgroup of $\sigZ$ with 
respect to the action of $\cor^\times$ on $\sig$ through $\tau_t$, i.e.,
\begin{align*}
\scor \seteq \{ t \in \cor^\times \mid ( \sigZ )_t = \sigZ \}.
\end{align*}

\begin{proposition} \label{Prop: sig and rW}
We have the following.
\bnum
\item \label{Prop: sig and rW: item1} $ \sig = \bigsqcup_{a\in \cor^\times / \scor  }
\bl\sigZ\br_a $.
\item $ \Delta = \bigsqcup_{a\in \cor^\times / \scor  }  (\Delta_0)_a $.
\item \label{Prop: sig and rW: item4}  For $k,k'\in\cor^\times$ such that $k/k'\not\in\scor$, we have
$\bl(\rW_0)_k,(\rW_0)_{k'}\br=0$.
\ee
\end{proposition}
\begin{proof}
\noi 
\ref{Prop: sig and rW: item1} follows from the fact that
any connected component of $\sig$ is a translation of $\sigZ$. 

\sn[1.5ex]
\ref{Prop: sig and rW: item4}\ It is enough to show that,
for $(i,a)\in(\sigZ)_k$ and $(j,b)\in(\sigZ)_{k'}$, we have
$(\rs_{i,a},\rs_{  j,b })=0$.

By the definition of $\sigZ$, 
$V(\varpi_i)_a$ and  $\dual^mV(\varpi_j)_b$ strongly  commute for any $m$, 
which tells us that 
$$
\Li(V(\varpi_i)_a, V(\varpi_j)_b)=0
$$
by Corollary \ref{Cor: comm} and Proposition \ref{Prop: La and d}.

\sn
(ii)\ It is enough to show
$$
\Delta_0\cap (\Delta_0)_k=\emptyset
\qt{for $k\in \cor^\times / \scor$.}
$$
For $(i,a) \in \sigZ$ and $(j,b) \in \sigZ_k$, 
we have
$(\rs_{i,a},\rs_{i,a})= 2$ by Lemma~\ref{Lem: Li = -2}
and $(\rs_{i,a},\rs_{j,b}) = 0$ by \ref{Prop: sig and rW: item4}.
Thus we conclude that  $\rs_{i,a} \ne \rs_{j,b}$.
\end{proof}

We set 
$$
\rES \seteq \R \otimes_\Z \rW \quad \text{ and }\quad \rES_0 \seteq \R \otimes_\Z \rW_0.
$$
Then the pairing $(-,-)$ gives a symmetric bilinear form on $\rES$. 
Theorem \ref{Thm: root system} below is the main theorem of this section whose proof is postponed 
until Section\;\ref{Sec: proofs}.

\begin{theorem} \label{Thm: root system} \
\bnum
\item The pair $(\rES_0, \Delta_0)$ is an irreducible 
simply-laced root system of the following type.

\renewcommand{\arraystretch}{1.5}
\begin{align} \label{Table: root system} \small
\begin{array}{|c||c|c|c|c|c|c|c|} 
\hline
\text{Type of $\g$} & A_n^{(1)}  & B_n^{(1)} & C_n^{(1)} & D_n^{(1)} & A_{2n}^{(2)} & A_{2n-1}^{(2)} & D_{n+1}^{(2)}  \\
&(n\ge1)&(n\ge2)&(n\ge3)&(n\ge4)&(n\ge1)&(n\ge2)&(n\ge3)\\
\hline
\text{Type of $(\rES_0, \Delta_0)$} & A_n & A_{2n-1}    & D_{n+1}   &  D_n & A_{2n} & A_{2n-1} & D_{n+1}  \\
\hline
\hline
\text{Type of $\g$} & E_6^{(1)}  & E_7^{(1)} & E_8^{(1)} & F_4^{(1)} & G_{2}^{(1)} & E_{6}^{(2)} & D_{4}^{(3)}  \\
\hline
\text{Type of $(\rES_0, \Delta_0)$} & E_6 & E_{7}    & E_{8}   & E_6 & D_{4} & E_{6} & D_{4}  \\
\hline
\end{array}
\end{align}
\vskip 0.7em

\item The bilinear form $(-,-)\vert_{\rW_0}$ is positive-definite.
Moreover, it is Weyl group invariant, i.e. $s_\al(\Delta_0)\subset \Delta_0$
for any $\al\in\Delta_0$. Here $s_\al\in\End(\rES_0)$ is the reflection
defined by $s_\al(\la)=\la-(\al,\la)\al$.

\ee
\end{theorem}

The corollary below follows from Proposition \ref{Prop: sig and rW} and Theorem \ref{Thm: root system}.

\begin{corollary} \label{Cor: root system for W}
We have
\bnum
\item $
\rW = \bigoplus_{k \in \cor^\times/ \scor }  (\rW_0)_k$.

\item
As a root system, $( (\rES_0)_k,  (\Delta_0)_k)$ is isomorphic to $(\rES_0, \Delta_0)$ for $k \in \cor^\times/ \scor$, and 
$$
(\rES, \Delta) = \bigsqcup_{k \in \cor^\times/\scor} ( (\rES_0)_k,  (\Delta_0)_k).
$$
\ee
\end{corollary}
\Proof
We know already that
$\rW = \sum_{k \in \cor^\times/ \scor }  (\rW_0)_k$.
Since $(\rW_0)_k$ and $(\rW_0)_{k'}$ are orthogonal if $k/k'\not\in \scor$,
the non-degeneracy of $(-,-)\vert_{\rES}$ implies that
$\rW = \bigoplus_{k \in \cor^\times/ \scor }  (\rW_0)_k$.

\sn
(ii) easily follows from (i) and Theorem \ref{Thm: root system}. 
\QED

The following corollary is an immediate consequence of Theorem~\ref{Thm: root system}.
\Cor
One has 
\bnum
\item
$(\la,\la)\in2\Z_{>0}$ for any $\la\in\rW_0\setminus\st{0}$, and
\item$\Delta_0=\st{\la\in\rW_0\mid (\la,\la)=2}$.
\ee 
\encor
Hence the root system $(\rES_0, \Delta_0)$ is completely determined by the pair
$(\rW_0,\,(-,-)\vert_{\rW_0})$.

\vskip 2em

\section{Block decomposition of $\Ca_\g$} \label{Sec: Block decomposition}

In this section, we give a block decomposition of $\Ca_\g$ parameterized by $\rW$.

\subsection{Blocks} 
We recall the notion of blocks.  
Let $\cat$ be an abelian category such that any object of $\cat$ has finite length.

\begin{definition}\label{def:block}
A \emph{block} $\catB$ of $\cat$ is a full abelian subcategory such that 
\bnum
\item there is a decomposition $\cat = \catB \bigoplus \cat'$ for some full abelian subcategory $\cat'$, \label{it:block1}
\item there is no non-trivial decomposition $\catB = \catB' \bigoplus \catB''$ 
with full abelian subcategories $\catB'$ and $\catB''$. 
\ee
\end{definition}

The following lemma is obvious.
\begin{lemma}
Let $\catB$ be a full subcategory of $\shc$ satisfying condition
\ref{it:block1} in {\rm Definition~\ref{def:block}}.
\bnum
\item
$\catB$ is stable by taking subquotients and extensions,
\item
for simple objects $S,S' \in \cat$
such that $\Ext^1_\cat(S,S') \not\simeq0$,
if one of them belongs to $\catB$ then so does the other.
\ee
\end{lemma}

\begin{lemma} \label{Lem: hom ext for blocks}
Let $X, X' \in \cat$. Suppose that 
$
\Ext^{1}_\cat( S,S' )=0
$
for any simple subquotient $S$ and $S'$ of $X$ and $X'$ respectively.
Then we have 
$\Ext^{1}_\cat( X,X' )=0$. 
\end{lemma}
\begin{proof}
Let $\ell$ and $\ell'$ be the lengths of $X$ and $X'$, respectively. We use an induction on $\ell+\ell'$. 
If $X$ and $X'$ are simple, then it is clear by the assumption.

Suppose that $X'$ is not simple. Then there exists an exact sequence 
$0 \rightarrow M \rightarrow X' \rightarrow N \rightarrow 0$ with a simple $M$.
Then it gives the exact sequence 
\begin{align*}
\Ext^1_\cat(X, M) \rightarrow \Ext^1_\cat(X, X') \rightarrow \Ext^1_\cat(X, N).
\end{align*}
By the induction hypothesis, we have 
$\Ext^1_\cat(X, M) = \Ext^1_\cat(X, N) = 0$, which tells us that  
 $\Ext^{1}_\cat( X,X' )=0$.

The case when  $X$ is not simple can be proved in the same manner.
\end{proof}

\begin{lemma} \label{Lem: decomposition}
Let $\mathfrak{c}$ be the set of the isomorphism classes of simple objects of $\cat$, and let $\mathfrak{c}=\bigsqcup_{a\in A}\mathfrak{c}_a$ be a partition of
$\mathfrak{c}$.
We assume that
\eqn
&&\text{\parbox{78ex}{for $a,a'\in A$ such that $a\not=a'$ and a simple object $S$ \ro resp.\ $S'$\rf
belonging to $\mathfrak{c}_a$ \ro resp.\ $\mathfrak{c}_{a'}$\rf,
one has $\Ext^1_{\cat}(S,S')=0$.}}
\eneqn
 For $a\in A$, let $\cat_a$ be the full subcategory of $\cat$ consisting of objects $X$
such that any simple subquotient of $X$ belongs to $\mathfrak{c}_a$.
Then $\cat=\soplus_{a\in A}\cat_a$.
\end{lemma}
\begin{proof}
It is enough to show that any object $X$ of
$\cat$ has a decomposition
$X\simeq \soplus_{a\in A}X_a$ with $X_a\in\cat_a$.
In order to prove this,
we shall argue by induction on the length of $X$. We may assume that X is non-zero.
 Let us take a subobject $Y$ of $X$ such that $X/Y$ is simple.
Then the induction hypothesis implies that $Y = \soplus_{a\in A}Y_a$ 
with $Y_a\in\cat_a$.

Take $a_0\in A$ such that
$X/Y$ belongs to $ \mathfrak{c}_{a_0} $. Then define $Z\in\cat$
by the following exact sequence
\begin{align} \label{Eq: split ex}
0 \rightarrow \soplus_{a\not=a_0}Y_a \rightarrow X \rightarrow Z \rightarrow 0.
\end{align}
Since we have an exact sequence
$0\to Y_{a_0}\to Z\to X/Y\to 0$,
$Z$ belongs to $\cat_{a_0}$.
Then, Lemma \ref{Lem: hom ext for blocks} says 
$\Ext^1(Z,\soplus_{a\not=a_0}Y_a )=0$.
Hence the exact sequence  $\eqref{Eq: split ex}$ splits, i.e.,
 $X \simeq Z\oplus\soplus_{a\not=a_0}Y_a$.
\end{proof}

Let $\approx$ be the equivalence relation on the set of the isomorphism classes of simple objects of $\cat$ generated by the following relation $\approx'$:
for simple objects $S,S' \in \cat$,
$$
[S] \approx' [S'] \quad \text{ if and only if } \quad \Ext^1_\cat(S,S') \ne 0.
$$

\begin{theorem} \label{Thm: blocks}
Let $A$ be the set of $\approx$-equivalence classes.
For $a\in A$, let $\cat_a$
be the full subcategory of $\cat$ consisting of objects $X$
such that any simple subquotient of $X$ belongs to $a$.
Then, $\cat_a$ is a block, and
the category $\cat$ has a decomposition $\cat = \bigoplus_{a \in A} \cat_a $.
Moreover, any block of $\cat$ is equal to $\cat_a$ for some $a$.
\end{theorem}
\begin{proof}
Lemma \ref{Lem: decomposition} implies the decomposition 
$$
\cat = \bigoplus_{a \in A} \cat_a.
$$
Moreover, since $a$ is a $\approx$-equivalence class, there is no non-trivial decomposition of $\cat_a$ for any $a \in A$. 
\end{proof}

The corollary below follows directly from Theorem \ref{Thm: blocks}.  
\begin{corollary} \label{Cor: indecomposable and block}
Let $X$ be an indecomposable object of $\cat$. Then $X$ belongs to some block. In particular, 
all the simple subquotients of $X$ belong to the same block. 
\end{corollary}

\medskip
\subsection{Direct decomposition of $\Ca_\g$}\label{Sec: direct decomp Cg} \
In this subsection, we shall prove that $\Ca_\g$ has a decomposition parameterized by elements of
$\rW$.

\begin{lemma} \label{Lem: aff hom}
For modules $M,N \in \Ca_\g$, there exists an isomorphism
\begin{align}
 \Psi\col \cor[z^{\pm1}] \otimes \Hom_{U'_q(\g)}(N, \trivial) \otimes \Hom_{U'_q(\g)}( \trivial, M)   
 \buildrel \sim\over \longrightarrow   \Hom_{  U'_q(\g)}(N, M_z)
\label{mor:Psi}
\end{align}
defined by 
$ \Psi(a(z) \otimes f \otimes g ) = a(z) ( g \circ f ) $ for $a(z) \in \cor[z^{\pm1}]$, $f \in \Hom_{U'_q(\g)}(N, \trivial)$ and $g \in \Hom_{U'_q(\g)}(\trivial, M)$.
\end{lemma}

\begin{proof}
Note that $ \cor[z^{\pm1}] \otimes\Hom_{U'_q(\g)}( \trivial, M)\isoto\Hom_{U'_q(\g)}( \trivial, M_z) $.
There is a quotient $N'$
of $N$ which is a direct sum of copies of $\one$
and $\Hom(N',\one)\isoto\Hom(N,\one)$.
Since \eqref{mor:Psi} for $N'$ is obviously an isomorphism,
$\Psi$ is injective.

 In order to prove that $\Psi$ is surjective, we shall decompose 
a given non-zero $f\col N \to M_z$ into $N \to \one^{\oplus \ell} \to M_z$ for some
$\ell \in \Z_{>0}$.  Here $\trivial^{\oplus \ell}$ is the direct sum of $\ell$ copies of the trivial module $\trivial$.
Without loss of generality, we may assume that $f$ is injective.
We set $\wt(N) \seteq \{ \la \in \wlP_\cl \mid N_\la \ne 0  \}$. 

If $\wt(N) =\{0\}$, then $N$ should be isomorphic to $\trivial ^{\oplus \ell}$ for some $\ell \in \Z_{>0}$, which is the desired result.

We suppose that $\wt(N)\not=\{0\}$. We choose a non-zero weight 
$\la \in \wt(N)$.

Note that the $U'_q(\g)$-module structure on $M_z$ extends to a
$U_q(\g)$-module structure and we have a weight 
decomposition $M_z=\soplus_{\mu\in\wlP}(M_z)_\mu$.
Then we have
$$f(N_\la)\subset \soplus_{\mu\in\wlP,\;\cl(\mu)=\la}(M_z)_\mu,$$
where  $\cl\colon \wl\to \wl_\cl$ is the classical projection.
There exist  $w \in \W$ and a non-zero integer $n$ 
such that $w( \mu) = \mu+n\delta$ for any $\mu\in\cl^{-1}(\la)$.
We now consider the braid group action $T_w$ 
defined by $w$  on an integral module (see \cite{Lus90,S94}). Then 
the $\cor$-linear automorphism $T_w$ sends $(M_z)_\mu$ to $(M_z)_{w\mu}$.
The space $f(N_\lambda)$ is invariant under the
automorphism $T_w$, but any  non-zero finite-dimensional subspace of
$\soplus_{\mu\in\wlP,\;\cl(\mu)=\la} (M_z)_{\mu}$
cannot be invariant under $T_w$. This is a contradiction. 
\end{proof}

\begin{proposition} \label{prop:ayx}
For modules $M,N \in \Ca_\g$ and a simple module $L \in \Ca_\g$, we have the following isomorphisms
$$
 \cor[z^{\pm1}] \otimes \Hom_{U_q'(\g)}(M,N) \buildrel \sim\over \longrightarrow 
\Hom_{  \cor[z^{\pm1}] \otimes U_q'(\g)}(M\otimes L_z,N \otimes L_z).
$$
\end{proposition}
\begin{proof}
By Lemma \ref{Lem: aff hom}, we obtain that 
\begin{align*}
 \Hom_{ \cor[z^{\pm1}] \otimes\uqpg  }(M\otimes L_z,N \otimes L_z) & \simeq  \Hom_{\uqpg  }(N^* \otimes M , (L \otimes L^*)_z) \\
 & \simeq \ko[z^{\pm1}]  \otimes    \Hom_{\uqpg}(N^* \otimes M ,\one)  \tens
 \Hom_{\uqpg}(\one, L\otimes L^*)\\
&  \simeq \ko[z^{\pm1}]  \otimes \Hom_{\uqpg}(M ,N).
\end{align*}
\end{proof}

\begin{lemma} \label{Lem: Hom Ext}
Let $M$ and $N$ be simple modules in $\Ca_\g$. If 
$$
\frac{c_{M,L}(z)}{c_{N,L}(z)} \notin\cor(z) \quad\text{ for some simple module }  L \in \Ca_\g,
$$
then we have 
$$
\Ext^1_{U_q'(\g)}(M,N)=0.
$$
\end{lemma}

\begin{proof}
Let $L \in \Ca_\g$ be a simple module such that $ \frac{c_{M,L}(z)}{c_{N,L}(z)} \notin\cor(z) $.

We shall prove that any exact sequence
$$
0 \rightarrow N \rightarrow X \rightarrow M \rightarrow 0
$$
splits.
We set $\widehat{L}_z \seteq \cor((z)) \otimes _{\cor [z^{\pm1}] } L_z$, where $L_z$ is 
the affinization of $L$. 
Then the following diagram commutes:
$$
\xymatrix{
0   \ar[r]& N \otimes \widehat{L}_z  \ar[d]^\bwr_{\Runiv_{N,  \widehat{L}_z}} \ar[r] & X \otimes \widehat{L}_z \ar[d]_{\Runiv_{X,  \widehat{L}_z}}^\bwr \ar[r] & M \otimes \widehat{L}_z \ar[d]_{\Runiv_{M,  \widehat{L}_z}}^\bwr \ar[r] & 0 \\
0  \ar[r]&  \widehat{L}_z  \otimes N  \ar[r] &  \widehat{L}_z \otimes X  \ar[r] &  \widehat{L}_z  \otimes M  \ar[r] & 0 .
}
$$
We set 
\begin{align*}
f(z) \seteq \frac{ c_{M, L}(z) }{ c_{N, L}(z) } \notin \cor(z), \qquad 
R \seteq c_{M, L}(z) \Runiv_{X, \widehat{L}_z}\col X\otimes \widehat{L}_z \rightarrow   \widehat{L}_z \otimes X.
\end{align*}
It follows from 
\begin{align*}
c_{M,L}(z) \Runiv_{M,  \widehat{L}_z}(  M \otimes L_z)  \subset L_z \otimes M, \quad c_{N,L}(z) \Runiv_{N,  \widehat{L}_z}(  N \otimes L_z)  \subset L_z \otimes N
\end{align*}
that 
\begin{align*}
R( X \otimes L_z) \subset L_z \otimes X + \widehat{L}_z \otimes N, \qquad 
R(N \otimes L_z ) \subset f(z) ( L_z \otimes N ).
\end{align*}
Therefore $R$ induces the $\cor[z^{\pm1}]\tens \Up$-linear homomorphism
\eqn
\mathcal{R} \col M \otimes L_z\simeq \dfrac{X\tens L_z}{N\tens L_z}
 &&\longrightarrow \dfrac{ \cor(z) \otimes L_z \otimes X + \widehat{L}_z \otimes N}{ \cor(z) \otimes L_z \otimes X + f(z) \cor(z) \otimes L_z \otimes N }\;.
\eneqn

We set $ \mathcal{P} \seteq \frac{\cor((z))}{\cor(z) + f(z) \cor(z)}$.
Since 
\eqn
\frac{ \cor(z) \otimes L_z \otimes X + \widehat{L}_z \otimes N}{ \cor(z) \otimes L_z \otimes X + f(z) \cor(z) \otimes L_z \otimes N }
&&\simeq
\dfrac{\widehat{L}_z \otimes N}
{ \cor(z) \otimes L_z \otimes N + f(z) \cor(z) \otimes L_z \otimes N }\\
&& \simeq \mathcal{P} \otimes_{\cor[z^{\pm1}]} L_z \otimes N, 
\eneqn
we have the homomorphism  of $\cor[z^{\pm1}] \tens\Up$-modules 
$$
\mathcal{R} \col M \otimes L_z \longrightarrow \mathcal{P} \otimes_{\cor[z^{\pm1}]} L_z \otimes N.
$$

Let us show that $\mathcal{R}$ vanishes.

Assume that $\mathcal{R}\not=0$.
Then,
$\Hom_{\corz\tens\Up}(M\tens L_z,\mathcal{P} \otimes_{\cor[z^{\pm1}]} L_z \otimes N)
\simeq 
\mathcal{P} \otimes_{\cor[z^{\pm1}]} \Hom_{\corz\tens\Up}(M\tens L_z,L_z \otimes N)$
implies that
$\Hom_{\corz\tens\Up}(M\tens L_z,L_z \otimes N)\not\simeq0$.

Since $\kz\tens_{\corz}(M\tens L_z)$ and $\kz\tens_{\corz}(L_z \otimes N)$ 
are simple $\kz\tens\Up$-module,
they are isomorphic.
Since
$\kz\tens_{\corz}(L_z \otimes N)$ and $\kz\tens_{\corz}(N\tens L_z)$ are isomorphic,
we conclude that
$\kz\tens_{\corz}(M\tens L_z)\simeq\kz\tens_{\corz}(N\tens L_z)$.
On the other hand, Proposition~\ref{prop:ayx} implies that 
$$ \kz \otimes \Hom_{U_q'(\g)}(M,N) \isoto
\Hom_{\kz\otimes U_q'(\g)}(\kz \otimes_{\cor[z^{\pm1}]}M\otimes L_z, \kz \otimes_{\cor[z^{\pm1}]} N \otimes L_z).$$
Hence $\Hom_{U_q'(\g)}(M,N) \not=0$, and
we obtain that $M$ and $N$ are isomorphic, which is a contradiction.
Hence $\mathcal{R}=0$, which means that
\eqn
&&R\bl\kz\tens (X\tens L_z)\br\subset
  \cor(z) \otimes L_z \otimes X + f(z) \cor(z) \otimes L_z \otimes N.
\eneqn

Let us consider the composition
$$\Phi\col K\seteq R\bl\kz\tens (X\tens L_z)\br\cap\bl\cor(z) \otimes L_z \otimes X \br
\monoto \cor(z) \otimes L_z \otimes X\epito \cor(z) \otimes L_z \otimes M.$$
We have
$$R\bl\kz\tens (X\tens L_z)\br\cap \widehat{L}_z \tens N
=R\bl \kz\tens (N\tens L_z)\br=f(z)\kz\tens L_z\tens N.$$
Hence $\ker(\Phi)=K\cap \bl\cor(z) \otimes L_z \otimes N\br
=\bl f(z)\cor(z) \otimes L_z \otimes N)\cap \bl\cor(z) \otimes L_z \otimes N\br$
vanishes,  which means that $\Phi$ is a monomorphism. 

Since $\cor(z) \otimes L_z\tens M$ and
$\cor(z) \otimes L_z\tens N$ are simple $\kz\tens\Up$-modules, 
$\cor(z) \otimes L_z \otimes X$ has length $2$.
Similarly $R\bl\kz\tens (X\tens L_z)\br$ has also length $2$.
On the other hand,
$\cor(z) \otimes L_z \otimes X + f(z) \cor(z) \otimes L_z \otimes N$ has length $\le3$, which implies that $K$ does not vanish.
Hence $\Phi$ is an isomorphism.
Thus we conclude that
the homomorphism
\eqn
\Hom(\kz\tens L_z\tens M,\kz\tens L_z\tens X)\to&&
\Hom(\kz\tens L_z\tens M,\kz\tens L_z\tens M)\\*
&&=\kz\id_{\kz\tens L_z\tens M}\eneqn
is surjective.
Then, Proposition~\ref{prop:ayx} implies that this homomorphism
is isomorphic to
$$\kz\tens \Hom(M,X)\epito \kz\tens \Hom(M,M).$$ 
Thus we conclude that
$\Hom(M,X)\to \Hom(M,M)$ is surjective,
that is,
$$0\To N\To X\To M\To0$$
splits.
\end{proof}

For $\al \in \rW$, let $\Ca_{\g, \al}$ be the full subcategory of $\Ca_\g$ consisting of objects $X$ such that
$\rE(S)=\al$ for any simple subquotient $S$ of $X$.

\begin{theorem} \label{Thm: decomposition for Ca}
There exist the following decompositions
\begin{align*}
\Ca_\g = \bigoplus_{\al \in \rW} \Ca_{\g, \al} \quad \text{and} \quad \Ca_\g^0 = \bigoplus_{\al \in \rW_0} \Ca_{\g, \al}.
\end{align*}
\end{theorem}
\begin{proof}

Let $\alpha, \beta \in \rW$ with $\alpha \ne \beta$. For simple modules $M \in \Ca_{\g, \al}$ and $N \in \Ca_{\g, \beta}$, Lemma \ref{Lem: prop for E} \ref{Lem: prop for E: item3} says that  
$ \displaystyle \frac{a_{M,V(\varpi_i)}(z) }{a_{N,V(\varpi_i)}(z) } \notin \cor(z)  $ for some $i\in I_0$. 
Hence Lemma \ref{Lem: Hom Ext} implies that
$
\Ext^1_{U_q'(\g)}(M,N)=0.
$
The desired result then follows from Lemma~\ref{Lem: decomposition}.
\end{proof}

\subsection{The block $\Ca_{\g, \al}$}\

 Recall the automorphism  $\tau_t$ on $\sig$ defined in $\eqref{Eq: tau}$.
 For $(i,a) \in \sig$, we write 
 $$
 V(i,a) \seteq V(\varpi_i)_a.
 $$ 
 Note that $ V(\tau_t \al) = V(\al)_t$ for $\al \in \sig$ and $t\in \cor^\times$.
 For $\al \in \sig$, we define $\al^* \in \sig$ by 
 $$
 V(\al^*) \simeq V(\al)^*.
 $$
 Thus we have 
 $$
 \al^{**} = \tau_{\tilde{p}^{-1}} (\alpha) \qquad \text{ for } \al \in \sig.
 $$

\begin{lemma} \label{Lem: block for tensor}
Let $\al_1, \ldots, \al_k \in \sig$ for $k\in \Z_{>0}$. Then all the simple subquotients of $V(\al_1) \otimes V(\al_2) \otimes \cdots \otimes V(\al_k)  $ are 
contained in the same block of $\Ca_\g$.
\end{lemma}
\begin{proof}
There exists a permutation $\sigma \in \sym_k$ such that 
the tensor product 
$V(\al_{\sigma(1)}) \otimes V(\al_{\sigma(2)}) \otimes \cdots \otimes V(\al_{\sigma(k)})  $ has a simple head   by Theorem \ref{Thm: basic properties}, 
and hence it is indecomposable. Thus, all the simple subquotients of $V(\al_{\sigma(1)}) \otimes V(\al_{\sigma(2)}) \otimes \cdots \otimes V(\al_{\sigma(k)})  $  
are contained in the same block by Corollary \ref{Cor: indecomposable and block}.
Since any simple subquotient of $V(\al_1) \otimes V(\al_2) \otimes \cdots \otimes V(\al_k)  $ is isomorphic to some
simple subquotient of $V(\al_{\sigma(1)}) \otimes V(\al_{\sigma(2)}) \otimes \cdots \otimes V(\al_{\sigma(k)}) $,
we obtain the desired result.
\end{proof}

We set 
$$
\lP \seteq \bigoplus_{\al \in \sig} \Z \lse_\al \quad \text{ and } \quad \lP_0 \seteq \bigoplus_{\al \in \sigZ} \Z \lse_\al
$$
and
$$
\lP^+ \seteq \sum_{\al \in \sig} \Z_{\ge0} \lse_\al\subset\lP,
$$ 
where $\lse_\al$ is a symbol. Define a group homomorphism 
$$
 p\col \lP \twoheadrightarrow \rW, \quad  \lse_{(i,a)} \mapsto \rs_{i,a}, 
$$
and set 
$$
 p_0 \seteq p|_{\lP_0}\col \lP_0 \twoheadrightarrow \rW_0. 
 $$

By Proposition \ref{Prop: sig and rW}, we have 
\begin{align} \label{Eq: lP and lQ1}
\lP = \bigoplus_{k \in \cor^\times / \scor}  (\lP_0)_k.
\end{align}

Let $\lQ_0$ be the subgroup of 
$\lP_0$ generated by the elements of the form $ \sum_{k=1}^m  \lse_{\al_k}$ 
($\al_k \in \sigZ$)
such that the trivial module $\trivial$ appears in $V(\al_1) \otimes V(\al_2) \otimes \cdots \otimes V(\al_{ m })  $
as a simple subquotient.  We then have $p_0 ( \lQ_0 ) = 0 $.

We set  

\begin{align} \label{Eq: lP and lQ2}
\lQ \seteq \bigoplus_{k \in \cor^\times / \scor}  (\lQ_0)_k.
\end{align}

Recall $\phi_Q\col \Df \isoto\sigQ$ in \eqref{Eq: phi}.
Let $\Pf\subset\Df$ be 
the set of simple roots of the positive root system $\Df$
and $\rlq$ the corresponding root lattice. 
Hence we have $\Pf\subset\Df\subset\rlq$.

In the proof of the following lemma, 
we do not use  Theorem~\ref{Thm: root system}.

\Lemma\label{lem:Pi}
For $\al\in\sigZ$, let us denote 
by $\bse_\al\in \lP_0/\lQ_0$ the image of
$\lse_\al$ by the projection $\lP_0\to\lP_0/\lQ_0$.
\bnum
\item
The map $\Df\ni\al\mapsto  \bse_{\phi_Q(\al)}\in\lP_0/\lQ_0$ extends to an additive map
$\psi'_Q\col\rlq\to\lP_0/\lQ_0$.
\item $\psi'_Q$ is surjective, i.e., we have
$$\lP_0/\lQ_0=\sum_{\beta\in\Pf}\Z\bse_{\phi_Q(\beta)}.$$
\item Let $\psi_Q\col\rlq\to\rW_0$ be the composition
$\rlq\To[{\psi'_Q}]\lP_0/\lQ_0\To\rW_0$.
Then we have
$$\psi_Q(\beta)=\rE\bl V_Q(\beta)\br.$$
\item $\psi_Q$ is surjective, i.e., we have 
$\rW_0=\sum_{\al\in\phi_Q(\Pf)}\Z   p_0( \lse_\al ) $.
\ee
\enlemma
\Proof

(i)\ 
The map $\Pf\ni \al\mapsto \bse_{\phi_Q(\al)}\in \lP_0/\lQ_0$
extends to a linear map $\psi'_Q\col\rlq\to\lP_0/\lQ_0$.
It is enough to show that
$\bse_{\phi_Q(\gamma)}=\psi'_Q(\gamma)$ for any $\gamma\in\Df$.
Let us show it by  induction on the length of $\gamma$.
If $\gamma$ is not a simple root,
take a minimal pair $(\beta,\beta')$ of $\gamma$
(see Proposition \ref{Prop: CQ categorification}).
Since $ V_Q(\gamma) $ appears as a composition factor of $ V_Q(\be) \otimes  V_Q(\be') $ by Proposition \ref{Prop: CQ categorification}, we have

$$\bse_{\phi_Q(\gamma)}=\bse_{\phi_Q(\be)}+\bse_{\phi_Q(\be')}=\psi'_Q(\beta)+\psi'_Q(\beta')
=\psi'_Q(\gamma).$$

\sn
(ii) follows from (i), and (iii) follows from
(ii) and a surjective map
$\lP_0/\lQ_0\epito\rW_0$.
\QED

In the proof of the following lemma, we use 
the fact that the rank of $\rW_0$ is at least the rank of $\Df$
(stated in Theorem~\ref{Thm: root system} whose proof is postponed 
 to  \S\,\ref{Sec: proofs} \eqref{eq:atleast}). 

\begin{lemma} \label{Lem: Q ker p} 
We have isomorphisms
\eqn&&
\lP_0/\lQ_0\isoto\rW_0\qtq \lP/\lQ\isoto\rW.
\eneqn
\end{lemma}
\begin{proof}
The second isomorphism easily follows from the first isomorphism and
\eqref{Eq: lP and lQ1}, \eqref{Eq: lP and lQ2}.
Hence let us only show that $\lP_0/\lQ_0\epito\rW_0$ is an isomorphism.

Let $r$ be the rank of $\Df$.
By \eqref{eq:atleast}, the rank of $\rW_0$
is at least $r$.
Let us consider a surjective homomorphism
\eq
&&\lP_0/\lQ_0\epito\rW_0.
\label{mor:PQW}
\eneq
By Lemma \ref{lem:Pi}, 
$\lP_0/\lQ_0$ is generated by $r$ elements.
Hence \eqref{mor:PQW} is an isomorphism.

\end{proof}

For $\la = \sum_{t=1}^k \lse_{\al_t} \in \lP^+$, we set 
$$
\overline{V}(\la) \seteq[ V(\alpha_1) \otimes V(\al_2) \otimes \cdots \otimes V(\al_k) ] \in K(\Ca_\g).
$$
Note that, for $\la, \mu \in \lP^+ $, if $\trivial $ appears in $\overline{V}(\la)$ and $\overline{V}(\mu)$, then 
$\trivial$ also appears in $\overline{V}(\la) \otimes \overline{V}(\mu)$.
Hence any element of $ \lQ $ can be written as $\la - \mu$ with $\la, \mu \in \lP^+$ such that $\one$ appears in both $\overline{V}(\la)$ and 
$\overline{V}(\mu)$.

\begin{theorem}  \label{Thm: block for Ca_al}
For any $\al \in \rW$, the subcategory $ \Ca_{\g, \al}$ is a block of $\Ca_\g$.
\end{theorem}
\begin{proof}

Let $\al \in \rW$ and let $S, S'$ be simple modules in $\Ca_{\g, \al}$. We shall show that $S$ and $S'$ belong to the same block.

Thanks to Theorem \ref{Thm: basic properties} \ref{Thm: bp5}, there exist $\la, \la' \in \lP^+$ such that $S$ appears in $\overline{V}(\la)$ and 
$S' $ appears in $\overline{V}(\la')$. 
By Lemma \ref{Lem: Q ker p}, we have $\la - \la' \in \ker p =  \lQ$. Thus, there exist $\mu, \mu' \in \lP^+$ such that 
\begin{itemize}
\item $\la - \la' = \mu' - \mu$,
\item $\trivial $ appears in $\overline{V}(\mu)$ and $\overline{V}(\mu')$.
\end{itemize}
Thus we have 
\bna
\item $ \la + \mu = \la' + \mu' $, i.e., $ \overline{V}(\la+\mu) = \overline{V}(\la'+\mu') $, 
\item  $S$ appears in $\overline{V}(\la) \otimes \overline{V}(\mu) = \overline{V}(\la+\mu) $, 
\item $S'$ appears in $ \overline{V}(\la') \otimes \overline{V}(\mu') = \overline{V}(\la'+\mu')$,
\ee
which tells us that $S$ and $S'$ belong to the same block by Lemma \ref{Lem: block for tensor}.
\end{proof}

Combining Theorem \ref{Thm: decomposition for Ca} with Theorem \ref{Thm: block for Ca_al}, 
we have the following block decomposition.

\begin{corollary}
There exist the following block decompositions
\begin{align*}
\Ca_\g = \bigoplus_{\beta \in \rW} \Ca_{\g, \beta} \quad \text{and} \quad \Ca_\g^0 = \bigoplus_{\beta \in \rW_0} \Ca_{\g, \beta}.
\end{align*}
\end{corollary}

\begin{remark} \label{Rmk: CM EM}
Lemma \ref{Lem: Q ker p} gives a group presentation of $\rW$ which parameterizes
the block decomposition of $\Ca_\g$.
When $\g$ is of untwisted type, the block decomposition of $\Ca_\g$ was given in \cite{CM05} and \cite{EM03}.
Considering the results of \cite{CM05} and \cite{EM03} in our setting,
their results give another group presentation of $\rW$.
Let us explain more precisely in our setting.

Suppose that $\g$ is of untwisted type. We define 
$$
\lP_S \seteq \bigoplus_{(i,a)\in \sig,\ i\in S} \Z \lse_{(i,a)},
$$
where 
$$ 
S = 
\begin{cases}
 \{1\} & \text{if $\g$ is of type $A_n^{(1)}$, $C_n^{(1)}$ or $E_6^{(1)}$,}   \\
\{n\} & \text{ if $\g$ is of type $B_n^{(1)}$ or type $D_n^{(1)}$ ($n$ odd)},  \\ 
\{ n-1, n \} & \text{ if $\g$ is of type $D_n^{(1)}$ ($n$ even)}, 
\end{cases}
$$
and $S$ is $\{2\}$, $\{4\}$, $\{7\}$ and $\{8\}$ if $\g$ is of type $G_2^{(1)}$, $F_4^{(1)}$, $E_7^{(1)}$ 
and $E_8^{(1)}$ respectively.

One can show that $ p( \lP_S ) = \rW$. Thus we have the surjective homomorphism
$$
p_S \seteq p\mid_{ \lP_S}\col \lP_S \twoheadrightarrow \rW.
$$
Then the results given in \cite[Proposition 4.1, Appendix A]{CM05} and \cite[Lemma 4.6, Section 6]{EM03} 
explain that 
the kernel $\ker (p_S)$ is generated by the subset $G$ described 
below: 
\bna
\item if $\g$ is of type $A_n^{(1)}$, then $G = \{  \sum_{k=0}^n \lse_{(1, tq^{2k})} \mid   t \in \cor^\times \}$,
\item if $\g$ is of type $B_n^{(1)}$, then $G = \{ \lse_{(n,t)} + \lse_{(n,tq^{2n-1})} \mid   t \in \cor^\times \}$,
\item if $\g$ is of type $C_n^{(1)}$, then $G = \{   \lse_{(1,t)} + \lse_{(1,tq^{ n+1})} \mid   t \in \cor^\times \}$,
\item if $\g$ is of type $D_n^{(1)}$ and $n$ is odd, then $G = \{   \lse_{n,t} + \lse_{n,tq^2} + \lse_{n,tq^{2n-2}} + \lse_{n,tq^{2n}} \mid   t \in \cor^\times \}$,
\item if $\g$ is of type $D_n^{(1)}$ and $n$ is even, then $G = \{   \lse_{(n-1,t)} + \lse_{(n-1,tq^2)} + \lse_{(n,tq^{2n-2})} + \lse_{(n,tq^{2n})}, \ \lse_{(n-1,t)} + \lse_{(n-1,tq^{2n-2})}, \   \lse_{(n ,t)} + \lse_{(n ,tq^{2n-2})}
  \mid   t \in \cor^\times \}$,
\item if $\g$ is of type $E_6^{(1)}$, then $G = \{  \lse_{(1,t)} + \lse_{(1,tq^{ 8})} + \lse_{(1,tq^{ 16})} , \  \lse_{(1,t)} + \lse_{(1,tq^{ 2})} + \lse_{(1,tq^{ 4})} + \lse_{(1,tq^{ 12})} +  \lse_{(1,tq^{ 14})} +  \lse_{(1,tq^{ 16})}    \mid   t \in \cor^\times \}$,
\item if $\g$ is of type $E_7^{(1)}$, then $G = \{   \lse_{(7,t)} + \lse_{(7,tq^{ 18})}, \  \lse_{(7,t)} + \lse_{(7,tq^{ 2})} + \lse_{(7,tq^{ 12})} + \lse_{(7,tq^{ 14})} +   \lse_{(7,tq^{ 24})} + \lse_{(7,tq^{ 26})}  \mid   t \in \cor^\times \}$,
\item if $\g$ is of type $E_8^{(1)}$, then $G = \{ \lse_{(8,t)} + \lse_{(8,tq^{ 30})}, \ \lse_{(8,t)} + \lse_{(8,tq^{ 20})} + \lse_{(8,tq^{ 40})} ,\  \lse_{(8,t)} + \lse_{(8,tq^{ 12})} + \lse_{(8,tq^{ 24})} +  \lse_{(8,tq^{ 36})} +  \lse_{(8,tq^{ 48})}   \mid   t \in \cor^\times \}$,
\item if $\g$ is of type $F_4^{(1)}$, then $G = \{ \lse_{(4,t)} + \lse_{(4,tq^{ 9})}   ,\   \lse_{(4,t)} + \lse_{(4,tq^{6})} + \lse_{(4,tq^{ 12})}   \mid   t \in \cor^\times \}$,
\item if $\g$ is of type $G_2^{(1)}$, then $G = \{\lse_{(2,t)} + \lse_{(2,tq^{ 4})}   ,\   \lse_{(2,t)} + \lse_{(2,{t(-q_t)}^{8})} + \lse_{(2,{t(-q_t)}^{ 16}})   \mid   t \in \cor^\times \}$.
\ee
We remark that there are typos in the descriptions for type $ E_8$ and $F_4$ in \cite[Appendix A]{CM05}.
\end{remark}

\vskip 2em 

\section{Proof of Theorem \ref{Thm: root system}} \label{Sec: proofs}

\subsection{Strategy of the proof}

\ 
We now start to prove Theorem \ref{Thm: root system}.
We shall use the same notations given in Section \ref{Sec: HL cat} and Section $\ref{Sec: CaQ}$. Recall the explicit descriptions for $\sigZ$ and $\sigQ$.
Let $\Pf=\st{\al_i}_{i\in \If}$ be the 
set of simple roots of
$\Df$, and let $\rlq$ be the root lattice of $\gf$.
Hence we have $$\Pf \subset\Df\subset\rlq.$$
Then by Lemma~\ref{lem:Pi}, we have
\eq
  \rW_0 = \sum_{i\in \If}\Z \rs_{\phi_Q(\al_i)}.
\label{eq:WQ}
\eneq
where $\phi_Q\col \Df\isoto\sigQ$ is the bijection given in $\eqref{Eq: phi}$.

Let $\MQ \seteq ( \mQ{i,j}  )_{i,j\in \If }$ be the square matrix given by 
$$
\mQ{i,j} \seteq ( \rs_{\phi_Q(\alpha_i)}, \rs_{\phi_Q(\al_j)} ).
$$
Thanks to Lemma \ref{Lem: Li = -2}, we know that 
$$
\mQ{i,i} = 2 \qquad \text{ for any } i\in \If. 
$$
To prove Theorem  \ref{Thm: root system}, it suffices to show that $\MQ$ is the Cartan matrix of the finite simple Lie algebra $\gf$,  i.e.,
\eq
( \rs_{\phi_Q(\alpha_i)}, \rs_{\phi_Q(\al_j)} )=(\al_i,\al_j).
\label{eq:isometry}\eneq

Indeed,  \eqref{eq:isometry} implies the following lemma, and 
 Theorem~\ref{Thm: root system} is its immediate consequence.

\Lemma Assume \eqref{eq:isometry}. Then the map
$\Df\ni\beta\mapsto \rE\bl V_Q(\beta)\br\in\Delta_0\subset\rW_0$
extends uniquely to an additive isomorphism
$$\psi_{Q}\col\rlq\isoto\rW_0.$$
Moreover, it preserves the inner products of $\rlq$ and $\rW_0$.
\enlemma
\Proof
Since the Cartan matrix is a symmetric positive-definite matrix,
$\st{\rs_{\phi_Q(\alpha_i)}}_{i\in \If}$ is linearly independent.
Hence we obtain
\eq
\text{the rank of $\rW_0$ is at least the rank $r$ of $\gf$.}
\label{eq:atleast}
\eneq 
On the other hand, 
Lemma~\ref{lem:Pi} implies that 
$\psi_Q\col\rlq\to\rW_0$ is surjective.
Hence, $\psi_Q$ is an isomorphism.
Moreover, \eqref{eq:isometry} shows that $\psi_Q$ preserves the
inner products of $\rlq$ and $\rW_0$.
The other assertions then  easily follow.  
\QED

\subsection{Calculation of the inner products} \label{Sec: calculation of (,)}\ 

In this subsection, we shall give a type-by-type proof of \eqref{eq:isometry}.

\begin{lemma} \label{Lem: Li for ADE}
Suppose that $ \g$ is of affine $ADE$ type.  Let $i,j \in I_0$.
\bnum
\item For $t\in \Z$, we have 
$$
\de( V(\varpi_i),  V(\varpi_{j})_{ (-q)^t} ) = \delta( 2 \le |t| \le h ) \tcmc_{i,j}(|t|-1),
$$
where $h$ is the Coxeter number of $\g$ and
$ \tcmc_{i,j}(k)$ is the integer defined in  $\eqref{Eq: def of tcmt}$ in Appendix $\ref{App: denominators}$.

\item If $0 < t < 2h$, then we have 
$$
\Li(V(\varpi_i), V(\varpi_j)_{(-q)^{t}} )   = \tcmc_{i,j}(t-1)-\tcmc_{i,j}(t+1),
$$
and $\Li(V(\varpi_i), V(\varpi_j)) = -2 \delta_{i,j}$.

\ee

\end{lemma}

\begin{proof}

(i)
For $i,j \in I$,  we write $d_{i,j}(z) \seteq d_{V(\varpi_i), V(\varpi_j)} (z)$. 
Combining Proposition \ref{Prop: d and d} with the denominator formula
$$
d_{i,j}(z) = \prod_{k=1}^{h-1}  (z - (-q)^{k+1})^{\tcmc_{i,j}(k)}
$$ 
written in $\eqref{Eq: denominators for ADE}$,  we compute 
\begin{align*}
\de( V(\varpi_i),  V(\varpi_{j})_{ (-q)^t} ) & = \delta( 2 \le t \le h ) \tcmc_{i,j}(t-1) + \delta( 2 \le -t \le h ) \tcmc_{i,j}(-t-1) \\
&= \delta( 2 \le |t| \le h ) \tcmc_{i,j}(|t|-1).
\end{align*}

(ii)
For $a \in \Z$, let $ [a]\seteq\prod_{n=0}^\infty (1 - (-q)^{a} \tp^n z) $. 
Combining the equation \cite[(A.13)]{AK}  with the denominator formula \eqref{Eq: denominators for ADE}, we have
\begin{align*}
a_{i,j} ((-q)^t z) &=\prod_{1\le k \le h-1} \dfrac{([h+k+1+t]^{\tcmc_{j,i^*}(k)}) ( [h-k-1+t]^{\tcmc_{j,i^*}(k)} )}
{ ([k+1+t]^{\tcmc_{i,j}(k)}) ([2h-k-1+t]^{\tcmc_{i,j}(k)}) }  \\
&=\prod_{1\le k \le h-1} \dfrac{([h+k+1+t]^{-\tcmc_{i,j}(h+k)}) ( [h-k-1+t]^{-\tcmc_{i,j}(h+k)} )}
{ ([k+1+t]^{\tcmc_{i,j}(k)}) ([2h-k-1+t]^{\tcmc_{i,j}(k)}) } \\
&=\prod_{1\le k \le 2h-1} \dfrac{1}{ ([k+1+t]^{\tcmc_{i,j}(k)}) ([2h-k-1+t]^{\tcmc_{i,j}(k)}) }. 
\end{align*}
for any $t\in \Z$,
up to a constant multiple. 
For the second equality, we used 
\eqn 
\tcmc_{i,j}(h+k) = -\tcmc_{i,j}(h-k)= -\tcmc_{j^*,i}(k) \quad \text{for} \ 1\le k \le h-1
\eneqn
which come from \cite[Lemma 3.7 (4), (5)]{Fu19}.
Hence we have
\begin{align*}
& \Li(V(\varpi_i), V(\varpi_j)_{(-q)^{t}} )
=- \Deg^{\infty} (a_{i,j} ((-q)^t z) )\\
&=\quad   \sum_{1\le k \le 2h-1} \left(\ \tcmc_{i,j}(k)(\delta(k+1+t\equiv 0 \mod 2h) + \delta(2h-k-1+t\equiv 0 \mod 2h)  ) \right) \\
&= \tcmc_{ij}(2h-t-1) +\tcmc_{ij}(t-1) \\ 
&= - \tcmc_{ij}( t+1) + \tcmc_{ij}(t-1) 
\end{align*}
for $1\le t\le 2h-1$.
If $t=0$, then we have  
$$
 \Li(V(\varpi_i), V(\varpi_j)  ) =  2 \tcmc_{i,j}(2h-1) =- 2\tcmc_{i,j}(1) =
- 2\delta_{i,j},
$$
as desired.

\end{proof}

\vskip 0.5 em

\sn
\textbf{(Type $A_n^{(1)}$)} \ 
If $n=1$, then it is obvious that $\MQ$ is a Cartan matrix. Thus, we may assume that $n\ge 2$.  
Recall the explicit description of $\sigQ$ for type $A_n^{(1)}$. Note that the Dynkin quiver corresponding to $\sigQ$ is given in $\eqref{Eq: DQ}$. In this case, $h=n+1$ and
 $$
 \phi_Q(\alpha_i) = (1, (-q)^{ 2 -2i}) \in \sigQ \quad \text{ for } i\in \If=
\st{1,\ldots,n}
$$ 
by \cite[Lemma 3.2.3]{KKK15B}. 
For example, if it is of type $A_4^{(1)}$, then elements $(i,(-q)^{k})$ of $ \sigQ $ with the values of $ \phi_Q^{-1} $ can be drawn as follows:
$$ 
\scalebox{0.8}{\xymatrix@C=1ex@R=  0.0ex{ 
 i \backslash k     \ar@{-}[ddddd]<3ex> \ar@{-}[rrrrrrrr]<-2.0ex>   & -6 & -5 & -4 & -3 & -2 & -1 & 0 &  & &   \\
1    & \underline{(0001)} & & \underline{(0010)} & & \underline{(0100)} & & \underline{(1000)} & & & \\
2    & & (0011)& & (0110)& & (1100)& &  &  \\
3     & & & (0111) & &(1110) & &   & \\
4    & & &  & (1111)& &  && & \\
& 
}}
$$
Here, $(a_1,a_2,a_3,a_4) \seteq \sum_{k=1}^4 a_k\al_k \in \Df$ is placed at the position $\phi_Q( a_1, a_2, a_3, a_4 )$, and the underlined ones are simple roots.
Using the formula given in Appendix  \ref{App: ADE}, one can compute that 
 $\tcmc_{1,1}(2k)=0$ and 
$$\tcmc_{1,1}(2k+1) = (\tau^k \al_1,\varpi_1) = (\al_{k+1},\varpi_1)=\delta_{k,0} 
 \quad  \text{ for } 0 \le k < n.  
$$ 

Lemma \ref{Lem: Li for ADE} implies that 
$$
 \Li(V(\varpi_1), V(\varpi_1)_{(-q)^{2k}} ) = \delta_{k,1} \quad \text{ for $k\in \Z$ with $ 1 \le k \le n-1 $.}
$$
Therefore, for $i > j$, we have 
 $$
\mQ{i,j} = -  \Li(V(\varpi_1), V(\varpi_1)_{(-q)^{2(i-j)}} ) = - \delta_{i-j, 1},
$$
which tells us that $ \MQ$ is a Cartan matrix of type $A_n$.

\sn
\textbf{(Type $B_n^{(1)}$)} \ 
Recall the explicit description of $\sigQ$ for type $B_n^{(1)}$ ($n \ge 2$), which can be obtained from \cite{KO18}. 
Note that the Dynkin diagram of $B_2^{(1)}$ is given in $\eqref{Eq: DD}$.
In this case, $\gf$ is of type $A_{2n-1}$ and, for 
 $i\in \If=\st{1,\ldots, 2n-1}$  we have 
 $$
 \phi_Q(\alpha_i) = 
\begin{cases}
 (1, (-1)^{n+1} q_s^{2n+1-4i}) & \text{ if } 1 \le i \le n-1, \\
  (n, q^{ -2n+2 }) & \text{ if } i=n, \\
  (n, q^{ -2n+3 }) & \text{ if } i=n+1, \\
  (1, (-1)^{n+1}   q_s ^{-6n+4i-1}) & \text{ if } n+2 \le i \le 2n-1.  
\end{cases}
$$ 
For example, if it is of type $B_3^{(1)}$, then elements of $ \sigQ $ with the values of $ \phi_Q^{-1} $ can be drawn as follows: 
\begin{align*} 
& \scalebox{0.7}{\xymatrix@C=0.0ex@R=-0.5ex{ 
 i \backslash k     \ar@{-}[ddddd]<1.8ex> \ar@{-}[rrrrrrrrrrrrr]<-1.3ex>   & -8 & -7 & -6 & -5 & -4 & -3 & -2 & -1 &0  &1 & 2&3  &\\
1    &  & & & (00111) &  & (11110) &  & \underline{(01000)} & & \underline{(00001)} & & \underline{(10000)} & &:  \text{ \large $(-1)^{i+3} q_s ^k$} \\
2    & & (00110) & & (01110) & & (01111) & & (11111) & & (11000)   \\
  \ar@{.}[rrrrrrrrrrrrr]<0ex>     &   &   &   &   &   &   &   &  & & &&& \\
3    & \underline{(00100)} & & \underline{(00010)} & & (01100) & & (00011) & & (11100) &   & && & :  \text{ \large $q_s^k$} \\
& 
}} 
\end{align*}
Here we set $  (a_1,a_2,a_3,a_4,a_5)  \seteq \sum_{k=1}^5 a_k\al_k \in \Df$ and the underlined ones are simple roots. 
Combining Proposition \ref{Prop: La and d} and Proposition \ref{Prop: d and d}
with the denominator formula given in Appendix \ref{App: denominators}, we compute that 
$ \de (V_Q(\al_i), \D^k V_Q(\al_j) ) = 0 $ for $ i\ne j, k \ne0$ and 
\begin{align*}
\Li(V(\varpi_1), V(\varpi_1)_{q^k}) &= \de ( V(\varpi_1), V(\varpi_1)_{q^k} ) \\ 
& = \delta_{k,2} \qquad \text{ for } k=1, 2, \ldots, 2n-4,\\
\Li(V(\varpi_n), V(\varpi_1)_{(-1)^{n+1}q_s^t}) &=  \de ( V(\varpi_n), V(\varpi_1)_{(-1)^{n+1}q_s^t} ) \\
& = \delta_{t, 2n+1} \qquad \text{ for } t=2n-1, 2n+1 , \ldots, 6n-7,\\
\Li(V(\varpi_n), V(\varpi_n)_{q}) &= 1.  
\end{align*}

Therefore, for $i > j$, we obtain
 $$
\mQ{i,j} =  - \delta_{i-j, 1},
$$
which tells us that $ \MQ$ is a Cartan matrix of type $A_{2n-1}$.

\sn
\textbf{(Type $C_n^{(1)}$)} \ 
Recall the explicit description of $\sigQ$ for type $C_n^{(1)}$ ($n \ge 3$), which can be obtained from \cite{KO18}. 
In this case, $\gf$ is of type $D_{n+1}$ and, for $1 \le i \le n+1$ we have 
 $$
 \phi_Q (\alpha_i) = 
\begin{cases}
 (1, (-q_s)^{2-2i}) & \text{ if } 1 \le i \le n, \\
  (n, (-q_s)^{-3n+1}) & \text{ if } i=n+1.  
\end{cases}
$$ 
For example, if it is of type $C_4^{(1)}$, then elements of $ \sigQ $ with the values of $ \phi_Q^{-1} $ can be drawn as follows: 
$$ 
\scalebox{0.65}{\xymatrix@C=-1ex@R=-0.5ex{ 
 i \backslash k     \ar@{-}[ddddd]<1.5ex> \ar@{-}[rrrrrrrrrrrrr]<-1.3ex>  &-11 &-10& -9 & -8 & -7 & -6 & -5 & -4 & -3 & -2  & -1 & 0 &  \\
1   &&  &   & \drt{1}{1110} &  & \underline{\drt{0}{0001}} &  &  \underline{\drt{0}{0010}}  &  & \underline{\drt{0}{0100}}  & & \underline{\drt{0}{1000}} & & : \text{\Large $(-q_s)^k$}  \\
2   && & \drt{1}{0110} &  & \drt{1}{1111}    & & \drt{0}{0011}  & &  \drt{0}{0110}    &  & \drt{0}{1100}     \\
3   &&  \drt{1}{0010} & & \drt{1}{0111}    &   & \drt{1}{1121}     &  &  \drt{0}{0111}      &   & \drt{0}{1110}    \\
4   & \underline{\drt{1}{0000}} & & {\drt{1}{0011}}   & & \drt{1}{0121}    &  & \drt{1}{1221}    &  & \drt{0}{1111}   & & \\
& 
}}
$$
Here we set $ {\scriptstyle \drt{a_5}{a_1a_2a_3a_4}} \seteq \sum_{k=1}^5 a_k\al_k \in \Df$ and the underlined ones are simple roots. 
Combining Proposition \ref{Prop: La and d} and Proposition \ref{Prop: d and d}
with the denominator formula given in Appendix \ref{App: denominators}, we compute that $ \de (V_Q(\al_i), \D^k V_Q(\al_j) ) = 0 $ for $ i\ne j, k \ne0$ and
\begin{align*}
\Li(V(\varpi_1), V(\varpi_1)_{(-q_s)^k}) &= \de ( V(\varpi_1), V(\varpi_1)_{(-q_s)^k} ) \\ 
& = \delta_{k,2} \qquad \text{ for } k=2, 4, \ldots, 2n-2,\\
\Li(V(\varpi_n), V(\varpi_1)_{(-q_s)^t}) &=  \de ( V(\varpi_n), V(\varpi_1)_{(-q_s)^t} ) \\
& = \delta_{t, n+3} \qquad \text{ for } t=n+1, n+3, \ldots, 3n-1.
\end{align*}
Therefore, for $i > j$, we have 
 $$
\mQ{i,j} = 
\begin{cases}
-1 & \text{ if ($i \le n$ and $i-j=1$) or  $(i,j)= (n+1,n-1) $},    \\
0 & \text{ otherwise,} 
\end{cases}
$$
which says that $ \MQ$ is a Cartan matrix of type $D_{  n+1}$.

\sn
\textbf{(Type $D_n^{(1)}$)} \ 
Recall the explicit description of $\sigQ$ for type $D_n^{(1)}$ ($n \ge 4$). Note that the Dynkin quiver corresponding to $\sig_Q$ is given in $\eqref{Eq: DQ}$. In this case $h=2n-2$ and,
 for $1 \le i \le n$,  
 $$
 \phi_Q(\alpha_i) = 
\begin{cases}
 (1, (-q)^{-2(i-1)}) & \text{ if } i \le n-2, \\
  (n-1, (-q)^{ -3n+6 }) & \text{ if ($i = n-1$ and $n$ is even) or ($i = n$ and $n$ is odd)}, \\
  (n, (-q)^{ -3n+6 }) & \text{ if ($i = n $ and $n$ is even) or ($i = n-1$ and $n$ is odd),}  
\end{cases}
$$ 
by \cite[Lemma 3.2.3]{KKK15B}. 
For example, if it is of type $D_5^{(1)}$, then elements $(i, (-q)^k)$ of $ \sigQ $ with the values of $ \phi_Q^{-1} $ can be drawn as follows: 
$$ 
\scalebox{0.65}{\xymatrix@C=0ex@R=-1ex{ 
 i \backslash k     \ar@{-}[dddddd]<1.7ex> \ar@{-}[rrrrrrrrrrr]<-1.2ex>   & -9 & -8 & -7 & -6 & -5 & -4 & -3 & -2  & -1 & 0 & \\
1    &   & &  & \drt{1}{1111} &  &  \underline{\drt{0}{0010}}  &  & \underline{\drt{0}{0100}}  & & \underline{\drt{0}{1000}}   \\
2    & &  & \drt{1}{0111}    & & \drt{1}{1121}  & &  \drt{0}{0110}    &  & \drt{0}{1100}     \\
3     & & \drt{1}{0011}    &   & \drt{1}{0121}     &  &  \drt{1}{1221}      &   & \drt{0}{1110}    \\
4    & \underline{\drt{1}{0000}}   & & \drt{0}{0011}    &  & \drt{1}{0110}    &  & \drt{0}{1111}   & & \\
5    & \underline{\drt{0}{0001}}   & & \drt{1}{0010}    &  & \drt{0}{0111}    &  & \drt{1}{1110}   & & \\
& 
}}
$$
Here we set $ {\scriptstyle \drt{a_5}{a_1a_2a_3a_4}} \seteq \sum_{k=1}^5 a_k\al_k \in \Df$ and the underlined ones are simple roots. 
Using the formula given in Appendix  \ref{App: ADE}, one can compute that, for $1 \le k < h$,  
\begin{align*}
&\tcmc_{1,1}(k) = \delta_{k,1} + \delta_{k,2n-3}, \qquad \tcmc_{n,1}(k)=\tcmc_{n-1,1}(k) = \delta_{k,n-1} , \\
 & \tcmc_{n,n}(k) = \tcmc_{n-1,n-1}(k) = \delta( k\equiv 1 \ \mod 4 ), \\ 
 & \tcmc_{n,n-1}(k) = \tcmc_{n-1,n}(k) = \delta( k\equiv 3 \ \mod 4 ).
\end{align*} 
Combining with Lemma \ref{Lem: Li for ADE}, we compute that
\begin{align*}
 \Li(V(\varpi_1), V(\varpi_1)_{(-q)^{k}} ) &= \delta_{k,2} \quad \text{ for $ 2 \le  k \le h-4 $}, \\
 \Li(V(\varpi_n), V(\varpi_1)_{(-q)^{k}} ) &= \delta_{k,n} \quad \text{ for $ n \le  k \le 3n-6 $}, \\
 \Li(V(\varpi_n), V(\varpi_{n-1}))  &= 0.
\end{align*}
Therefore, for $i > j$, we have 
 $$
\mQ{i,j} = 
\begin{cases}
-1 & \text{ if ($i \le n-1$ and $i-j=1$) or  $(i,j)= (n,n-2) $},    \\
0 & \text{ otherwise,} 
\end{cases}
$$
which says that $ \MQ$ is a Cartan matrix of type $D_n$.

\sn
\textbf{(Type $A_{2n}^{(2)}$)} \ 
Recall the explicit description of $\sigQ$ for type $A_{2n}^{(2)}$ ($n \ge 1$), which can be obtained from \cite{KKKO15D}. 
In this case, $\gf$ is of type $A_{2n}$ and, for $1 \le i \le 2n$ we have 
 $$
 \phi_Q(\alpha_i) =  (1, (-q)^{2-2i}).
$$
For example, if it is of type $A_4^{(2)}$, then elements of $ \sigQ $ with the values of $ \phi_Q^{-1} $ can be drawn as follows:
$$ 
\scalebox{0.8}{\xymatrix@C=2ex@R=-1ex{ 
 i \backslash k     \ar@{-}[dddddd]<1.5ex> \ar@{-}[rrrrrrrr]<-1.3ex>   & -6 & -5 & -4 & -3 & -2 & -1 & 0 &  & &   \\
1    & \underline{(0001)} & & \underline{(0010)} & & \underline{(0100)} & & \underline{(1000)} & & & \\
2    & & (0011)& & (0110)& & (1100)& &  &  \\
  \ar@{.}[rrrrrrrr]<0ex>     &   &   &   &   &   &   &   &  &  &  \\
2     & & & (0111) & &(1110) & &   & \\
1    & & &  & (1111)& &  && & \\
& 
}}
$$
Here, $(a_1,a_2,a_3,a_4) \seteq \sum_{k=1}^4 a_k\al_k \in \Df$, and the underlined ones are simple roots.
It follows from Proposition \ref{Prop: La and d}, Proposition \ref{Prop: d and d}
and the denominator formula given in Appendix \ref{App: denominators} that 
$ \de (V_Q(\al_i), \D^k V_Q(\al_j) ) = 0 $ for $ i\ne j, k \ne0$ and 
\begin{align*}
\Li(V(\varpi_1), V(\varpi_1)_{(-q)^k}) &= \de ( V(\varpi_1), V(\varpi_1)_{(-q )^k} )  \\ 
& = \delta_{k,2} \qquad \text{ for } k=2, 4, \ldots, 4n-2.
\end{align*}
Therefore, for $i > j$, we have 
 $$
\mQ{i,j} = -  \Li(V(\varpi_1), V(\varpi_1)_{(-q)^{2(i-j)}} ) = - \delta_{i-j, 1},
$$
which tells us that $ \MQ$ is a Cartan matrix of type $A_{2n}$.

\sn
\textbf{(Type $A_{2n-1}^{(2)}$)} \ 
Recall the explicit description of $\sigQ$ for type $A_{2n-1}^{(2)}$ ($n \ge 2$), which can be obtained from \cite{KKKO15D}. 
In this case, $\gf$ is of type $A_{2n-1}$ and, for $1 \le i \le 2n-1$ we have 
 $$
 \phi_Q (\alpha_i) =  (1, (-q)^{2-2i}).
$$
For example, if it is of type $A_5^{(2)}$, then elements of $ \sigQ $ with the values of $ \phi_Q ^{-1}$ can be drawn as follows:
$$ 
\scalebox{0.75}{\xymatrix@C=1.5ex@R=-1ex{ 
 i \backslash k     \ar@{-}[ddddddd]<1.5ex> \ar@{-}[rrrrrrrrrr]<-1.3ex>  &-8 &-7 & -6 & -5 & -4 & -3 & -2 & -1 & 0 &      \\
1   &\underline{(00001)} & & \underline{(00010)} & & \underline{(00100)} & & \underline{(01000)} & & \underline{(10000)} & &  : \text{\large $(-q)^k$} \\
2   && (00011) & & (00110)& & (01100)& & (11000)& &  &  \\
3   &&  & (00011) & & (01110) & &(11100) & &   &   \\
  \ar@{.}[rrrrrrrrrr]<0ex>     &   &   &   &   &   &   &   &  &  &  \\
2   && & & (01111) &  & (11110)& &  &&   & :  \text{\large $- (-q)^k$} \\
1   && & & &  (11111) & & &  && & \\
& 
}}
$$
Here, $(a_1,a_2,a_3,a_4,a_5) \seteq \sum_{k=1}^5 a_k\al_k \in \Df$, and the underlined ones are simple roots.
Note that $V(\varpi_n)_{a} \simeq V(\varpi_n)_{-a}$.
It follows from Proposition \ref{Prop: La and d}, Proposition \ref{Prop: d and d}
and the denominator formula given in Appendix \ref{App: denominators} that 
$ \de (V_Q(\al_i), \D^k V_Q(\al_j) ) = 0 $ for $ i\ne j, k \ne0$ and
\begin{align*}
\Li(V(\varpi_1), V(\varpi_1)_{(-q)^k}) &= \de ( V(\varpi_1), V(\varpi_1)_{(-q )^k} ) \\ 
& = \delta_{k,2} \qquad \text{ for } k=2, 4, \ldots, 4n-4.
\end{align*}
Thus,  we obtain 
 $$
\mQ{i,j} = -  \Li(V(\varpi_1), V(\varpi_1)_{(-q)^{2(i-j)}} ) = - \delta_{i-j, 1}, \quad \text{ for }i > j,
$$
which implies that $ \MQ$ is a Cartan matrix of type $A_{2n-1}$.

\sn
\textbf{(Type $D_{n+1}^{(2)}$)} \ 
Recall the explicit description of $\sigQ$ for type $D_{n+1}^{(2)}$ ($n \ge3$), which can be obtained from \cite{KKKO15D}. 
In this case, $\gf$ is of type $D_{ n+1}$ and, for $1 \le i \le  n+1$ we have 
 $$
 \phi_Q (\alpha_i) = 
\begin{cases}
 (1, (\sqrt{-1})^{n} (-q)^{-2(i-1)}) & \text{ if } i \le n-1, \\
  (  n, (-1)^i (-q)^{ -3n+3 }) & \text{ if } i=n, n+1.
\end{cases}
$$ 
For example, if it is of type $D_5^{(2)}$, then elements of $ \sigQ $ with the values of $ \phi_Q^{-1} $ can be drawn as follows:
$$ 
\scalebox{0.60}{\xymatrix@C=1ex@R=0ex{ 
 i \backslash k     \ar@{-}[ddddddd]<2.5ex> \ar@{-}[rrrrrrrrrrr]<-2ex>   & -9 & -8 & -7 & -6 & -5 & -4 & -3 & -2  & -1 & 0 & \\
1    &   & &  & \drt{1}{1111} &  &  \underline{\drt{0}{0010}}  &  & \underline{\drt{0}{0100}}  & & \underline{\drt{0}{1000}}  & &   : \text{\Large $(-q)^k$} \\
2    & &  & \drt{1}{0111}    & & \drt{1}{1121}  & &  \drt{0}{0110}    &  & \drt{0}{1100}      && &   : \text{\Large $- \sqrt{-1}(-q)^k$}\\
3     & & \drt{1}{0011}    &   & \drt{1}{0121}     &  &  \drt{1}{1221}      &   & \drt{0}{1110}   &&& &   : \text{\Large $- (-q)^k$}   \\
  \ar@{.}[rrrrrrrrrrr]<0ex>     &   &   &   &   &   &   &   &  &  &  &&\\
4    & \underline{\drt{1}{0000}}   & & \drt{0}{0011}    &  & \drt{1}{0110}    &  & \drt{0}{1111}   & &   & & &   : \text{\Large $(-q)^k$} \\
4    & \underline{\drt{0}{0001}}   & & \drt{1}{0010}    &  & \drt{0}{0111}    &  & \drt{1}{1110}   & &  & & &   :  \text{\Large $ - (-q)^k$} \\
& 
}}
$$
Here we set $ {\scriptstyle \drt{a_5}{a_1a_2a_3a_4}} \seteq \sum_{k=1}^5 a_k\al_k \in \Df$ and the underlined ones are simple roots. 
Note that $V(\varpi_i)_{a} \simeq V(\varpi_i)_{-a}$ for $i < n$.
It follows from Proposition \ref{Prop: La and d}, Proposition \ref{Prop: d and d}
and the denominator formula given in Appendix \ref{App: denominators} that 
$ \de (V_Q(\al_i), \D^k V_Q(\al_j) ) = 0 $ for $ i\ne j, k \ne0$ and 
\begin{align*}
\Li(V(\varpi_1), V(\varpi_1)_{(-q)^k}) &= \de ( V(\varpi_1), V(\varpi_1)_{(-q)^k} ) \\
& = \delta_{k,2} \qquad \text{ for } k=2, 4, \ldots, 2n-4, \\
\Li(V(\varpi_n), V(\varpi_1)_{\pm \sqrt{-1}^{n } (-q)^{k}}) &= 
\de ( V(\varpi_{n }), V(\varpi_1)_{\pm \sqrt{-1}^{n } (-q)^{k}} ) \\
&  = \delta_{k,n+1} \qquad \text{ for } k=n+1, n+3, \ldots, 3n-3,\\
\Li(V(\varpi_n), V(\varpi_n)_{-1}) &= 0.
\end{align*}
which give the values of $\mQ{i,j}$. Thus, one can check that the matrix $\MQ$ is a Cartan matrix of type $D_{n+1}$.

\sn
\textbf{(Type $E_6^{(1)}$)} \ 
Recall the explicit description of $\sigQ$ for type $E_6^{(1)}$. The Dynkin quiver corresponding to $\sig_Q$ is given in $\eqref{Eq: DQ}$. In this case, $h=12$ and
elements $(i, (-q)^k)$ of $ \sigQ $ with the values of $ \phi_Q^{-1} $ can be drawn as follows:  
\begin{equation*}
\raisebox{6em}{\scalebox{0.6}{\xymatrix@R=-1.0ex@C=-1.0ex{
i \backslash k   \ar@{-}[ddddddd]<1.3ex> \ar@{-}[rrrrrrrrrrrrrrrr]<-1.3ex>  & -14 & -13 & -12 & -11 & -10 & -9 & -8 & -7 & -6 & -5 & -4 & -3 & -2 & -1 & 0 &   \\
1& \underline{\scriptstyle\prt{000}{001}}  && \underline{\scriptstyle\prt{000}{010}}   && \underline{\scriptstyle\prt{000}{100}}  && {\scriptstyle\prt{011}{111}} 
&&{\scriptstyle\prt{101}{110}}  && {\scriptstyle\prt{010}{100}}  && \underline{\scriptstyle\prt{001}{000}}  && \underline{ {\scriptstyle\prt{100}{000}} } \\
3&& {\scriptstyle\prt{000}{011}}  && {\scriptstyle\prt{000}{110}}  && {\scriptstyle\prt{011}{211}} 
&& {\scriptstyle\prt{112}{221}}  && {\scriptstyle\prt{111}{210}}  && {\scriptstyle\prt{011}{100}} 
&& {\scriptstyle\prt{101}{000}}  \\ 
4&&& {\scriptstyle\prt{000}{111}} && {\scriptstyle\prt{011}{221}} && {\scriptstyle\prt{112}{321}} 
&& {\scriptstyle\prt{122}{321}}  && {\scriptstyle\prt{112}{210}} && {\scriptstyle\prt{111}{100}}  \\
2&&&& {\scriptstyle\prt{010}{111}} && {\scriptstyle\prt{001}{110}} && {\scriptstyle\prt{111}{211}} && {\scriptstyle\prt{011}{110}} 
&& {\scriptstyle\prt{101}{100}} && \underline{\scriptstyle\prt{010}{000}} \\
5&&&& {\scriptstyle\prt{001}{111}} && {\scriptstyle\prt{111}{221}} && {\scriptstyle\prt{011}{210}} 
&& {\scriptstyle\prt{112}{211}} && {\scriptstyle\prt{111}{110}}  \\
6&&&&& {\scriptstyle\prt{101}{111}} && {\scriptstyle\prt{010}{110}} && {\scriptstyle\prt{001}{100}} && {\scriptstyle\prt{111}{111}}  \\
& \\
}}}
\end{equation*}
Here we set $ {\scriptstyle\prt{a_1a_2a_3}{a_4a_5a_6}} \seteq \sum_{i=1}^6 a_i \al_i  \in \Df$, and the underlined ones are simple roots.
Using the formula given in Appendix  \ref{App: ADE}, one can compute that, for $1 \le k < h$,  
\begin{align*}
\tcmc_{1,1}(k) = \delta_{k,1} + \delta_{k,7}, \quad \tcmc_{1,2}(k) = \delta_{k,4} + \delta_{k,8}.
\end{align*} 
By Lemma \ref{Lem: Li for ADE}, we compute the following:
\begin{align*}
 \Li(V(\varpi_1), V(\varpi_1)_{(-q)^{k}} ) &= \delta_{k,2} + \delta_{k,8}  \quad \text{ for $ k=2,4,8,10,12,14 $}, \\
 \Li(V(\varpi_1), V(\varpi_2)_{(-q)^{k}} ) &= \delta_{k,9} \quad \text{ for $ k=  -1, 1,9,11,13$},
\end{align*}
which give the values of $\mQ{i,j}$. Therefore, one can check that the matrix $\MQ$ is a Cartan matrix of type $E_6$.

\sn
\textbf{(Type $E_7^{(1)}$)} \ 
Recall the explicit description of $\sigQ$ for type $E_7^{(1)}$. The Dynkin quiver corresponding to $\sig_Q$ is given in $\eqref{Eq: DQ}$. In this case, $h=18$ and
elements $(i, (-q)^k)$ of $ \sigQ $ with the values of $ \phi_Q^{-1} $ can be drawn as follows:  
\[
\scalebox{0.6}{\xymatrix@R=-0.5ex@C=-3ex{
i \backslash k  \ar@{-}[dddddddd]<0.0ex> \ar@{-}[rrrrrrrrrrrrrrrrrrrrrrr]<-1.5ex>   &-21&-20&-19&-18&-17&-16&-15&-14&-13&-12&-11&-10&-9&-8&-7&-6&-5&-4&-3&-2&-1&0  &\\
1 \quad &&&&&&{\scriptstyle\prt{1011}{111}}  &&{\scriptstyle\prt{0101}{110}}  &&{\scriptstyle\prt{0011}{100}}  &&{\scriptstyle\prt{1112}{111}}  
&&{\scriptstyle\prt{0111}{110}}   &&{\scriptstyle\prt{1011}{100}}  && {\scriptstyle\prt{0101}{000}}  
&&\underline{\scriptstyle\prt{0010}{000}}  &&\underline{\scriptstyle\prt{1000}{000}} \\
3\quad  &&&&&{\scriptstyle\prt{0011}{111}}  &&{\scriptstyle\prt{1112}{221}} &&{\scriptstyle\prt{0112}{210}} 
&&{\scriptstyle\prt{1123}{211}} &&{\scriptstyle\prt{1223}{221}}  
&&{\scriptstyle\prt{1122}{210}} &&{\scriptstyle\prt{1112}{100}} &&{\scriptstyle\prt{0111}{000}} &&{\scriptstyle\prt{1010}{000}} \\
4\quad &&&&{\scriptstyle\prt{0001}{111}} &&{\scriptstyle\prt{0112}{221}} 
&&{\scriptstyle\prt{1123}{321}} &&{\scriptstyle\prt{1224}{321}} &&{\scriptstyle\prt{1234}{321}} 
&&{\scriptstyle\prt{2234}{321}} &&{\scriptstyle\prt{1223}{210}} 
&&{\scriptstyle\prt{1122}{100}} &&{\scriptstyle\prt{1111}{000}} 
\\
2\quad &&&&&{\scriptstyle\prt{0101}{111}} &&{\scriptstyle\prt{0011}{110}} &&{\scriptstyle\prt{1112}{211}} &&{\scriptstyle\prt{0112}{110}} 
&&{\scriptstyle\prt{1122}{211}} &&{\scriptstyle\prt{1112}{110}} &&{\scriptstyle\prt{0111}{100}} &&{\scriptstyle\prt{1011}{000}} 
&&\underline{\scriptstyle\prt{0100}{000}}
\\
5\quad &&&{\scriptstyle\prt{0000}{111}} &&{\scriptstyle\prt{0001}{110}} &&{\scriptstyle\prt{0112}{211}} 
&&{\scriptstyle\prt{1123}{221}} &&{\scriptstyle\prt{1223}{321}} &&{\scriptstyle\prt{1123}{210}} 
&&{\scriptstyle\prt{1223}{221}} &&{\scriptstyle\prt{1122}{110}} &&{\scriptstyle\prt{1111}{100}} 
\\
6\quad &&{\scriptstyle\prt{0000}{011}} &&{\scriptstyle\prt{0000}{110}} &&{\scriptstyle\prt{0001}{100}} 
&&{\scriptstyle\prt{0112}{111}} &&{\scriptstyle\prt{1122}{221}} 
&&{\scriptstyle\prt{1112}{210}} &&{\scriptstyle\prt{0112}{100}} &&{\scriptstyle\prt{1122}{111}} 
&&{\scriptstyle\prt{1111}{110}} 
\\
7\quad & \underline{\scriptstyle\prt{0000}{001}}  &&\underline{\scriptstyle\prt{0000}{010}} &&\underline{\scriptstyle\prt{0000}{100}} &&\underline{\scriptstyle\prt{0001}{000}} 
&&{\scriptstyle\prt{0111}{111}}  &&{\scriptstyle\prt{1011}{110}} &&{\scriptstyle\prt{0101}{100}} 
&&{\scriptstyle\prt{0011}{000}} &&{\scriptstyle\prt{1111}{111}}  \\
&
}}
\]
Here we set $ {\scriptstyle\prt{a_1a_2a_3a_4}{a_5a_6a_7}} \seteq \sum_{i=1}^7 a_i \al_i  \in \Df$, and the underlined ones are simple roots.
Using the formula given in Appendix  \ref{App: ADE}, one can compute that, for $1 \le k < h$,  
\begin{align*}
& \tcmc_{1,1}(k) = \delta_{k,1} + \delta_{k,7} + \delta_{k,11} + \delta_{k,17}, \qquad 
\tcmc_{1,2}(k) = \delta_{k,4} + \delta_{k,8} + \delta_{k,10} + \delta_{k,14}, \\
& \tcmc_{7,1}(k) = \delta_{k,6} + \delta_{k,12}, \quad 
 \tcmc_{7,2}(k) = \delta_{k,5} + \delta_{k,9} + \delta_{k,13}, \quad 
 \tcmc_{7,7}(k) = \delta_{k,1} + \delta_{k,9} + \delta_{k,17}. 
\end{align*} 
By Lemma \ref{Lem: Li for ADE}, we compute the following:
\begin{align*}
 \Li(V(\varpi_1), V(\varpi_1)_{(-q)^{2}} ) &= 1, \quad   \Li(V(\varpi_1), V(\varpi_2)_{(-q)} ) = \Li(V(\varpi_2), V(\varpi_1)_{(-q)} )= 0,  \\
 \Li(V(\varpi_7), V(\varpi_1)_{(-q)^{k}} ) &= \delta_{k,13} \quad \text{ for $ k=13,15,17,19,21$}, \\
 \Li(V(\varpi_7), V(\varpi_{2})_{{(-q)}^k})  &= \delta_{k,14}  \quad \text{ for $ k=14,16,18,20$}, \\ 
 \Li(V(\varpi_7), V(\varpi_{7})_{{(-q)}^k}  )  &= \delta_{k,2}  \quad \text{ for $ k=2,4,6$},
\end{align*}
which give the values of $\mQ{i,j}$. Therefore, one can check that the matrix $\MQ$ is a Cartan matrix of type $E_7$.

\sn
\textbf{(Type $E_8^{(1)}$)} \ 
Recall the explicit description of $\sigQ$ for type $E_8^{(1)}$. The Dynkin quiver corresponding to $\sig_Q$ is given in $\eqref{Eq: DQ}$. In this case, $h=30$ and
elements $(i, (-q)^k)$ of $ \sigQ $ with the values of $ \phi_Q^{-1} $ can be drawn as follows:  
\[
\scalebox{0.50}{\xymatrix@R=-1ex@C=-3ex{
 i \backslash k   \ar@{-}[ddddddddd]<0.0ex> \ar@{-}[rrrrrrrrrrrrrrrrrrrrrrrrrrrrrrrrrrrr]<-1.3ex>
& -34 & -33 & -32&-31&-30&-29&-28&-27&-26&-25&-24&-23&-22&-21&-20&-19&-18&-17&-16&-15&-14&-13&-12&-11&-10&-9&-8&-7&-6&-5&-4&-3&-2&-1&0&   \\
1 \quad  &&&&&&&\pprt{10}{11}{11}{11} &&\pprt{01}{01}{11}{10} &&\pprt{00}{11}{11}{00} &&\pprt{11}{12}{21}{11} &&\pprt{01}{12}{11}{10} 
&&\pprt{11}{22}{22}{11} &&\pprt{11}{12}{21}{10} &&\pprt{01}{12}{11}{00} &&\pprt{11}{22}{21}{11} &&\pprt{11}{12}{11}{10} 
&&\pprt{01}{11}{11}{00} &&\pprt{10}{11}{10}{00} &&\pprt{01}{01}{00}{00} && \underline{\pprt{00}{10}{00}{00}} && \underline{\pprt{10}{00}{00}{00}}
\\
3\quad &&&&&&\pprt{00}{11}{11}{11} &&\pprt{11}{12}{22}{21} &&\pprt{01}{12}{22}{10} &&\pprt{11}{23}{32}{11} 
&&\pprt{12}{24}{32}{21} &&\pprt{12}{34}{33}{21} &&\pprt{22}{34}{43}{21} &&\pprt{12}{24}{32}{10} 
&&\pprt{12}{34}{32}{11} &&\pprt{22}{34}{32}{21} &&\pprt{12}{23}{22}{10} &&\pprt{11}{22}{21}{00} 
&&\pprt{11}{12}{10}{00} &&\pprt{01}{11}{00}{00} &&\pprt{10}{10}{00}{00} 
\\
4\quad &&&&&\pprt{00}{01}{11}{11} &&\pprt{01}{12}{22}{21} &&\pprt{11}{23}{33}{21} &&
\pprt{12}{24}{43}{21} &&\pprt{12}{35}{43}{21} &&\pprt{23}{46}{54}{32} 
&&\pprt{23}{46}{54}{31} &&\pprt{23}{46}{54}{21} &&\pprt{23}{46}{53}{21} 
&&\pprt{23}{46}{43}{21} &&\pprt{23}{45}{43}{21} &&\pprt{22}{34}{32}{10} 
&&\pprt{12}{23}{21}{00} &&\pprt{11}{22}{10}{00} &&\pprt{11}{11}{00}{00} 
\\
2\quad &&&&&&\pprt{01}{01}{11}{11} &&\pprt{00}{11}{11}{10} &&\pprt{11}{12}{22}{11} &&\pprt{01}{12}{21}{10} &&\pprt{11}{23}{22}{11} 
&&\pprt{12}{23}{32}{21} &&\pprt{11}{23}{32}{10} &&\pprt{12}{23}{32}{11} &&\pprt{11}{23}{21}{10} &&\pprt{12}{23}{22}{11} 
&&\pprt{11}{22}{21}{10} &&\pprt{11}{12}{11}{00} &&\pprt{01}{11}{10}{00} &&\pprt{10}{11}{00}{00} && \underline{\pprt{01}{00}{00}{00}}
\\
5\quad &&&&\pprt{00}{00}{11}{11} &&\pprt{00}{01}{11}{10} &&\pprt{01}{12}{22}{11} &&\pprt{11}{23}{32}{21} 
&&\pprt{12}{24}{33}{21} &&\pprt{12}{34}{43}{21} &&\pprt{22}{35}{43}{21} &&\pprt{13}{35}{43}{21} 
&&\pprt{22}{45}{43}{21} &&\pprt{23}{35}{43}{21} &&\pprt{12}{34}{32}{10} &&\pprt{22}{34}{32}{11} 
&&\pprt{12}{23}{21}{10} &&\pprt{11}{22}{11}{00} &&\pprt{11}{11}{10}{00} 
\\
6\quad &&&\pprt{00}{00}{01}{11} &&\pprt{00}{00}{11}{10} &&\pprt{00}{01}{11}{00} &&\pprt{01}{12}{21}{11} &&\pprt{11}{23}{22}{21} 
&&\pprt{12}{23}{33}{21} &&\pprt{11}{23}{32}{10} &&\pprt{12}{24}{32}{11} &&\pprt{12}{34}{32}{21} 
&&\pprt{22}{34}{33}{21} &&\pprt{12}{23}{32}{10} &&\pprt{11}{23}{21}{00} &&\pprt{12}{23}{21}{11} 
&&\pprt{11}{22}{11}{10} &&\pprt{11}{11}{11}{00} 
\\
7\quad &&\pprt{00}{00}{00}{11} &&\pprt{00}{00}{01}{10} &&\pprt{00}{00}{11}{00} &&\pprt{00}{01}{10}{00} 
&&\pprt{01}{12}{11}{11} &&\pprt{11}{22}{22}{21} &&\pprt{11}{12}{22}{10} 
&&\pprt{01}{12}{21}{00} &&\pprt{11}{23}{21}{11} &&\pprt{12}{23}{22}{21} &&\pprt{11}{22}{22}{10} 
&&\pprt{11}{12}{21}{00} &&\pprt{01}{12}{10}{00} &&\pprt{11}{22}{11}{11}  &&\pprt{11}{11}{11}{10}
\\
8\quad &\underline{\pprt{00}{00}{00}{01}} && \underline{\pprt{00}{00}{00}{10}} && \underline{\pprt{00}{00}{01}{00}} && \underline{\pprt{00}{00}{10}{00}} 
&& \underline{\pprt{00}{01}{00}{00}} &&\pprt{01}{11}{11}{11} &&\pprt{10}{11}{11}{10} &&\pprt{01}{01}{11}{00} 
&&\pprt{00}{11}{10}{00} &&\pprt{11}{12}{11}{11} &&\pprt{01}{11}{11}{10} 
&&\pprt{10}{11}{11}{00} &&\pprt{01}{01}{10}{00} &&\pprt{00}{11}{00}{00} &&\pprt{11}{11}{11}{11}  \\
& 
}}
\]
Here we set $ {\scriptstyle\pprt{a_1a_2}{a_3a_4}{a_5a_6}{a_7a_8}} \seteq \sum_{i=1}^8 a_i \al_i  \in \Df$, and the underlined ones are simple roots.
Using the formula given in Appendix  \ref{App: ADE}, one can compute that, for $1 \le k < h$,  
\begin{align*}
 \tcmc_{1,1}(k) & = \delta(k=1,7,11,13,17,19,23,29), \\ 
\tcmc_{1,2}(k) & = \delta(k=4,8,10,12,14,16,18,20,22,26), \\
\tcmc_{8,1}(k) &= \delta(k=7,13,17,23), \\ 
 \tcmc_{8,2}(k) &= \delta(k=6,10,14,16,20,24), \\ 
 \tcmc_{8,8}(k) &= \delta(k=1,11,19,29). 
\end{align*} 
By Lemma \ref{Lem: Li for ADE}, we compute the following:
\begin{align*}
 \Li(V(\varpi_1), V(\varpi_1)_{(-q)^{2}} ) &= 1, \quad   \Li(V(\varpi_1), V(\varpi_2)_{(-q)} ) = \Li(V(\varpi_2), V(\varpi_1)_{(-q)} )= 0,  \\
 \Li(V(\varpi_8), V(\varpi_1)_{(-q)^{k}} ) &= \delta_{k,24} \quad \text{ for $ k=24,26,28,30,32,34$}, \\
 \Li(V(\varpi_8), V(\varpi_{2})_{(-q)^k} )  &= \delta_{k,25}  \quad \text{ for $ k=25,27,29,31,33$}, \\ 
 \Li(V(\varpi_8), V(\varpi_{8})_{(-q)^k} ) &= \delta_{k,2}  \quad \text{ for $ k=2,4,6,8$},
\end{align*}
which give the values of $\mQ{i,j}$. Therefore, one can check that the matrix $\MQ$ is a Cartan matrix of type $E_8$.

\sn
\textbf{(Type $F_4^{(1)}$)} \ 
Recall the explicit description of $\sigQ$ for type $F_4^{(1)}$, which can be obtained from \cite{OS19}. In this case, $\gf$ is of type $E_{6}$ and, elements of $ \sigQ $ with the values of $ \phi_Q^{-1} $ can be drawn as follows:  
\begin{align*} 
 \raisebox{5.6em}{\scalebox{0.60}{\xymatrix@C=-1.3ex@R=0.0ex{
 i \backslash k    \ar@{-}[ddddd]<1.3ex> \ar@{-}[rrrrrrrrrrrrrrrrrrrrr]<-3ex> & -  {19}   & -18 & - {17}  & -16 & -  {15}  & -14 & - {13}  & -12 & - {11}  & -10 & - {9}  & -8 & - {7}  & -6 & - {5}  & -4 & - {3}  & -2 & - {1}  & 0  &  \\
1 &&&& {\scriptstyle\prt{000}{111}}  &&{\scriptstyle\prt{111}{210}}&& {\scriptstyle\prt{011}{110}} && {\scriptstyle\prt{001}{111}}
&& {\scriptstyle\prt{111}{211}}
&& {\scriptstyle\prt{111}{110}}
&& \underline{\scriptstyle{\prt{001}{000}}}   &&
\underline{\scriptstyle{\prt{000}{001}}} && \underline{\scriptstyle{\prt{100}{000}}} &  & \quad : \text{\Large $(-1)^i q_s^k$} \\
2 && {\scriptstyle\prt{000}{110}}  &&  {\scriptstyle\prt{011}{210}}   && {\scriptstyle\prt{011}{221}} && {\scriptstyle\prt{112}{321}}
&&{\scriptstyle\prt{122}{321}}
&& {\scriptstyle\prt{112}{221}}
&&{\scriptstyle\prt{112}{211}}  &&
{\scriptstyle\prt{111}{111}}  && {\scriptstyle\prt{101}{000}}\\
3 & \underline{\scriptstyle{\prt{000}{100}}}   && {\scriptstyle\prt{010}{110}}  && {\scriptstyle\prt{001}{110}}
&& {\scriptstyle\prt{011}{211}}  && {\scriptstyle\prt{111}{221}} &&{\scriptstyle\prt{112}{210}}
&&{\scriptstyle\prt{011}{111}} &&
{\scriptstyle\prt{101}{111}} && {\scriptstyle\prt{111}{100}}\\
4 && {\scriptstyle\prt{010}{100}} && \underline{\scriptstyle{\prt{000}{010}}}  && {\scriptstyle\prt{001}{100}}  &&
{\scriptstyle\prt{010}{111}} &&{\scriptstyle\prt{101}{110}} && {\scriptstyle\prt{011}{100}} && {\scriptstyle\prt{000}{011}}
 &&{\scriptstyle\prt{101}{100}}  &&\underline{\scriptstyle{\prt{010}{000}}}\\
&
}}}
\end{align*}
Here we set $ {\scriptstyle\prt{a_1a_2a_3}{a_4a_5a_6}} \seteq \sum_{i=1}^6 a_i \al_i  \in \Df$, and the underlined ones are simple roots.
It follows from Proposition \ref{Prop: La and d}, Proposition \ref{Prop: d and d}
and the denominator formula given in Appendix \ref{App: denominators} that 
 $ \de (V_Q(\al_i), \D^k V_Q(\al_j) ) = 0 $ for $ i\ne j, k \ne0$ and
\begin{align*}
\Li(V(\varpi_1), V(\varpi_1)_{ q_s ^k}) &= \de ( V(\varpi_1), V(\varpi_1)_{ q_s^k} ) = \delta_{k,4} \qquad \text{ for } k=2,4, \\
\Li(V(\varpi_3), V(\varpi_1)_{ q_s ^k}) &= \de ( V(\varpi_3), V(\varpi_1)_{ q_s^k} ) = \delta_{k, {15} } \qquad \text{ for } k= {15}, {17},  {19}, \\
\Li(V(\varpi_4), V(\varpi_1)_{ -q_s ^k}) &= \de ( V(\varpi_4), V(\varpi_1)_{ - q_s^k} ) \\
& = \delta_{k,14} \qquad \qquad \qquad \text{ for } k=-2,0,2,12,14,16, \\
\Li(V(\varpi_3), V(\varpi_4)_{ - q_s^k}) &= \de ( V(\varpi_3), V(\varpi_4)_{- q_s^k} ) = 1 \qquad \text{ for } k= {3}, {17},  \\
\Li(V(\varpi_4), V(\varpi_4)_{  q_s^{14}}) & = \de ( V(\varpi_4), V(\varpi_4)_{ q_s^{14}} ) = 0, 
\end{align*}
which give the values of $\mQ{i,j}$. Thus, one can check that the matrix $\MQ$ is a Cartan matrix of type $E_6$.

\sn
\textbf{(Type $G_2^{(1)}$)} \ 
Recall the explicit description of $\sigQ$ for type $G_2^{(1)}$, which can be obtained from \cite{OS19}. In this case, $\gf$ is of type $D_{4}$ and, elements of $ \sigQ $ with the values of $ \phi_Q^{-1} $ can be drawn as follows:  
\begin{align*} 
& \scalebox{0.65}{\xymatrix@C=-1ex@R=0.3ex{ 
 i \backslash k     \ar@{-}[ddd]<1.5ex> \ar@{-}[rrrrrrrrrrrrr]<-1.7ex>   & -11 & -10 & -9 & -8 & -7 & -6 & -5 & -4 &-3  &-2 & -1&0  &\\
1  &  \underline{\ddrt{1}{000}}  & &\ddrt{1}{011} &  &  \ddrt{1}{111} &  &  \ddrt{1}{121}&  & \ddrt{0}{111} &  & \ddrt{0}{110} &  & &:  \text{ \large $ (-q_t)^k$} \\
2   &   &  \ddrt{1}{010} & &  \underline{\ddrt{0}{001}} & &  \ddrt{1}{110} & &  \ddrt{0}{011} & & \underline{\ddrt{0}{100}} & &  \underline{\ddrt{0}{010}}   \\
& 
}} 
\end{align*}
Here we set $ { \scriptstyle { \left( \begin{matrix} \ \ \ \   a_4 \\ a_1 a_2 a_3 \end{matrix}\right) } }   \seteq \sum_{k=1}^4 a_k\al_k \in \Df$ and the underlined ones are simple roots. 
It follows from Proposition \ref{Prop: La and d}, Proposition \ref{Prop: d and d}
and the denominator formula given in Appendix \ref{App: denominators} that 
$ \de (V_Q(\al_i), \D^k V_Q(\al_j) ) = 0 $ for $ i\ne j, k \ne0$ and 
\begin{align*}
\Li(V(\varpi_1), V(\varpi_2)_{ (-q_t)^k}) &= \de ( V(\varpi_1), V(\varpi_2)_{ (-q_t)^k} ) = \delta_{k,11} \qquad \text{ for } k=3,9,11, \\
\Li(V(\varpi_2), V(\varpi_2)_{ (-q_t)^k}) &= \de ( V(\varpi_2), V(\varpi_2)_{ (-q_t)^k} ) = \delta_{k,2} + \delta_{k,8} \qquad \text{ for } k=2,6,8, 
\end{align*}
which give the values of $\mQ{i,j}$. Thus, one can check that the matrix $\MQ$ is a Cartan matrix of type $D_4$.

\sn
\textbf{(Type $E_6^{(2)}$)} \ 
Recall the explicit description of $\sigQ$ for type $E_6^{(2)}$, which can be obtained from \cite{OS19}. In this case, $\gf$ is of type $E_{6}$ and, elements of $ \sigQ $ with the values of $ \phi_Q^{-1} $ can be drawn as follows:  
\begin{equation*}
\raisebox{6em}{\scalebox{0.55}{\xymatrix@R=-0.5ex@C=1.0ex{
i \backslash k   \ar@{-}[ddddddddd]<3ex> \ar@{-}[rrrrrrrrrrrrrrrr]<-1.5ex>  & -14 & -13 & -12 & -11 & -10 & -9 & -8 & -7 & -6 & -5 & -4 & -3 & -2 & -1 & 0 &   \\
1& \underline{\scriptstyle\prt{000}{001}}  && \underline{\scriptstyle\prt{000}{010}}   && \underline{\scriptstyle\prt{000}{100}}  && {\scriptstyle\prt{011}{111}} 
&&{\scriptstyle\prt{101}{110}}  && {\scriptstyle\prt{010}{100}}  && \underline{\scriptstyle\prt{001}{000}}  && \underline{ {\scriptstyle\prt{100}{000}} } & & : \text{ \Large${ (- q)}^k$} \\
2&& {\scriptstyle\prt{000}{011}}  && {\scriptstyle\prt{000}{110}}  && {\scriptstyle\prt{011}{211}} 
&& {\scriptstyle\prt{112}{221}}  && {\scriptstyle\prt{111}{210}}  && {\scriptstyle\prt{011}{100}} 
&& {\scriptstyle\prt{101}{000}}  \\ 
  \ar@{.}[rrrrrrrrrrrrrrrr]<0ex>     &&   &   &   &   &   &   &   &  &  &&&&  &&&&\\
3&&& {\scriptstyle\prt{000}{111}} && {\scriptstyle\prt{011}{221}} && {\scriptstyle\prt{112}{321}} 
&& {\scriptstyle\prt{122}{321}}  && {\scriptstyle\prt{112}{210}} && {\scriptstyle\prt{111}{100}}  & &&& : \text{\Large $\sqrt{-1}(-q)^k$}  \\
4&&&& {\scriptstyle\prt{010}{111}} && {\scriptstyle\prt{001}{110}} && {\scriptstyle\prt{111}{211}} && {\scriptstyle\prt{011}{110}} 
&& {\scriptstyle\prt{101}{100}} && \underline{\scriptstyle\prt{010}{000}}  & &&   \\
  \ar@{.}[rrrrrrrrrrrrrrrr]<0ex>  &&      &   &   &   &   &   &   &  &  &  &&&&&&&  \\
2&&&& {\scriptstyle\prt{001}{111}} && {\scriptstyle\prt{111}{221}} && {\scriptstyle\prt{011}{210}} 
&& {\scriptstyle\prt{112}{211}} && {\scriptstyle\prt{111}{110}}   & &&&& : \text{\Large $- (-q)^k$} \\
1&&&&& {\scriptstyle\prt{101}{111}} && {\scriptstyle\prt{010}{110}} && {\scriptstyle\prt{001}{100}} && {\scriptstyle\prt{111}{111}}   \\
& 
}}}
\end{equation*}
Here we set $ {\scriptstyle\prt{a_1a_2a_3}{a_4a_5a_6}} \seteq \sum_{i=1}^6 a_i \al_i  \in \Df$, and the underlined ones are simple roots.
Note that $V(\varpi_i)_{a} \simeq V(\varpi_i)_{-a}$ for $i =3,4$.
It follows from Proposition \ref{Prop: La and d}, Proposition \ref{Prop: d and d}
and the denominator formula given in Appendix \ref{App: denominators} that 
 $ \de (V_Q(\al_i), \D^k V_Q(\al_j) ) = 0 $ for $ i\ne j, k \ne0$ and 
\begin{align*}
\Li(V(\varpi_1), V(\varpi_1)_{ q ^k}) &= \de ( V(\varpi_1), V(\varpi_1)_{ q^k} ) = \delta_{k,2} + \delta_{k,8} \qquad \text{ for } k=2,4,8,10,12,14, \\
\Li(V(\varpi_1), V(\varpi_4)_{ \sqrt{-1}q ^k}) &= \de ( V(\varpi_1), V(\varpi_4)_{ \sqrt{-1} q^k} )  \\
& = \delta_{k,9} \qquad \qquad \qquad \text{ for } k=-1,0,1,9,11,13.
\end{align*}
which give the values of $\mQ{i,j}$. Thus, one can check that the matrix $\MQ$ is a Cartan matrix of type $E_6$.

\sn
\textbf{(Type $D_4^{(3)}$)} \ 
Recall the explicit description of $\sigQ$ for type $D_4^{(3)}$, which can be obtained from \cite{OS19}. In this case, $\gf$ is of type $D_{4}$ and, elements of $ \sigQ $ with the values of $ \phi_Q^{-1} $ can be drawn as follows:  
\begin{align*} 
& \scalebox{0.65}{\xymatrix@C=-1ex@R=-0.5ex{ 
 i \backslash k     \ar@{-}[ddddddd]<1.5ex> \ar@{-}[rrrrrrrr]<-1.4ex>   & -6 & -5 & -4 &-3  &-2 & -1&0  &\\
1  &       &  &  \ddrt{1}{111} &  & \underline{\ddrt{0}{010}}   &  & \underline{\ddrt{0}{100}}  &  &:  \text{ \large $ q^k$} \\
  \ar@{.}[rrrrrrrr]<0ex>  &&      &   &   &   &   &   &   &  &  &  &&&  \\
2   &   &      \ddrt{1}{011}  & &  \ddrt{1}{121} & &  \ddrt{0}{110} & & &  :  \text{ \large $ -q^k$} \\
  \ar@{.}[rrrrrrrr]<0ex>  &&      &   &   &   &   &   &   &  &  &  &&&  \\
1   & \underline{\ddrt{0}{001}}  &   & \ddrt{1}{010} &&  {\ddrt{0}{111}}   &   & &   &  :  \text{ \large $ \omega q^k$}  \\
1   &  \underline{\ddrt{1}{000}} &   &\ddrt{0}{011}  &&  {\ddrt{1}{110}}  &  & &   &   :  \text{ \large $ \omega^2 q^k$} \\
& 
}} 
\end{align*}
Here we set $ { \scriptstyle { \left( \begin{matrix} \ \ \ \   a_4 \\ a_1 a_2 a_3 \end{matrix}\right) } }   \seteq \sum_{k=1}^4 a_k\al_k \in \Df$ and the underlined ones are simple roots. 
Note that $V(\varpi_2)_{a} \simeq V(\varpi_2)_{\omega^t a}$ for $t=1,2$.
It follows from Proposition \ref{Prop: La and d}, Proposition \ref{Prop: d and d}
and the denominator formula given in Appendix \ref{App: denominators} that 
 $ \de (V_Q(\al_i), \D^k V_Q(\al_j) ) = 0 $ for $ i\ne j, k \ne0$ and 
\begin{align*}
\Li(V(\varpi_1), V(\varpi_1)_{ \omega^t q^k}) =
\begin{cases}
1 & \text{ if } (t,k)= (0,2), (1,4), (2,4), \\
0 & \text{ if } (t,k)= (1,0), (2,0), (1,6), (2,6),
\end{cases} 
\end{align*}
which give the values of $\mQ{i,j}$. Thus, one can check that the matrix $\MQ$ is a Cartan matrix of type $D_4$.

\vskip 2em

\appendix

\section{Denominator formulas} \label{App: denominators}

The denominator formulas were studied and computed in \cite{AK, DO94, Fu19, KKK15B, Oh15, OS19}. 
In this Appendix, we write the denominator formulas for all types.

Let $q_s, q_t \in \cor^\times$ such that $q = q_s^2= q_t^3$, and $\omega \in \cor$ such that $\omega^2+\omega+1=0$.
For $i,j \in I$, we set 
$$
d_{i,j}(z) \seteq d_{V(\varpi_i), V(\varpi_j)} (z).
$$

\subsection{Simply-laced affine ADE types} \label{App: ADE}

Suppose that the Cartan matrix 
$\cmC = (c_{i,j})_{i,j\in I_0}$ is of type $A_n $, $D_n $ 
or $E_k $ ($k=6,7,8$). 
The quantum Cartan matrix $\cmC(z)= (c_{i,j}(z))_{i,j\in I_0}$ is defined by 
$$
c_{i,j}(z) \seteq \delta(i=j) (z + z^{-1}) + \delta(i\ne j) c_{i,j}.
$$
We denote by $\tcmC (z) = ( \tcmc_{i,j}(z) )_{i,j\in I_0}$ the inverse of $\cmC(z)$, and write 
\begin{align} \label{Eq: def of tcmt}
\tcmc_{i,j}(z) = \sum_{k\in \Z_{\ge0}} \tcmc_{i,j}(k) z^k \qquad \text{ for } i,j \in I_0.
\end{align}
Then the following beautiful formula is given in \cite[Theorem 2.10]{Fu19}
\begin{align} \label{Eq: denominators for ADE}
d_{i,j}(z) = \prod_{k=1}^{h -1}  (z - (-q)^{k+1})^{\tcmc_{i,j}(k)},
\end{align}
where $h $ is the Coxeter number. Note that the dual Coxeter number is equal to the Coxeter number in this case.

Let $\g_0$ be a simple Lie algebra of type $ADE$ with a index set $I_0$, and $Q$ be a Dynkin quiver of $\g$. 
Let $\xi\col I_0 \rightarrow \Z$ be a height function such that $ \xi_j = \xi_i - 1 $ for $i \rightarrow j$ in $Q$.
Choose a total order $>$ on $I$ such that $i > j$ for $\xi_i > \xi_j$ and write $I_0 = \{  i_1 > i_2 > \cdots > i_n  \}$.
We set $\tau \seteq s_{i_1} \cdots s_{i_n}$, which is a Coxeter element.
For $i\in I_0$, we set 
$
\gamma_i \seteq \sum_{j\in B(i)} \al_j,
$
where $B(i)$ is the subset of $I_0$ consisting of all elements $j$ such that there is a path from $j$ to $i$ in $Q$.
Then we have the following.
\begin{proposition} [\protect{\cite[Proposition 2.1]{HL15}}] \label{Prop: formula for ADE}
For $i,j\in I$ and $ k \in \Z_{>0}$, we have 
\begin{align*}
\tcmc_{i,j}(k) = 
\begin{cases}
( \tau^{(k +\xi_i - \xi_j -1)/2} (\gamma_i), \varpi_j) & \text{ if $k +\xi_i - \xi_j -1$ is even}, \\
0 & \text{ otherwise.}
\end{cases}
\end{align*}
\end{proposition}

In this paper, we take the following choice of Dynkin quivers:
\begin{equation} \label{Eq: DQ}
\begin{aligned} 
& A_{n} :   
\xymatrix@C=4ex@R=3ex{
  *{\circ}<3pt> \ar[r]_<{1  } & *{ \circ }<3pt> \ar[r]_<{2}   & \cdots \ar[r]_<{ }   &*{\circ}<3pt> \ar[r]_<{n-1} &*{ \circ }<3pt>  \ar@{-}[l]^<{n}}, \quad 
D_{n} :   
\raisebox{2.3em}{\xymatrix@C=4ex@R=3ex{
 &&& *{ \circ }<3pt> & \\  
  *{\circ}<3pt> \ar[r]_<{1  } & *{ \circ }<3pt> \ar[r]_<{2}   & \cdots \ar[r]_<{ }   &*{\circ}<3pt> \ar[r]_<{n-2} \ar[u]_>{n} &*{ \circ }<3pt>  \ar@{-}[l]^<{n-1}   }},
\\
&
 E_6 :  \raisebox{2.3em}{\xymatrix@C=3.4ex@R=3ex{  && *{\circ}<3pt>\ar[d]^<{2} \\
*{ \circ }<3pt> \ar[r]_<{1}  &*{\circ}<3pt>
\ar[r]_<{3} &*{ \circ }<3pt> \ar[r]_<{4} &*{\circ}<3pt>
\ar[r]_<{5} &*{\circ}<3pt>
\ar@{-}[l]^<{\ \ 6}}}, \qquad \quad 
E_7 : \raisebox{2.3em}{\xymatrix@C=3.4ex@R=3ex{ &&  *{\circ}<3pt>\ar[d]^<{2} \\
 *{ \circ }<3pt> \ar[r]_<{1}  &*{\circ}<3pt>
\ar[r]_<{3} &*{ \circ }<3pt> \ar[r]_<{4} &*{\circ}<3pt>
\ar[r]_<{5} &*{\circ}<3pt>
\ar[r]_<{6} &*{\circ}<3pt>
\ar@{-}[l]^<{7} } }, \allowdisplaybreaks \\
& E_8  :  \raisebox{2.3em}{\xymatrix@C=3.4ex@R=3ex{ && *{\circ}<3pt>\ar[d]^<{2} \\
*{ \circ }<3pt> \ar[r]_<{1}  &*{\circ}<3pt>
\ar[r]_<{3} &*{ \circ }<3pt> \ar[r]_<{4} &*{\circ}<3pt>
\ar[r]_<{5} &*{\circ}<3pt>
\ar[r]_<{6} &*{\circ}<3pt>
\ar[r]_<{7} &*{\circ}<3pt>
\ar@{-}[l]^<{8} } }.
\end{aligned}
\end{equation}
In this case we have the following data, which allow us to compute $\tcmc_{i,j}(k)$ explicitly.
\bna
\item \textbf{(Type $A_n$)} \ 
 $\tau = s_1s_2 \cdots s_n$, $\xi_i =1-i$ and $\gamma_i = \sum_{j=1}^i \alpha_j$.
\item \textbf{(Type $D_n$)} \ 
 $\tau = s_1s_2 \cdots s_{n-1} s_n$ and 
\begin{align*}
\xi_i = 
\begin{cases}
1-i & \text{ if } i < n-1, \\
-n+2 & \text{ if } i=n-1,n,
\end{cases} \quad
\gamma_i = 
\begin{cases}
\sum_{j=1}^i \alpha_j & \text{ if } i < n, \\
\sum_{j=1}^{n-2} \alpha_j + \al_n& \text{ if } i=n.
\end{cases}
\end{align*}

\item \textbf{(Type $E_n$) ($n=6,7,8$)} \ 
  $\tau = s_1s_2\ldots s_n$,  $\xi_1= 0$, $\xi_2=-1$, $\xi_k =2-k$ ($k=3,4,\ldots,n$), and 
$\gamma_1 = \alpha_1$, $\gamma_2 = \alpha_2$, $\gamma_3 = \alpha_1+\al_3$, $\gamma_t = \sum_{k=1}^t \al_k$ ($t=4,\ldots,n$).

\ee

Indeed,  in the figures of Section \ref{Sec: calculation of (,)}, the root $\gamma_i$ is the rightmost one in the row labeled by $i$, and $\tau$ corresponds to the horizontal translation by $-2$.
Hence one can read such values of $\tcmc_{i,j}(k)$ easily from the figures.

\subsection{Other classical affine types}\ 

The denominator formulas for other classical affine types can be found in \cite[Appendix C.4]{AK} for type $C_n^{(1)}$ and in \cite[Appendix]{Oh15} for type $B_n^{(1)}$, $D_{n+1}^{(2)}$, $A_N^{(2)}$ ($N=2n, 2n-1$). 

\bnum
\item  Type $B_{n}^{(1)}$ ($n \ge 2$)

\bna
\item   $d_{k,l}(z) =  \displaystyle \prod_{s=1}^{\min (k,l)} \big(z-(-q)^{|k-l|+2s}\big)\big(z+(-q)^{2n-k-l-1+2s}\big)$ for $1 \le k,l \le n-1$,
\item $d_{k,n}(z) = \displaystyle  \prod_{s=1}^{k}\big(z-(-1)^{n+k}q_s^{2n-2k-1+4s}\big)$ for $1 \le k \le n-1$,
\item $ d_{n,n}(z)=\displaystyle \prod_{s=1}^{n} \big(z-(q_s)^{4s-2}\big)$.

\ee

\item Type $C_{n}^{(1)}$ ( $n \ge  2  $ )  
\bna
\item  $ d_{k,l}(z)= \hspace{-2ex}\displaystyle \hspace{-3ex}\prod_{s=1}^{ \min(k,l,n-k,n-l)} \hspace{-5ex}
\big(z-(-q_s)^{|k-l|+2s}\big)\hspace{-2ex}\prod_{s=1}^{ \min(k,l)}\hspace{-1.5ex} \big(z-(-q_s)^{2n+2-k-l+2s}\big)$  for $1 \le k,l \le n$.
\ee

\item Type $A_{2n-1}^{(2)}$ ($n\ge 2$) 

\bna
\item  $ d_{k,l}(z)= \displaystyle \prod_{s=1}^{\min(k,l)} \big(z-(-q)^{|k-l|+2s}\big)\big(z+(-q)^{2n-k-l+2s}\big)$ for $1 \le k,l \le n$.
\ee

\item Type $A_{2n}^{(2)}$ ($n \ge 1$)

\bna
\item   $d_{k,l}(z) = \displaystyle \prod_{s=1}^{\min(k,l)} \big(z-(-q)^{|k-l|+2s}\big)\big(z-(-q)^{2n+1-k-l+2s}\big)$ for $1 \le k,l \le n$.
\ee

\item Type $D_{n+1}^{(2)}$ ($n \ge 3$)

\bna
\item    $d_{k,l}(z) = \displaystyle \prod_{s=1}^{\min(k,l)} \big(z^2 - \mqs^{|k-l|+2s}\big)\big(z^2 - \mqs^{2n-k-l+2s}\big)$ for  $1 \le k,l \le n-1$,
\item $d_{k,n}(z) = \displaystyle \prod_{s=1}^{k}\big(z^2+(-q^{2})^{n-k+2s}\big)$  for $1 \le k \le n-1$,
\item  $ d_{n,n}(z)=\displaystyle \prod_{s=1}^{n} \big(z+\mqs^{s}\big)$ for $k=l=n$.
\ee

\ee

\subsection{Other exceptional affine types } \

The denominator formulas for exceptional affine type can be found in \cite[Section 4 and Section 7]{OS19}.

\bnum

\item Type $G_2^{(1)}$

\bna
\item  $ d_{1,1}(z)  =   (z  - q_t^6) (z  -  q_t^{8}) (z  -  q_t^{10}) (z  -  q_t^{12}) $,
\item  $ d_{1,2}(z)  =   (z  + q_t^7) (z  +  q_t^{11})$,
\item $ d_{2,2}(z) = (z -q_t^2) (z -q_t^{8}) (z -q_t^{12}) $.

\ee

\item Type $F_4^{(1)}$

\bna
\item $ d_{1,1}(z) = (z-q_s^4)  (z-q_s^{10}) (z-q_s^{12}) (z-q_s^{18})  $,
\item $d_{1,2}(z) = (z+q_s^{6}) (z+q_s^{8}) (z+q_s^{10}) (z+q_s^{12}) (z+q_s^{14}) (z+q_s^{16})  $,
\item $d_{1,3}(z) = (z-q_s^{7}) (z-q_s^{9}) (z-q_s^{13}) (z-q_s^{15})  $,
\item $d_{1,4}(z) = (z+q_s^{8}) (z+q_s^{14})  $,
\item $d_{2,2}(z) = (z-q_s^{4}) (z-q_s^{6}) (z-q_s^{8})^2 (z-q_s^{10})^2 (z-q_s^{12})^2 (z-q_s^{14})^2 (z-q_s^{16}) (z-q_s^{18})    $,
\item $d_{2,3}(z) = (z+q_s^{5}) (z+q_s^{7}) (z+q_s^{9}) (z+q_s^{11})^2 (z+q_s^{13}) (z+q_s^{15}) (z+q_s^{17})    $,
\item $ d_{2,4}(z) = (z-q_s^{6}) (z-q_s^{10}) (z-q_s^{12}) (z-q_s^{16}) $,
\item $ d_{3,3}(z) = (z - q_s^{2}) (z - q_s^{6}) (z - q_s^{8}) (z - q_s^{10}) (z - q_s^{12})^2 (z - q_s^{16}) (z - q_s^{18})  $,
\item $ d_{3,4}(z) = (z+q_s^{3}) (z+q_s^{7}) (z+q_s^{11}) (z+q_s^{13}) (z+q_s^{17})  $,
\item $ d_{4,4}(z) = (z - q_s^{2}) (z - q_s^{8}) (z - q_s^{12}) (z - q_s^{18})  $.

\ee

\item Type $ D_4^{(3)}$

\bna
\item  $ d_{1,1}(z) = (z-q^2) (z-q^6) (z-\omega q^4) (z - \omega^2 q^4)$,
\item $ d_{1,2}(z) =  (z^3 + q^9) (z^3 +  q^{15})$,
\item $ d_{2,2}(z) = (z^3-q^6) (z^3-q^{12})^{2} (z^3-q^{18}) $.
\ee

\item Type $E_6^{(2)}$

\bna
\item $d_{1,1}(z) = (z-q^{2})  (z+q^{6}) (z-q^{8}) (z+q^{12})  $,
\item $d_{1,2} (z) = ( z+q^{3}) ( z-q^{5}) ( z-q^{7}) ( z+q^{7}) ( z+q^{9}) ( z-q^{11})   $,
\item $d_{1,3}(z) = (z^2 + q^{8}) (z^2 + q^{12}) (z^2 + q^{16}) (z^2 + q^{20}) $, 
\item $d_{1,4}(z) = (z^2 + q^{10}) (z^2 + q^{18}) $,
\item $ d_{2,2}(z) = (z-q^{2}) (z-q^{4}) (z-q^{6}) (z-q^{8})^2 (z-q^{10}) (z+q^{4}) (z+q^{6})^2 (z+q^{8}) $ 

\qquad \qquad $(z+q^{10}) (z+q^{12})    $,

\item $d_{2,3}(z) = (z^2 + q^{6}) (z^2 + q^{10})^2 (z^2 + q^{14})^2 (z^2 + q^{18})^2 (z^2 + q^{22}) $,
\item $d_{2,4}(z) = (z^2 + q^8) (z^2 + q^{12}) (z^2 + q^{16}) (z^2 + q^{20})  $,
\item $d_{3,3}(z) = (z^2 - q^{4}) (z^2 - q^{8})^2 (z^2 - q^{12})^3 (z^2 - q^{16})^3 (z^2 - q^{20})^2 (z^2 - q^{24}) $,
\item $d_{3,4}(z) = ( z^2 - q^{6}) ( z^2 - q^{10}) ( z^2 - q^{14})^2 ( z^2 - q^{18})^{2} ( z^2 - q^{22})  $,

\item $d_{4,4}(z) = (z^2 - q^4)  (z^2 - q^{12}) (z^2 - q^{16}) (z^2 - q^{24})  $.

\ee

\ee

\end{document}